\documentclass{mcom-l}
\usepackage{amsmath}
\usepackage{amsfonts}
\usepackage{euscript}

\begin{document}
\mathchardef\Gamma="7100
\mathchardef\Sigma="7106
\mathchardef\Omega="710A
\mathchardef\Upsilon="7107
\newcommand\esA{{\EuScript A}}
\newcommand\esB{{\EuScript B}}
\newcommand\esC{{\EuScript C}}
\newcommand\esD{{\EuScript D}}
\newcommand\esE{{\EuScript E}}
\newcommand\esF{{\EuScript F}}
\newcommand\esG{{\EuScript G}}
\newcommand\esH{{\EuScript H}}
\newcommand\esI{{\EuScript I}}
\newcommand\esJ{{\EuScript J}}
\newcommand\esK{{\EuScript K}}
\newcommand\esL{{\EuScript L}}
\newcommand\esM{{\EuScript M}}
\newcommand\esN{{\EuScript N}}
\newcommand\esO{{\EuScript O}}
\newcommand\esP{{\EuScript P}}
\newcommand\esQ{{\EuScript Q}}
\newcommand\esR{{\EuScript R}}
\newcommand\esS{{\EuScript S}}
\newcommand\esT{{\EuScript T}}
\newcommand\esU{{\EuScript U}}
\newcommand\esV{{\EuScript V}}
\newcommand\esW{{\EuScript W}}
\newcommand\esX{{\EuScript X}}
\newcommand\esY{{\EuScript Y}}
\newcommand\esZ{{\EuScript Z}}
\newcommand\reals{{\mathbb R}}
\newcommand\Reals{{\mathbb R}}
\newcommand\Rt{ {\Reals^3} }
\newcommand\Rn{ {\Reals^n} }
\newcommand\Rk{ {\Reals^k} }
\newcommand\Rkk{ {\Reals^{k\times k}} }
\newcommand\veps{ {{\epsilon}} }
\newcommand\tstyle{ \textstyle }
\newcommand\hxi{ \smash{\widehat\xi} }
\newcommand\diam{\mathop{\rm diam}\nolimits}
\newcommand\dist{\mathop{\rm dist}\nolimits}
\newcommand\supp{\mathop{\rm supp}\nolimits}
\newtheorem{theorem}{Theorem}[section]
\newtheorem{lemma}[theorem]{Lemma}
%


\title
[A convergence theorem for a class of Nystr{\"o}m methods]
{A convergence theorem for a class of Nystr{\"o}m methods 
for weakly singular integral equations on surfaces in $\Rt$}

\author[O. Gonzalez]{Oscar Gonzalez}
\address{Department of Mathematics, 
The University of Texas at Austin,
Austin, TX 78712}
\curraddr{}
\email{og@math.utexas.edu}
\thanks{This work was supported by 
the National Science Foundation.}

\author[J. Li]{Jun Li}
\address{Graduate Program in Computational and Applied 
Mathematics, The University of \break \indent Texas at 
Austin, Austin, TX 78712}
\curraddr{Schlumberger Corporation, Houston, TX}
\email{JLi49@slb.com}
\thanks{}

\date{\today}

\begin{abstract}
A convergence theorem is proved for a class of Nystr{\"o}m
methods for weakly singular integral equations on surfaces
in three dimensions.  Fredholm equations of the second kind 
as arise in connection with linear elliptic boundary value 
problems for scalar and vector fields are considered.  In 
contrast to methods based on product integration, coordinate 
transformations and singularity subtraction, the family 
of Nystr{\"o}m methods studied here is based on a 
local polynomial correction determined by an auxiliary 
system of moment equations.  The polynomial
correction is shown to remove the weak singularity 
in the integral equation and provide control over
the approximation error.  Convergence results for the 
family of methods are established under minimal regularity 
assumptions consistent with classic potential theory.  
Rates of convergence are shown to depend on the regularity 
of the problem, the degree of the polynomial correction, 
and the order of the quadrature rule employed in the 
discretization.  As a corollary, a simple method based 
on singularity subtraction which has been employed by 
many authors is shown to be convergent.
\end{abstract}

\maketitle


\section{Introduction}

The method of integral equations has a long and rich
history in both the analysis and numerical treatment
of boundary value problems 
\cite{Atkinson:1997,Brebbia:1984,Chen:1992,Goldberg:1997,
Gunter:1967,Hackbush:1995,Hsaio:2008,Kellogg:1929,
Kress:1989,Rjasanow:2007}.
On the analysis side, the method provides a classic 
approach to the study of existence and uniqueness 
questions.  On the numerical side, the method provides 
an alternative formulation which can be efficiently 
discretized.  In the basic approach, a boundary value 
problem described by a partial differential equation 
on a domain of interest is reduced to an integral
equation on the bounding surface.  The unknown 
scalar or vector field throughout the domain 
is represented in terms of one or more integral 
potentials which depend on an unknown surface 
density.  The representation can take various 
different direct and indirect forms and the 
integral potentials usually involve kernels 
that are at least weakly singular.  The reduction 
in dimension makes the method of integral equations 
extremely attractive, especially for problems on 
exterior domains.  Indeed, the method has been 
used in many different modern applications in
acoustics and electromagnetics 
\cite{Colton:1998,Nedelec:2001},
hydrodynamics 
\cite{Power:1995,Pozrikidis:1992}, 
elastostatics 
\cite{Jaswon:1977, Linkov:2002},
and molecular modeling 
\cite{Allison:1998,Aragon:2006,Bardhan:2009,
Brune:1993,Gonzalez:2008}.

Once a problem has been reduced to an integral
equation, various different types of methods are
available for its numerical resolution 
\cite{Atkinson:1997,Goldberg:1997,Hackbush:1995,
Kress:1989,Rjasanow:2007}.  In Galerkin methods, 
the unknown density and integral equation are 
projected onto a finite-dimensional subspace 
of functions on the surface, where the projection 
is defined through an inner-product on the space.  
Basis functions for the space may be piecewise 
polynomials, splines or wavelets with local supports
\cite{Alpert:1993,Atkinson:1997,Goldberg:1997,Lage:1999}, 
or related eigenfunctions with global supports
\cite{Atkinson:1982,Ganesh:2004,Ganesh:1998,Graham:2002}.  
A method is completed through 
the specification of appropriate quadrature formulae 
for the evaluation of double surface integrals in the 
inner-product.  In collocation methods, the unknown 
density and integral equation are projected similar to 
before, but the projection is defined through pointwise 
interpolation on the surface.  As a consequence, only 
single surface integrals arise, which is computationally
attractive.  In Nystr{\"o}m methods, the integrals in 
the integral equation are directly approximated by a 
quadrature rule, which leads to a direct approximation
of the unknown density at the quadrature nodes, which
can then be interpolated over the surface.  For problems 
with continuous kernels, the Nystr{\"o}m approach leads 
to simple, efficient and well-understood schemes, and 
for problems with weakly singular kernels, various 
modifications are required to handle the singularities, 
such as product integration, coordinate transformations
and singularity subtraction
\cite{Anselone:1981,Atkinson:1997,Bruno:2001,
Canino:1998,Goldberg:1997,Kress:1989,Li:submitted,
Rathsfeld:2000,Ying:2006}.

In this article, we study a class of modified 
Nystr{\"o}m methods based on local polynomial
corrections.  The methods studied here share 
some similarities with product integration 
and singularity subtraction methods, but have 
various important differences.  The basic approach 
is to employ standard quadrature weights away from 
the diagonal of the weakly singular kernel, and 
generalized or corrected weights near the diagonal.  
In product integration methods, the corrected
weights are defined so as to include the action
of the kernel on a mesh-dependent basis of 
interpolating polynomials.  In the methods 
considered here, the corrected weights include 
the action of the kernel on a mesh-independent 
basis of local floating polynomials, and furthermore 
the weights themselves are expressed as the pointwise 
values of such a polynomial. The local polynomial that 
generates the corrected weights is determined through 
an auxiliary system of moment equations.  Such a local 
polynomial correction is shown to implicitly regularize 
the weak singularity in the kernel, as would occur in 
a singularity subtraction method, and provide control 
over the approximation error.

We establish a convergence theorem for a family of 
modified Nystr{\"o}m methods for integral equations on 
surfaces in three dimensions.  We restrict attention 
to Fredholm equations of the second kind with weakly 
singular kernels as arise in connection with linear 
elliptic boundary value problems for scalar and vector 
fields.  In contrast to methods based on product integration, 
coordinate transformations and singularity subtraction, 
we consider methods based on a local polynomial correction 
as described above.  Convergence results in the standard 
maximum norm are established under minimal regularity 
assumptions consistent with classic potential theory.
Rates of convergence are shown to depend on the 
regularity of the surface and data, the degree of 
the local polynomial correction, and the order of 
the underlying quadrature rule.  In the minimal 
regularity case, there is no lower bound on the rate, 
and in the smooth case, there is no upper bound on 
the rate.  Indeed, arbitrarily high rates of 
convergence can be achieved with local polynomial 
corrections of arbitrarily high degree.  Various 
important assumptions on the weakly singular kernel 
and the quadrature rule are discussed in detail.  
As a corollary of our result, we show that a simple 
method based on singularity subtraction which has 
been employed by many authors is convergent under 
the same assumptions.

The mathematical theory of convergence of Nystr{\"o}m
methods of integral equations of the second kind on 
regular surfaces is a well-studied subject.  The 
convergence theory for standard methods on
problems with continuous kernels is classic
\cite{Atkinson:1997,Goldberg:1997,Kress:1989}.  Similarly,
the theory for product integration methods on problems
with weakly singular kernels is also well-established
\cite{Atkinson:1997,Goldberg:1997,Kress:1989}, although 
some of the hypotheses may be difficult to verify in
three-dimensional problems. Convergence results
for methods based on floating polar coordinate
transformations are described in
\cite{Bruno:2001,Kunyansky:submitted,Ying:2006}.
Convergence results for methods based on the classic
idea of singularity subtraction as introduced by
Kantorovich and Krylov \cite{Krylov:1958} have been 
established for one-dimensional problems in 
\cite{Anselone:1981}, and have been studied for 
higher-dimensional problems in \cite{Rathsfeld:2000}.  
Methods similar to those introduced here were 
considered in \cite{Canino:1998}, but no 
mathematical convergence results were given.  
Numerical experiments which illustrate several
aspects of the lowest-order method of the family 
introduced here can be found in \cite{Li:submitted}.
Related convergence results under different regularity
assumptions are given in \cite{Li:2010}.

The presentation is structured as follows.  In Section 2
we define the class of boundary integral equations to
be studied.  We outline various important assumptions on 
the surface and kernels associated with the equation, 
and introduce notation and results that will be needed 
throughout our developments. In Section 3 we define 
our family of Nystr{\"o}m methods with local polynomial 
corrections as described above. We outline essential 
assumptions on the quadrature rule and other elements 
of the method and then state our main result.  In 
Section 4 we provide a proof of our result.  The
proof is based on the theory of collectively compact 
operators and several technical lemmas; the latter
are established herein and may be of independent 
interest.

\section{Boundary integral equation}
\label{IntEqnSect}

\subsection{Preliminaries, notation}

Throughout our developments, we consider a surface $\Gamma$ 
in $\Reals^3$ consisting of a finite number of disjoint,
closed, bounded and orientable components.  We invoke 
the standard assumptions of classic potential theory 
\cite{Gunter:1967,Kellogg:1929} and assume 
that $\Gamma$ is a Lyapunov surface.  Thus:
\begin{itemize}
\item[(L1)]
there exists a well-defined outward unit normal $\nu(x)$ 
and tangent plane $T_x\Gamma$ at every $x\in\Gamma$,
\item[(L2)]
there exists constants $C>0$ and $0<\lambda\le 1$ such that 
$\theta(\nu(x),\nu(y))\le C|x-y|^\lambda$ for all
$x,y\in\Gamma$, where $\theta(\nu(x),\nu(y))$ is the angle
between $\nu(x)$ and $\nu(y)$, and $|x-y|$ denotes the
Euclidean distance,
\item[(L3)]
there exists a constant $d>0$ such that, for every 
$x\in\Gamma$, the subset $\Gamma\cap B(x,d)$ is a graph 
over $T_x\Gamma$, where $B(x,d)$ denotes the closed ball 
of radius $d$ centered at $x$.
\end{itemize}
We refer to $\lambda$ and $d$ as a Lyapunov exponent and 
radius associated with $\Gamma$, and for any $x\in\Gamma$, 
refer to $B(x,d)$ as a Lyapunov ball at $x$.  Notice that, 
if (L2) and (L3) hold for some values of $\lambda$ and 
$d$, then they also hold for all smaller values.  We 
assume that values for $\lambda$ and $d$ are fixed once
and for all, and for simplicity we assume from the outset
that $\lambda=1$.

For any $x_0\in\Gamma$, we use $\Gamma_{x_0,d}$ to denote 
the portion of $\Gamma$ within the Lyapunov ball at $x_0$, 
and use $\Omega_{x_0,d}$ to denote the image of $\Gamma_{x_0,d}$ 
on $T_{x_0}\Gamma$ under projection parallel to $\nu(x_0)$.  
We refer to $\Gamma_{x_0,d}$ as the Lyapunov patch at 
$x_0$.  Without loss of generality, we identify $\Omega_{x_0,d}$ 
with a subset of $\Reals^2$, and identify $x_0$ with the origin.  
We reserve the notation $T_{x_0}\Gamma$ to indicate the tangent
plane considered as a subspace of $\Rt$.  The Lyapunov 
condition (L3) implies that the map 
\begin{equation}
   y=\psi_{x_0}(\xi),
\quad
y\in\Gamma_{x_0,d},
\quad
\xi\in \Omega_{x_0,d},
\label{LocCarCoords}
\end{equation}
defined by projection parallel to $\nu(x_0)$, is a 
bijection.  We refer to $y=\psi_{x_0}(\xi)$ with 
inverse $\xi=\psi^{-1}_{x_0}(y)$ as a local 
Cartesian parameterization of $\Gamma$ at $x_0$.  
When there is no cause for confusion, we abbreviate
$\psi_{x_0}(\xi)$ by $y_{x_0}(\xi)$, and abbreviate
$\psi^{-1}_{x_0}(y)$ by $\xi_{x_0}(y)$.  For any
$x_0\in\Gamma$ and $y\in\Gamma_{x_0,d}$, the local 
Cartesian coordinates $\xi_{x_0}(y)$ are uniquely
defined up to the choice of orthonormal basis 
$\{t_{x_0,1},t_{x_0,2}\}$ in $T_{x_0}\Gamma$.  
Specifically, we have
\begin{equation}
\xi_{\alpha} = \psi^{-1}_{x_0,\alpha}(y) 
     = (y-x_0)\cdot t_{x_0,\alpha},
\quad
\alpha=1,2.
\end{equation}

Throughout our analysis, we will also have need to consider 
a local polar parameterization of $\Gamma$ at $x_0$ of the 
form
\begin{equation}
y=\psi_{x_0}^{\rm polar} (\rho,\hxi)
 =\psi_{x_0}(\varpi(\rho,\hxi)),
\quad
y\in\Gamma_{x_0,d},
\quad
(\rho,\hxi)\in\Omega_{x_0,d}^{\rm polar},
\label{LocPolCoords}
\end{equation}
where $\smash{\Omega_{x_0,d}^{\rm polar}}$ is a subset of
$\Reals_+\times S$.  Here $\Reals_+$ denotes the 
non-negative reals and $S$ denotes the unit circle.  
Notice that the map $\xi=\varpi(\rho,\hxi)=\rho\hxi$ 
is one-to-one for all $\xi\ne 0$, with inverse 
$(\rho,\hxi)=\varpi^{-1}(\xi)=(|\xi|,\xi/|\xi|)$, 
and is many-to-one at $\xi=0$, with  
$0=\varpi(0,\hxi)$ for all $\hxi\in S$.  For
any $x_0\in\Gamma$ and $y\in\Gamma_{x_0,d}$ with
$y\ne x_0$ the map in (\ref{LocPolCoords}) is
invertible.  Specifically, for any choice of orthonormal 
basis $\{t_{x_0,1},t_{x_0,2}\}$ in $T_{x_0}\Gamma$, we 
have
\begin{equation}
(\rho,\hxi) = [\psi_{x_0}^{\rm polar}]^{-1}(y) 
\quad\hbox{\rm where}\quad
\rho=|y-x_0|,
\;
\hxi_{\alpha} = \frac{(y-x_0)_{\alpha}}{|y-x_0|},
\;
\alpha=1,2.
\end{equation}

We use $C^m(U,V)$ to denote the space of $m$-times
continuously differentiable functions from 
$U\subset\Rn$ into $V\subset\Rk$,
and $C^{m,1}(U,V)$ to denote the subspace of 
functions whose derivatives up to order $m$ are 
Lipschitz.  When $U$ is not an open set or the 
closure of an open set in $\Rn$, we interpret 
$C^m(U,V)$ to mean functions possessing an 
$m$-times continuously differentiable extension
to an open set or the closure of an open set which 
contains $U$.  For instance, when 
$U$ is a Lyapunov patch on $\Gamma$, we 
may consider extensions that are constant 
in the normal direction.  We say that $\Gamma$ 
is of class $C^m$, and use the notation 
$\Gamma\in C^m$, if 
$\psi_{x_0}\in C^m(\Omega_{x_0,d},\Reals^3)$ 
and 
$\psi_{x_0}^{-1}\in C^m(\Gamma_{x_0,d},\Reals^2)$ 
for every $x_0\in\Gamma$.  Similarly, we say that 
$\Gamma$ is of class $C^{m,1}$, and use the 
notation $\Gamma\in C^{m,1}$, if 
$\psi_{x_0}\in C^{m,1}(\Omega_{x_0,d},\Reals^3)$ 
and
$\psi_{x_0}^{-1}\in C^{m,1}(\Gamma_{x_0,d},\Reals^2)$ 
for every $x_0\in\Gamma$, and additionally the Lipschitz
constant for each derivative is uniform in $x_0$.

To any $f\in C^{1}(\Gamma,\Rk)$ we associate a surface
derivative $Df\in C^{0}(\Gamma,\Reals^{k\times 3})$.  
Specifically, for any Lyapunov patch $\Gamma_{x_0,d}$ 
and any point $y=\psi_{x_0}(\xi)\in \Gamma_{x_0,d}$, 
we define $Df(y)$ by
\begin{alignat}{2}
Df(y)t 
 &= \frac{\partial(f\circ\psi_{x_0})}{\partial\xi}\mu,
 &\quad &\forall t = \frac{\partial\psi_{x_0}}
                         {\partial\xi}\mu\in T_y\Gamma,
  \quad \mu\in\Reals^2, \\
Df(y)\nu 
 &= 0,
 &\quad &\forall \nu\in [T_y\Gamma]^\perp.
\end{alignat}
Since any vector $v$ can be decomposed as $v=v_t+v_\nu$, 
where $v_t\in\ T_y\Gamma$ and $v_\nu\in [T_y\Gamma]^\perp$, 
we have $Df(y)v = Df(y)v_t$.  Thus $Df(y)\in \Reals^{k\times 3}$ 
is defined for all vectors $v\in\Rt$.  For arbitrary 
$f\in C^{1}(\Gamma,\Rk)$, the matrix $Df(y)$ is equivalent
to the usual derivative matrix of an extension of $f$, where 
the extension is constant in the normal direction to $\Gamma$.  
When $f\in C^{1}(\Gamma,\Rk)$ is the restriction of some 
given $g\in C^{1}(\Rt,\Rk)$, the matrix $Df(y)$ is equivalent 
to the composition of the usual derivative matrix of $g$ with 
the projection matrix for orthogonal projection onto 
$T_y\Gamma$.  In the case when $g$ is constant in the
normal direction, the projection matrix can be replaced 
by the identity matrix and the two descriptions coincide. 
Notice that, for any curve $y(\tau)=\psi_{x_0}(\xi(\tau))$ 
in any Lyapunov patch $\Gamma_{x_0,d}$, we have
\begin{equation}
\frac{d}{d\tau} f(y(\tau)) 
  = Df(y(\tau)) \frac{d}{d\tau}y(\tau).
\end{equation}

Given a function of the form $f(x,y)$ with $x,y\in\Gamma$,
we use the notation $f_x(y)$ and $f_y(x)$ to denote 
the functions obtained by fixing the values of $x$
and $y$, respectively.  Similarly, we use the notation 
$D_x f(x,y)$ and $D_y f(x,y)$ to denote the surface 
derivatives of $f(x,y)$ with respect to $x$ and $y$, 
respectively, where the other variable is held fixed.
Hence $D_x f(x,y) = D f_y(x)$ and $D_y f(x,y) = D f_x(y)$.   
Higher-order surface derivatives for a function $f(y)$ are 
defined in a natural way.  For instance, if we identify
$\Reals^{k\times 3}$ with $\Reals^{3k}$ and
$Df\in C^{1}(\Gamma,\Reals^{3k})$, then 
$D^2 f= D(Df)\in C^{0}(\Gamma,\Reals^{3k\times 3})$
is defined exactly as above.  We denote the diagonal 
subset of $\Gamma\times\Gamma$ by 
$\Upsilon=\{(x,y)\in\Gamma\times\Gamma\;|\;y=x\}$.
Moreover, for any $\delta>0$ we also define an open 
$\delta$-neighborhood of the diagonal by
$\Upsilon_\delta=\{(x,y)\in\Gamma\times\Gamma\;|\;
|y-x|<\delta\}$.  Throughout our developments, we
use $|\cdot|$ to denote a Euclidean norm or the 
measure of a surface, as determined by the context.

\subsection{Problem statement}

Given a Lyapunov surface $\Gamma$ in 
three-dimensional space $\Rt$, we consider 
the problem of finding a function 
$\varphi\in C^0(\Gamma,\Rk)$ that satisfies
\begin{equation}
c \varphi - \esA\varphi = f,
\label{IntEqn}
\end{equation}
where $c \neq 0$ is a given constant,
$f \in C^0(\Gamma,\Rk)$ is a given function,
and $\esA:C^0(\Gamma,\Rk) \to C^0(\Gamma,\Rk)$
is a given integral operator of the form
\begin{equation}
\esA\varphi = \esG\varphi + \esH\varphi,
\label{DefAop}
\end{equation}
where
\begin{equation}
(\esG\varphi)(x)
= \int_\Gamma G(x,y) \varphi(y) dA_y,
\quad
(\esH\varphi)(x)
=  \int_\Gamma H(x,y) \varphi(y) dA_y.
\label{DefGHops}
\end{equation}
In the above, $dA_y$ denotes an infinitesimal area 
element at $y$, $G\in C^0(\Gamma\times\Gamma,\Rkk)$
is a continuous kernel, and
$H\in C^0(\Gamma\times\Gamma\backslash\Upsilon,\Rkk)$
is a weakly singular kernel, which is unbounded along
the diagonal subset $\Upsilon$.  The above equation 
is typically considered with $k=1$ or $3$.  The 
first case corresponds to a scalar-valued problem 
in three-dimensional space, as would arise in 
applications to electrostatics, and the second 
case corresponds to a vector-valued problem in
three-dimensional space, as would arise in 
applications to elastostatics and hydrodynamics.

Our study of (\ref{IntEqn}) will rely on various 
assumptions about the surface $\Gamma$ and the 
kernels $G$ and $H$.  Specifically, we assume:
\begin{itemize}
\item[(A0)] $\Gamma\in C^{m+1,1}$ for some $m\ge0$, 
with Lyapunov radius $d>0$. 
\item[(A1)] $G\in C^{m,1}(\Gamma\times\Gamma,\Rkk)$.
\item[(A2)] $H\in C^{m,1}
(\Gamma\times\Gamma\backslash\Upsilon_\delta,\Rkk)$ for
every $0<\delta\le d$. Moreover, $H$ can be decomposed 
as
\begin{equation}
H(x,y)=u(x,y)/|x-y|^{2-\mu}, 
\label{Hker}
\end{equation}
for some function 
$u\in C^0(\Gamma\times\Gamma\backslash\Upsilon,\Rkk)$ 
and exponent $0 < \mu \le 1$.  Furthermore, on
$\Gamma\times\Gamma\backslash\Upsilon$ the 
component functions $u_{ij}$ ($1\le i,j\le k$) satisfy, 
for some positive constant $C$,
\begin{alignat}{2}
&|u_{ij}(x,y)| \le C, 
  &\quad & \label{HkerUbnd}\\ 
&|D_x^s u_{ij}(x,y)| \le C|y-x|^{-s}, 
  &\quad &1\le s\le m+1, \label{HkerUx} \\
&|D_y^s u_{ij}(x,y)| \le C|y-x|^{-s}, 
  &\quad  &1\le s\le m.  \label{HkerUy}
\end{alignat}
\item[(A3)] When $m\ge 1$ we additionally assume
that $\mu=1$ and that $u$ has the properties
\begin{alignat*}{3}
&{\rm (i)} &\quad 
&u_{x_0}^{\rm polar} 
\in C^{0,1}(\Omega_{x_0,d}^{\rm polar},\Rkk),
  &\quad &\forall x_0\in\Gamma, \\ 
&{\rm (ii)} &\quad 
&u_{x_0}^{\rm polar}(0,-\hxi) 
= u_{x_0}^{\rm polar}(0,\hxi),
  &\quad &\forall \hxi\in S, \\
&{\rm (iii)} &\quad 
&u_{x_0}^{\Delta} \in C^{0,1}(\Gamma_{x_0,d},\Rkk),
  &\quad  &\forall x_0\in\Gamma,  
\end{alignat*}
where Lipschitz constants are uniform in $x_0$ and 
$\smash{u_{x_0}^{\rm polar}} = 
\smash{u_{x_0}\circ\psi_{x_0}^{\rm polar}}$ 
and
\begin{equation}
u_{x_0}^{\Delta}(y) = 
\begin{cases}
[u_{x_0}^{\rm polar}(\rho,\hxi)-u_{x_0}^{\rm polar}(0,\hxi)]
\lvert_{(\rho,\hxi)=[\psi_{x_0}^{\rm polar}]^{-1}(y)}, 
&y\ne x_0, \\
0,
&y=x_0. \\
\end{cases}
\end{equation}
\end{itemize}

Assumption (A0) states that $\Gamma$ must be at
least of class $C^{1,1}$, which implies that
$\Gamma$ is differentiable, with a curvature 
that is defined almost everywhere and bounded.  
Assumption (A1) states that the regularity of $G$
may be one order less than that of $\Gamma$, which
will be convenient for our analysis.  Assumption 
(A2) is essentially the definition of a 
weakly singular kernel.  Typically, the parameter 
$\mu$ is identified with the Lyapunov exponent 
$\lambda$, but it is not necessary to do so.  Away 
from the diagonal, the regularity of $H$ is assumed 
to be the same as that of $G$.  Assumption (A3) 
outlines additional regularity conditions on $H$ in 
the case when $\Gamma$ is at least of class $C^{2,1}$.  
The most specific condition is (A3)(ii), which states
that the function $\smash{u_{x_0}^{\rm polar}}$ is an 
even function on $S$ at $\rho=0$.  As we will see, this 
implies that certain local moments of $H$ will vanish.  
This condition will be important in our analysis of 
numerical methods and is related to the classic 
Tricomi condition that arises in the study of singular 
integral operators \cite{Tricomi:1928}.  This condition 
is satisfied by weakly singular kernels $H$ arising
in different applications, for example, the
classic double-layer kernel in three-dimensional 
electrostatics, and incompressible elastostatics and 
hydrodynamics.  We remark that condition (A3) is 
required for only the lowest-order method studied
herein corresponding to a local polynomial correction
of degree zero; it is not required for the higher-order 
methods.

\subsection{Solvability theorem}

The following result establishes the solvability of the
boundary integral equation in (\ref{IntEqn}).  Its proof 
follows from the classic Fredholm Theorems for compact 
operators \cite{Kress:1989,Mikhlin:1960} and is omitted 
for brevity.  The compactness of $\esA=\esG+\esH$ will 
be established in Lemma \ref{LemmaFour}.

\begin{theorem} \label{IntEqnThm}
Under conditions {\rm (A0)}--{\rm (A2)} the operator 
$\esA$ on $C^0(\Gamma,\Rk)$ is compact.  Hence, provided 
$c$ is not an associated eigenvalue, there exists a unique 
solution $\varphi\in C^0(\Gamma,\Rk)$ for any closed, 
bounded Lyapunov surface $\Gamma\in C^{1,1}$ and
boundary data $f\in C^0(\Gamma,\Rk)$.
\end{theorem}

Thus, under mild conditions, the boundary integral
equation in (\ref{IntEqn}) has a unique solution
$\varphi$ for any data $f$.  Moreover, because
the operator $cI-\esA$ can be shown to have a bounded 
inverse, it follows that $\varphi$ depends continuously 
on $f$.  The regularity of the solution $\varphi$ is 
determined by that of the data $f$ and properties of 
the operator $\esA$.  When $\esA$ is a smoothing operator, 
so that $\esA\varphi$ is smoother than $\varphi$, it follows
from (\ref{IntEqn}) that $\varphi$ is in the same regularity 
class as $f$.  The form of the operator $\esA$ arises in 
various applications as mentioned above.  The component
$\esH$ may be interpreted as a double-layer potential.
The component $\esG$ may be zero, or may be interpreted
as a range completion term necessary for the uniqueness
of solutions.  In the latter case, $\esG$ may represent
a potential due to a discrete or continuous distribution
of sources located away from the surface.  Such operators
can be motivated beginning from a generalized boundary
integral representation of a field by a linear combination
of single- and double-layer potentials, and then moving 
the single-layer potential to an offset surface, and
then possibly also shrinking the offset surface to a 
curve or point, or collection thereof.  Various forms 
of such operators have been previously considered 
\cite{Gonzalez:2009,Gunter:1967,Hebeker:1985,Kim:1991,
Ladyzhenskaya:1963,Mikhlin:1960,Odqvist:1930,Power:1987}.

\section{Nystr{\"o}m approximation}
\label{NysAppSect}

\subsection{Mesh, quadrature rule}

We consider an arbitrary decomposition of $\Gamma$
into non-overlapping quadrature elements $\Gamma^e$,
$e=1,\ldots,E$, each with area $|\Gamma^e|>0$.
To any such decomposition we associate a size
parameter $h = \max_{e}(\diam(\Gamma^e))>0$.
For simplicity, we assume that the elements
are either all quadrilateral or all triangular.
In our analysis, we consider sequences of 
decompositions with increasing $E$, or equivalently, 
decreasing $h$.  We will only consider sequences 
that satisfy a uniform refinement condition in 
the sense that the area of all elements is 
reduced at the same rate.  Specifically, we
assume 
\begin{equation*}
Ch^2 \le |\Gamma^e| \le C'h^2,
\quad\forall e=1,\ldots,E, \quad E\ge E_0.
\eqno({\rm A4})
\end{equation*}
Here $C$, $C'$ and $E_0$ are positive constants
whose values may change from one appearance to
the next.

In each element $\Gamma^e$, we introduce quadrature 
nodes $x^e_q$ and weights $W^e_q>0$, $q = 1,\dots,Q$, 
such that
\begin{equation}
\int_\Gamma f(x)\; dA_x =
\sum_{e=1}^{E} \int_{\Gamma^e} f(x) \;dA_x 
      \approx \sum_{e=1}^{E}\sum_{q=1}^{Q} f(x^e_q) W^e_q.
\label{QuadDefn}
\end{equation}
Without loss of generality, we assume that the
quadrature nodes and weights are defined by
mapping each element $\Gamma^e$ to a standard,
planar domain using a local parameterization 
of the same regularity as the surface, and applying 
a local quadrature rule in the standard domain.  In 
this case, the Jacobian of the parameterization would 
be included in the weights $W^e_q$.  We assume that 
the quadrature weights remain bounded and that the 
quadrature points remain distinct and in the element 
interiors.  Specifically, for any sequence of 
decompositions that satisfy the uniform refinement 
condition, we assume
\begin{equation*}
\begin{gathered}
Ch^2 \le {\tstyle\sum_{q=1}^Q} W^e_q \le C'h^2,
\quad\forall 
e=1,\ldots,E, 
\quad 
E\ge E_0, \cr
Ch \le \dist(x^e_q,\partial\Gamma^e)\le h,
\quad\forall 
q=1,\ldots,Q, 
\quad 
e=1,\ldots,E, 
\quad 
E\ge E_0, \cr
Ch \le \min_{(e,q)\ne(e',q')}|x^e_q-x^{e'}_{q'}|\le h,
\quad\forall 
q,q'=1,\ldots,Q, 
\quad 
e,e'=1,\ldots,E, 
\quad 
E\ge E_0.
\end{gathered}
\eqno({\rm A5})
\end{equation*}

To quantify the error in a quadrature rule for a
given function $f$ on a given surface $\Gamma$, we 
introduce the normalized local truncation errors
\begin{equation}
\tau(e,f,h) = \frac{1}{|\Gamma^e|}
\left|\int_{\Gamma^e} f(x) \;dA_x 
      - \sum_{q=1}^{Q} f(x^e_q) W^e_q \right|,
\quad e=1,\ldots,E.
\end{equation}
For sequences of decompositions that satisfy a uniform 
refinement condition, we require that the above 
truncation errors vanish uniformly in $e$ depending 
on properties of $f$.  Specifically, we assume there
exists an integer $\ell\ge 1$ such that
\begin{equation*}
\begin{gathered}
\hbox{\rm $\tau(e,f,h)\to 0$ as $h\to 0$ uniformly
in $e$, for each $f\in C^0(\Gamma,\Reals)$,} \cr
\hbox{\rm $\tau(e,f,h)\le Ch^\ell$ uniformly in
$e$, for each $f\in C^{\ell-1,1}(\Gamma,\Reals)$.} 
\end{gathered}
\eqno({\rm A6})
\end{equation*}
In the above, the constant $C$ may depend on $f$, 
but is independent of $e$ and $h$, and the integer 
$\ell\ge 1$ is called the order of convergence of 
the quadrature rule.  

For convenience, we will often replace the element 
and node indices $e=1,\dots,E$ and $q=1,\ldots,Q$ 
with a single, general index $a=1,\ldots,n$, where 
$n=EQ$.  We will use the multi- and single-index 
notation interchangeably with the understanding 
that there is a bijective map between the two.

\subsection{Partition of unity functions}

To each quadrature node $x_a$ in a decomposition
of $\Gamma$ we associate nodal partition of unity 
functions 
$\zeta_a,\smash{\widehat\zeta_a}\in C^0(\Gamma,\Reals)$.  
We assume that these functions take values in the 
unit interval and are complementary so that their
sum is equal to one.  Moreover, we assume that $\zeta_a$ 
vanishes at least quadratically in a neighborhood of
$x_a$, and that the support of $\smash{\widehat\zeta_a}$ 
is bounded from above by a multiple of the mesh
parameter $h$.  Specifically, we assume
\begin{equation*}
\begin{gathered}
0\le \zeta_a(x), \widehat\zeta_a(x)\le 1,
\quad
\zeta_a(x)+\widehat\zeta_a(x)=1,
\quad\forall x\in\Gamma, \\
|\zeta_a(x)| \le \frac{C|x-x_a|^2}{h^2},
\quad
\diam(\supp(\widehat\zeta_a)) \le Ch,
\quad\forall 
a=1,\ldots,n, 
\quad 
n\ge n_0. 
\end{gathered}
\eqno({\rm A7})
\end{equation*}
The functions $\zeta_a$ and $\smash{\widehat\zeta_a}$ will
play an important role in the family of numerical 
methods that we introduce and in the associated 
convergence proof.  Specifically, these functions 
will help isolate the weak singularity of the 
kernel $H$ at each quadrature point and control 
the numerical error there.  

\subsection{Discretization of integral equation}

Let a decomposition, quadrature rule, and nodal 
partition of unity functions for $\Gamma$ be given.
For any function $f$, let $\esG_h$ and $\esH_h$ 
denote approximations to $\esG$ and $\esH$ of 
the form
\begin{equation}
(\esG_h f)(x)
  = \sum_{b=1}^n G_b(x) f(x_b),
\quad
(\esH_h f)(x)
  = \sum_{b=1}^n H_b(x) f(x_b),
\label{DefOpApprox}
\end{equation}
where $G_b$ and $H_b$ are functions to be specified.  
Moreover, in view of (\ref{IntEqn}), let $\varphi_h$ 
denote an approximation to $\varphi$ defined by
\begin{equation}
c\varphi_h - \esA_h\varphi_h = f
\quad \hbox{\rm where}\quad
\esA_h = \esG_h + \esH_h.
\label{IntEqnApprox}
\end{equation}

The equation for $\varphi_h$ can be reduced to
an $n\times n$ linear system for the nodal
values $\varphi_h(x_a)$.  Indeed, from
(\ref{DefOpApprox}) and (\ref{IntEqnApprox})
we get, for each $a=1,\ldots,n$, 
\begin{equation}
c\varphi_h(x_a) 
 - \sum_{b=1}^n G_b(x_a) \varphi_h(x_b) 
   - \sum_{b=1}^n H_b(x_a) \varphi_h(x_b) 
     = f(x_a).
\label{IntEqnApproxNodal}
\end{equation}
Notice that, for every solution of the discrete
system (\ref{IntEqnApproxNodal}), we obtain a
solution of the continuous system
(\ref{IntEqnApprox}), namely
\begin{equation}
\varphi_h(x) 
  = \frac{1}{c}\left[ f(x) 
     + \sum_{b=1}^n G_b(x) \varphi_h(x_b) 
     + \sum_{b=1}^n H_b(x) \varphi_h(x_b) \right].
\label{IntEqnApproxSoln}
\end{equation}
Moreover, the converse is also true; every solution
of the continuous system provides a solution of the
discrete system by restriction to the nodes.

The method is completed by specifying the functions
$G_b(x)$ and $H_b(x)$.  In view of (\ref{DefGHops}), 
(\ref{QuadDefn}) and (\ref{DefOpApprox}), we define 
$G_b(x)=G(x,x_b)W_b$.  However, due to the weakly 
singular nature of the kernel function $H(x,y)$, 
a similar definition cannot be made for $H_b(x)$, 
because $H(x,x_b)$ is not continuous or even defined 
when $x=x_b$.  Instead, for any integer $p\ge 0$ we 
define 
\begin{equation}
H_b(x) 
   = \zeta_b(x)H(x,x_b)W_b
                  + \widehat\zeta_b(x)R_x(x_b),
\label{LocCorrTerm}
\end{equation}
where $R_x(x_b)$ is a local polynomial of degree $p$ 
at $x$, evaluated at quadrature point $x_b$.  By a 
local polynomial at $x$ we mean a polynomial in any 
system of rectangular coordinates in the tangent plane 
with origin at $x$; see (\ref{LocPolyDefn}) below.  
For a Lyapunov surface $\Gamma$, such polynomials 
are well-defined in a neighborhood of each point 
$x$, and there is a uniform bound on the size of 
this neighborhood; indeed, they are well-defined 
in the Lyapunov patch $\Gamma_{x,d}$.  The unknown 
coefficients in $R_x$ are defined by enforcing the 
local moment conditions
\begin{equation}
\hbox{$(\esH_h \, \eta_x f_x)(x) 
          = (\esH \, \eta_x f_x)(x)$, 
for all local polynomials $f_x$ up to degree $p$}.
\label{LocMomCond}
\end{equation}
Here $\eta_x\in C^{m,1}(\Gamma,[0,1])$ is any given cutoff 
function which is identically one in a fixed neighborhood 
of $x$, for example $\Gamma_{x,d/4}$, and which is 
identically zero outside some fixed neighborhood of $x$, 
for example $\Gamma_{x,d/2}$.  The local polynomials 
$R_x$ and $f_x$ can be described as floating since they 
are defined in a tangent plane that depends on $x$.

The basic structural form of $H_b$ is similar to
$G_b$, but with a correction in a neighborhood
of $x_b$.  By construction, for all $x$ outside 
a neighborhood of $x_b$, we have $\zeta_b=1$ and 
$\smash{\widehat\zeta_b}=0$, so that $H_b$ is defined 
by the quadrature rule applied to $H$.  On the 
other hand, as $x$ approaches $x_b$, we have
$\zeta_b\to 0$ and $\smash{\widehat\zeta_b}\to 1$, so 
that $H_b$ is determined by the local polynomial $R_x$.
The value of $R_x$ at a node $x_b$ can be interpreted
as a generalized quadrature weight as would arise in
a product integration method.  There are various freedoms 
in the choice of $\eta_x$ used in the moment conditions.  
For the lowest-order method with $p=0$, the local polynomial 
$R_x$ reduces to a constant polynomial which can be extended 
to the entire surface, and the cut-off function is unnecessary 
and can be taken as unity.  However, for higher-order methods 
with $p\ge 1$, the local polynomial $R_x$ in general cannot 
be extended to the entire surface, and a non-trivial cut-off 
function as described above is necessary.  Moreover, a 
different cut-off function could be used for each different 
coefficient in $R_x$.  For convenience, we employ a single 
cut-off function $\eta_x$ for all the moment conditions in 
our analysis below.  

The moment conditions in (\ref{LocMomCond}) lead to a linear 
system of equations for the unknown coefficients of the
local polynomial $R_x$, which can be solved for any given 
point $x$.  Specifically, for any given $x\in\Gamma$, the 
local polynomial $R_x:\Gamma_{x,d}\to\Rkk$ of degree
$p\ge 0$ has the form
\begin{equation}
\begin{split}
R_x(z) 
= C_{x,0} 
  \;+\;&C_{x,\alpha_1} \xi_{x,\alpha_1}(z)
  \;+\; C_{x,\alpha_1\alpha_2} 
                    \xi_{x,\alpha_1}(z)\xi_{x,\alpha_2}(z) \\
  &+ \cdots +\;
  C_{x,\alpha_1\alpha_2\cdots\alpha_p} 
     \xi_{x,\alpha_1}(z)\xi_{x,\alpha_2}(z)\cdots
                                        \xi_{x,\alpha_p}(z), \\
\end{split}
\label{LocPolyDefn}
\end{equation}
where the usual summation convention on pairs of repeated 
indices $\alpha_1,\ldots,\alpha_p$ is implied.  Substituting 
(\ref{LocPolyDefn}) into
(\ref{LocCorrTerm}), and using the standard basis for the 
local polynomials $f_x$, we find that the moment conditions 
in (\ref{LocMomCond}) lead to a linear system for the 
$k\times k$ coefficients $C_{x,0}$ and $C_{x,\alpha_1}$ 
through $C_{x,\alpha_1\alpha_2\cdots\alpha_p}$, namely
\begin{equation}
\begin{gathered}
M_{x,0}^{0}C_{x,0} 
+ \sum_{s=1}^{p} 
  M_{x,\alpha_1\cdots\alpha_s}^{0}
  C_{x,\alpha_1\cdots\alpha_s}
= \Delta_{x}^{0}, \\
M_{x,0}^{\beta_1}C_{x,0} 
+ \sum_{s=1}^{p} 
  M_{x,\alpha_1\cdots\alpha_s}^{\beta_1}
  C_{x,\alpha_1\cdots\alpha_s}
= \Delta_{x}^{\beta_1}, \\
\vdots \\
M_{x,0}^{\beta_1\cdots\beta_p}C_{x,0} 
+ \sum_{s=1}^{p} 
  M_{x,\alpha_1\cdots\alpha_s}^{\beta_1\cdots\beta_p}
  C_{x,\alpha_1\cdots\alpha_s}
= \Delta_{x}^{\beta_1\cdots\beta_p}. \\
\end{gathered}
\label{LocMomCondMat}
\end{equation}
When $p=0$, the indicated sums and all equations 
except the first are empty and the system reduces 
to the single equation 
$\smash{M_{x,0}^{0}C_{x,0}}=\smash{\Delta_{x}^{0}}$,
where
\begin{equation}
\begin{gathered}
M_{x,0}^{0}
= \sum_{b=1}^{n} \widehat\zeta_b(x) \eta_x(x_b)
                                           \in\Reals, \\
\Delta_{x}^{0} 
= \int_{\Gamma} H(x,y)\eta_x(y)\;dA_y
   - \sum_{b=1}^{n} \zeta_b(x)H(x,x_b)\eta_x(x_b)W_b
                                             \in\Rkk. \\
\end{gathered}
\label{LocMomCondCoeffOne}
\end{equation}
When $p\ge 1$, all sums and equations are non-empty 
and the system has the full form indicated in 
(\ref{LocMomCondMat}).  In addition to 
$\smash{M_{x,0}^{0}}$ and $\smash{\Delta_{x}^{0}}$,
we have, for $1\le s,t\le p$, 
\begin{equation}
\begin{gathered}
M_{x,\alpha_1\cdots\alpha_s}^{0}
= \sum_{b=1}^{n} \widehat\zeta_b(x) \eta_x(x_b)
    \xi_{x,\alpha_1}(x_b)\cdots\xi_{x,\alpha_s}(x_b)
                                         \in\Reals, \\
M_{x,0}^{\beta_1\cdots\beta_t}
= \sum_{b=1}^{n} \widehat\zeta_b(x) \eta_x(x_b)
    \xi_{x,\beta_1}(x_b)\cdots\xi_{x,\beta_t}(x_b)
                                         \in\Reals, \\
M_{x,\alpha_1\cdots\alpha_s}^{\beta_1\cdots\beta_t}
= \sum_{b=1}^{n} \widehat\zeta_b(x) \eta_x(x_b)
    \xi_{x,\alpha_1}(x_b)\cdots\xi_{x,\alpha_s}(x_b)
    \xi_{x,\beta_1}(x_b)\cdots\xi_{x,\beta_t}(x_b)
                                         \in\Reals, \\
\end{gathered}
\label{LocMomCondCoeffTwo}
\end{equation}
and
\begin{equation}
\begin{gathered}
\Delta_{x}^{\beta_1\cdots\beta_t} 
= \int_{\Gamma} H(x,y)\eta_x(y)
    \xi_{x,\beta_1}(y)\cdots\xi_{x,\beta_t}(y)\;dA_y 
\hbox to 1.2in{\hfill} \\
\hbox to 0.6in{\hfill} 
   - \sum_{b=1}^{n} \zeta_b(x)H(x,x_b)\eta_x(x_b)
      \xi_{x,\beta_1}(x_b)\cdots\xi_{x,\beta_t}(x_b) W_b
                                              \in\Rkk. \\
\end{gathered}
\label{LocMomCondCoeffThree}
\end{equation}

In contrast to what the number of indices would suggest,
the number of independent $k\times k$ equations in 
(\ref{LocMomCondMat}) grows only quadratically in $p$.  
This follows from the fact that all quantities in the 
equations are fully symmetric in the indices 
$\alpha_1\cdots\alpha_s$ and $\beta_1\cdots\beta_t$ 
for all $1\le s,t\le p$.  Accounting for symmetry, the 
number of independent components of 
$C_{x,\alpha_1 \cdots \alpha_s}$ is not $2^s$, but only 
$(s+1)$.  Summing over $s=1,\ldots,p$, and including the 
single contribution from $C_{x,0}$, we find that the number 
of independent $k\times k$ unknowns, or equivalently, 
independent $k\times k$ equations, is equal 
to \hbox{$(p+1)(p+2)/2$}.  Notice that the implementation 
of the system in (\ref{LocMomCondMat}) requires various 
integral moments of the weakly singular kernel $H$.
These moments can be evaluated numerically using 
techniques such as Duffy \cite{Duffy:1982} or local 
polar coordinate \cite{Bruno:2001,Kunyansky:submitted,
Ying:2006} transformations.  
In the case when $p=0$ and the cut-off function 
is taken as unity, the required moment is known 
analytically in many applications.  However, in 
the case when $p\ge 1$ and the cut-off function is 
non-trivial, the required moments must generally
be obtained numerically.

For any given point $x\in\Gamma$, the solvability 
of the linear system in (\ref{LocMomCondMat}) for the
coefficients of $R_x$ depends on the support of the 
nodal functions $\smash{\widehat\zeta_b}$.  Specifically, 
let $J_x=\{b\;|\;\smash{\widehat\zeta_b(x)}>0\}$ and 
consider sufficiently refined surface decompositions 
such that $\eta_x(x_b)=1$ for all $b\in J_x$.  For 
any local scalar-valued polynomial $f_x$, let $F_x$ 
denote the vector of size \hbox{$(p+1)(p+2)/2$} of 
independent coefficients, where each is weighted according 
to its multiplicity.  Moreover, let $M_x$ denote 
the corresponding square, symmetric coefficient matrix of 
size \hbox{$(p+1)(p+2)/2$} associated with the independent 
equations in (\ref{LocMomCondMat}).  Then by direct 
computation we find
\begin{equation}
F_x\cdot M_x F_x 
= \sum_{b=1}^{n} \widehat\zeta_b(x)[f_x(x_b)]^2\eta_x(x_b)
= \sum_{b\in J_x} \widehat\zeta_b(x)[f_x(x_b)]^2 \ge 0.
\end{equation}
From this we can deduce sufficient conditions for the 
positive-definiteness of $M_x$ and hence the unique 
solvability of (\ref{LocMomCondMat}).  Specifically,
for each $x\in\Gamma$, it is sufficient that $J_x$ be 
non-empty, and that the only polynomial of degree $p$ 
which satisfies $f_x(x_b)=0$ for all $b\in J_x$ be the
zero polynomial.  In view of the number of polynomial
coefficients, this condition implies that $J_x$ must 
contain at least $(p+1)(p+2)/2$ quadrature nodes 
for each $x\in\Gamma$.  When $x$ is itself a 
quadrature node, this implies that the support of 
each function $\smash{\widehat\zeta_b}$ must contain 
at least as many nodes.  Hence the support of the
functions $\smash{\widehat\zeta_b}$ determines the 
solvability of (\ref{LocMomCondMat}).  In our analysis, 
we will assume that $M_x$ has a uniformly bounded 
inverse, namely
\begin{equation*}
\begin{gathered}
|M_x^{-1}|\le C
\quad 
\forall x\in\Gamma, 
\quad
n\ge n_0. 
\end{gathered}
\eqno({\rm A8})
\end{equation*}
A straightforward choice of functions 
$\smash{\widehat\zeta_b}$ which satisfy conditions 
(A7) and (A8) in the case $p=0$ is described below; 
see also \cite{Li:submitted}.

The implementation of the numerical method is 
centered upon (\ref{IntEqnApproxNodal}), 
(\ref{IntEqnApproxSoln}) and (\ref{LocMomCondMat}).
For each quadrature node $x_a$, the linear
system (\ref{LocMomCondMat}) is solved 
to obtain the local polynomials $R_{x_a}$,
which are then used in the linear system 
(\ref{IntEqnApproxNodal}) to obtain the nodal 
values $\varphi_h(x_a)$.  Once these nodal 
values are determined, they can be extended to 
a continuous function.  Specifically, given any 
point $x$, the linear system in (\ref{LocMomCondMat}) 
can be solved to obtain the local polynomial 
$R_{x}$, which can then be used in the 
interpolation equation (\ref{IntEqnApproxSoln}) 
to obtain the value of $\varphi_h(x)$.  
Notice that, in general, the construction of 
the local polynomials $R_{x_a}$ and $R_{x}$
requires the evaluation of local moments of
the weakly singular kernel $H$, and the
evaluation of local Cartesian coordinates
$\xi_{x_a}$ and $\xi_{x}$.  As will be shown 
later, the values of $R_{x_a}$ and $R_{x}$
are independent of the choice of orthonormal 
basis associated with $\xi_{x_a}$ and $\xi_{x}$.

\subsection{Illustrative example}

Here we illustrate the form of the nodal equations
(\ref{IntEqnApproxNodal}) in the lowest-order case
with $p=0$.  In this case, the equations take a 
particularly simple form and the evaluation of local 
Cartesian coordinates is not necessary, and the 
evaluation of weakly singular integrals is typically 
not necessary.

We begin by describing partition of unity functions
$\zeta_a$ and $\smash{\widehat\zeta_a}$ which satisfy 
conditions (A7) and (A8).  Consider an auxiliary 
decomposition of $\Gamma$ into Voronoi cells, where 
each cell contains a single quadrature point $x_a$.  
Each cell can be mapped to a unit circle, with $x_a$ 
mapped to the center.  A simple quadratic function 
$z = (x^2 + y^2)/2$ in the unit circle can be mapped 
back to the cell as the central part of $\zeta_a$.  
We can then introduce an offset boundary which is 
displaced outward from the cell boundary by a distance 
of $\veps/2$, where $\veps = \min_{a \neq b}|x_a - x_b|$.  
The mapped quadratic function, which by design has the 
value $1/2$ on the cell boundary, can be extended to 
achieve a value of $1$ on the offset boundary, and then 
further extended to the rest of $\Gamma$ with the 
constant value $1$.  Notice that the functions 
$\zeta_a$ and $\smash{\widehat\zeta_a}$ so constructed 
have the convenient nodal property that
$\zeta_a(x_b)=1-\delta_{ab}$ and
$\smash{\widehat\zeta_a}(x_b)=\delta_{ab}$.  Moreover, 
the supports of $\smash{\widehat\zeta_a}$ overlap on 
the entire surface and we have 
$\sum_a \smash{\widehat\zeta_a}(x)\ge 1/2$ for all 
$x\in\Gamma$.

For the method with $p=0$, the local polynomial $R_x$ 
reduces to a constant polynomial $R_x(z)\equiv C_{x,0}$, 
and the cutoff function $\eta_x$ can be taken as unity.  
The linear system in (\ref{LocMomCondMat}) reduces
to the single equation 
$\smash{M_{x,0}^{0}C_{x,0}}=\smash{\Delta_{x}^{0}}$,
which implies
\begin{equation}
C_{x,0} = {\int_\Gamma H(x,y) \; dA_y 
               - \sum_{b=1}^n \zeta_{b}(x)  
                        H(x,x_{b}) W_{b} 
                \over \sum_{b=1}^n \widehat\zeta_{b}(x)}.
\end{equation}
This solution is well-defined and bounded for any sequence 
of decompositions by properties of the nodal functions 
$\zeta_a$ and $\smash{\widehat\zeta_a}$ and the kernel function 
$H$.  In various applications, the weakly singular integral 
in the above expression is known analytically and hence 
numerical evaluation is not necessary.  

The nodal equations in (\ref{IntEqnApproxNodal}) take a
particularly simple form.  Indeed, because our choice 
of the nodal partition of unity functions $\zeta_a$ and 
$\smash{\widehat\zeta_a}$ has the property that 
$\zeta_a(x_b)=1-\delta_{ab}$ and 
$\smash{\widehat\zeta_a}(x_b)=\delta_{ab}$, 
the equations become
\begin{equation}
\begin{gathered}
\gamma(x_a)\varphi_h(x_a)
 - \smash{\sum_{b=1}^n} G(x_a,x_b) \varphi_h(x_b) W_b 
\hskip1.5in\\
\hskip1.5in
  - \sum_{b=1 \atop b\ne a}^n 
     H(x_a,x_b)[\varphi_h(x_b)-\varphi_h(x_a)] W_b
      = f(x_a),
\end{gathered}
\label{IntEqnApproxNodalLow}
\end{equation}
where $\gamma(x_a)=cI-\int_\Gamma H(x_a,y)\,dA_y$.
This discrete system is similar to the classic 
singularity subtraction method discussed by various 
authors \cite{Anselone:1981,Krylov:1958}.  The
factor $[\varphi_h(x_b)-\varphi_h(x_a)]$ can be 
interpreted as cancelling the weak singularity in
$H(x_a,x_b)$.  Indeed, since the sum extends over 
$b\ne a$ only, the product 
$H(x_a,x_b)[\varphi_h(x_b)-\varphi_h(x_a)]$
can be interpreted as vanishing when $b=a$.  For 
methods with $p\ge 1$, a similar but higher-order 
cancellation can be interpreted to occur.  Once 
the nodal values of $\varphi_h$ are determined, 
they can be extended to a continuous function 
using the interpolation equation in 
(\ref{IntEqnApproxSoln}).  Notice that the
nodal values of $\varphi_h$ can be computed 
without explicit knowledge of the partition 
of unity functions.  Various numerical examples 
with this method are given in \cite{Li:submitted}.

\subsection{Solvability and convergence theorem}

The following result establishes the solvability and
convergence of the locally-corrected Nystr{\"o}m method
defined in (\ref{DefOpApprox})--(\ref{LocMomCondCoeffThree}).
We consider the method with a quadrature rule of arbitrary 
order $\ell\ge 1$, a local polynomial correction of arbitrary 
degree $p\ge 0$, and a surface with regularity index $m\ge 0$.
In view of Theorem \ref{IntEqnThm}, we suppose that
the constant $c\ne 0$ is not an associated eigenvalue of
the operator $\esA$, so that the given boundary integral
equation has a unique solution $\varphi$.  Below we use 
$C_\varphi$ to denote a constant depending on $\varphi$.

\medskip
\begin{theorem} \label{ConvThm}
Under conditions {\rm (A0)}--{\rm (A8)}, there exists
a unique approximation $\varphi_h\in C^0(\Gamma,\Rk)$ for 
any closed, bounded Lyapunov surface $\Gamma\in C^{1,1}$
and boundary data $f\in C^0(\Gamma,\Rk)$ for all $h>0$ 
sufficiently small.  Moreover, if 
$\varphi\in C^{m,1}(\Gamma,\Rk)$ and $\Gamma\in C^{m+1,1}$, 
then as $h\to 0$ 
\begin{equation*}
\begin{matrix}
{\rm (i)} \hfill &\quad
||\varphi_h-\varphi|| \to 0, \hfill
      &\quad\forall \ell\ge 1, p\ge 0, m\ge 0, \hfill \\
\\
{\rm (ii)} \hfill &\quad
||\varphi_h-\varphi|| \le C_\varphi h, \hfill
      &\quad\forall \ell\ge 1, p = 0, m\ge 1, \hfill \\
\\
{\rm (iii)} \hfill &\quad
||\varphi_h-\varphi|| \le C_\varphi 
                   h^{\min(\ell,p,m)}, \hfill
      &\quad\forall \ell\ge 1, p\ge 1, m\ge 1. \hfill \\
\end{matrix}
\end{equation*}
\end{theorem}

Thus, under suitable assumptions,  the method defined
by (\ref{DefOpApprox})--(\ref{LocMomCondCoeffThree}) is 
convergent in the usual maximum or $C^0$-norm.  The rate 
of convergence depends on the order $\ell$ of the
quadrature rule, the degree $p$ of the local polynomial
correction, and the the regularity index $m$ of the 
exact solution $\varphi$ and the surface $\Gamma$.
In the minimal regularity case with $m=0$, there is 
no lower bound on the rate, and in the higher regularity 
case with $m\ge 1$, the rate is at least linear.
The rate of convergence is independent of the order 
of the quadrature rule in the case when $p=0$, which 
corresponds to the lowest degree of correction.  Higher 
rates of convergence are obtained when $p\ge 1$, which 
corresponds to higher degrees of correction.  Notice that 
the rates of convergence stated above are lower bounds; 
they could possibly be higher in certain circumstances, 
for example in smooth problems with periodicity, for which 
some quadrature rules are known to have special properties
\cite{Bruno:2001,Kurganov:2009}.  We remark that the
method considered here is based on open quadrature rules 
as required by condition (A5).

As a special case, Theorem \ref{ConvThm} with $p=0$ 
establishes the convergence of a method similar 
to the classic singularity subtraction method 
considered previously by various authors 
\cite{Anselone:1981,Krylov:1958}.  The two methods
lead to apparently identical discrete systems for
the nodal approximations, but differ in how the
nodal approximations are interpolated over
the surface.  Convergence results for this method
appear to be not well-known.  The results derived
here make crucial use of the structure of the nodal
functions $H_b(x)$ defined in (\ref{LocCorrTerm}),
the moment conditions defined in (\ref{LocMomCond}),
and various properties of the nodal partition of unity
functions $\zeta_b$ and $\smash{\widehat\zeta_b}$.  
Such ingredients appear to have not been considered 
in previous studies of the classic method.  We 
remark that the linear convergence result for $p=0$ 
is delicate and relies on a Tricomi-like property 
of the potential $\esH$ implied by condition (A3)(ii).  
The results for $p\ge 1$ are qualitatively different 
and rely mainly on the regularizing effect of the 
local polynomial correction. Specifically, the linear 
convergence result in part (ii) requires condition 
(A3) whereas the convergence results in parts (i) 
and (iii) do not.

\section{Proof}

In this section we provide a proof of Theorem \ref{ConvThm}.
We use the same notation and conventions as in previous
sections.  Specifically, we use $C$, $C'$, $C''$ and so 
on, to denote generic positive constants whose value may 
change from one appearance to the next, use $|\cdot|$ 
to denote a Euclidean norm or the measure of a surface,
as determined by the context, and use $||\cdot||$ to 
denote the usual maximum norm on $C^0(\Gamma,\Rk)$.

\subsection{Collective compactness}
 
We first outline a result, based on the theory of 
collectively compact operators \cite{Anselone:1971}, 
which plays a fundamental role in the analysis of 
Nystr{\"o}m methods \cite{Atkinson:1997,Goldberg:1997,
Kress:1989}. Let $\esA$ be a linear, compact operator 
on $C^0(\Gamma,\Rk)$ as given in (\ref{DefAop}), and 
let $\esA_n$ ($n\ge n_0$) be a sequence of linear, 
finite-rank operators on $C^0(\Gamma,\Rk)$ as given 
in (\ref{IntEqnApprox}), where for convenience we 
consider subscripts $n\to\infty$ in place of 
$h_n\to 0$.  Consider the following conditions:
\begin{itemize}
\item[(C1)] For each $v\in C^0(\Gamma,\Rk)$, 
$(\esA_n v)(x)\to (\esA v)(x)$ uniformly in $x\in\Gamma$.
\item[(C2)] $|(\esA_n v)(x)|\le C$ for all $x\in\Gamma$ 
and $v\in C^0(\Gamma,\Rk)$ with $||v||\le 1$.
\item[(C3)] For every $\veps>0$ there exists $\delta>0$ 
and $N\ge n_0$ such that
$|(\esA_n v)(x)-(\esA_n v)(y)|< \veps$ for all $x,y\in\Gamma$ 
with $|x-y|<\delta$, $v\in C^0(\Gamma,\Rk)$ with 
$||v||\le 1$, and $n\ge N$.
\end{itemize}

Condition (C1) states that $\esA_n$ converges to $\esA$ 
pointwise in $C^0(\Gamma,\Rk)$.  Conditions (C2) and (C3) 
state, respectively, that 
$S=\{\esA_n v\;|\; n\ge n_0, \; ||v||\le 1\}$ is an 
equibounded and equicontinuous subset of $C^0(\Gamma,\Rk)$.  
These two conditions imply, by the Arzela-Ascoli Theorem, 
that $S$ is relatively compact; by definition, the sequence 
$\esA_n$ is then called collectively compact.  For such 
sequences the following well-known result holds
\cite{Atkinson:1997, Goldberg:1997,Kress:1989}.

\begin{theorem}
\label{thmConv}
Let $\esA$ and $\esA_n$ $(n\ge n_0)$ be linear operators 
on $C^0(\Gamma,\Rk)$, where $\esA$ is compact and $\esA_n$ 
are finite-rank, and assume that {\rm (C1)}--{\rm (C3)} 
hold.  If $c\varphi-\esA\varphi=f$ is uniquely solvable for 
$\varphi$, then there exist constants $C_\varphi>0$ and 
$N_\varphi\ge n_0$ such that 
$c \varphi_n - \esA_n\varphi_n = f$ is uniquely 
solvable for $\varphi_n$, and moreover
\begin{equation}
||\varphi_n - \varphi || 
   \le C_\varphi || \esA_n \varphi - \esA\varphi ||, 
\quad\forall n\ge N_\varphi.
\end{equation}
\end{theorem}

In what follows, we establish conditions (C1)--(C3) for the
operator $\esA_n=\esG_n+\esH_n$ defined in (\ref{IntEqnApprox}).
We concentrate on $\esH_n$ since the result for $\esG_n$ is 
straightforward by continuity of its kernel.  Once 
(C1)--(C3) are established, the result in Theorem \ref{ConvThm}
will follow from a bound on $||\esA_n\varphi - \esA\varphi||$ 
under regularity assumptions on $\Gamma$ and $\varphi$.

\subsection{Lemmata}

We begin with a collection of useful results regarding
the weakly singular kernel $H$.  Parts (i) and (ii)
below show that the integral operator defined by $H$
indeed maps $C^0(\Gamma,\Rk)$ into itself, and is
equibounded and equicontinuous on bounded subsets,
hence compact.  Part (iii) is an important inequality
that will be used in the sequel.

\begin{lemma}
\label{LemmaOne}
Let $\Gamma$ satisfy {\rm (A0)} with Lyapunov radius 
$d>0$ and let $H$ satisfy {\rm (A2)} with exponent 
$0 < \mu \le 1$.  Then:
\begin{itemize}
\item[{\rm (i)}] $\int_\Gamma |H(x,y)|\, dA_y\le C$ for all 
$x\in\Gamma$.
\item[{\rm (ii)}] For every $\veps>0$ there exists $\delta>0$ 
such that $\int_\Gamma |H(x_*,y) - H(x_0,y)|\, dA_y\le \veps$ 
for all $x_0,x_*\in\Gamma$ with $|x_0-x_*|\le \delta$. 
\item[{\rm (iii)}] $|H(x_*,y) - H(x_0,y)| \le 
C|x_*-x_0|/|y - x_0|^{3-\mu}$ for all $x_0,x_*,y\in\Gamma$ 
with $0<|x_*-x_0|\le d/3$ and $2|x_*-x_0|\le |y-x_0| \le d$.
\end{itemize}
\end{lemma}

\begin{proof}
Parts (i) and (ii) are classic results for weakly singular 
integrals, see for example \cite{Gunter:1967}, where the 
proof of (ii) relies on (iii).  For brevity, we illustrate 
only (iii).  To begin, we consider the form of $H$ in 
(\ref{Hker}) and use the triangle inequality to get
\begin{equation}
\begin{split}
|H(x_*,y) - H(x_0,y)| 
&=   \left| {u(x_*,y) \over |y - x_*|^{2-\mu}}
                  - {u(x_0,y) \over |y - x_0|^{2-\mu}} \right| \\
&\le  {|u(x_*,y) - u(x_0,y)| \over |y - x_*|^{2 - \mu}} 
    +\left| {u(x_0,y) \over |y - x_*|^{2 - \mu}} 
              - {u(x_0,y) \over |y - x_0|^{2 - \mu}} \right|.
\end{split}
\label{LOne1}
\end{equation}

For the first term in (\ref{LOne1}), we consider the local
Cartesian coordinate map $\psi_{x_0}$ in the Lyapunov patch 
$\Gamma_{x_0,d}$, and the curve $x(\tau)=\psi_{x_0}(\tau\xi_*)$, 
$0\le\tau\le 1$, from $x_0=\psi_{x_0}(0)$ to 
$x_*=\psi_{x_0}(\xi_*)$. Using this curve, we have
\begin{equation}
u(x_*,y) - u(x_0,y) 
= \int_0^1 D_x u(x(\tau),y) {d\over d\tau}x(\tau)\; d\tau.
\end{equation}
From (\ref{HkerUx}) and the relation
${dx / d\tau} = (\partial\psi_{x_0}/\partial\xi) \xi_*$,
and the facts that $|\partial\psi_{x_0}/\partial\xi|\le C$ 
and $|\xi_*|\le C|x_*-x_0|$, which follow from the Lipschitz
properties of $\psi_{x_0}$ and $\psi^{-1}_{x_0}$, we get
\begin{equation}
|u(x_*,y) - u(x_0,y)| 
\le  \int_0^1 {C|x_*-x_0|\over |x(\tau)-y|} \; d\tau.
\label{LOne1.5}
\end{equation}
From the definition of $x(\tau)$, we deduce that $|x(\tau)-x_0|$ 
is an increasing function, and hence $|x(\tau)-x_0|\le |x_*-x_0|$ 
for all $0\le \tau\le 1$.  Moreover, by hypothesis, we have
$0<|x_*-x_0|\le |y-x_0|/2$.  These two results imply, with 
the help of the triangle inequality,
\begin{equation}
{1\over 2}\le {|y-x(\tau)|\over |y-x_0|}\le {3\over 2},
\quad
\tau\in[0,1],
\quad
\hbox{\rm and}
\quad
{1\over 2}\le {|y-x_*|\over |y-x_0|}\le {3\over 2},
\label{LOne1.6}
\end{equation}
and using (\ref{LOne1.6}) in (\ref{LOne1.5}), we find
\begin{equation}
|u(x_*,y) - u(x_0,y)| 
\le   {C|x_*-x_0|\over |y-x_*|}.
\label{LOne2}
\end{equation}

For the second term in (\ref{LOne1}), we notice
that $|u|\le C$ by (\ref{HkerUbnd}).  Hence, using 
the notation $r=|y-x|$, we have
\begin{equation}
\left| {u(x_0,y) \over |y - x_*|^{2 - \mu}} 
        - {u(x_0,y) \over |y - x_0|^{2 - \mu}} \right|
\le C \left|{1 \over r_*^{2 - \mu}} 
        - {1 \over r_0^{2 - \mu}}\right|
= {C|r_0^{2-\mu} - r_*^{2-\mu}|
                       \over r_0^{2-\mu} r_*^{2-\mu}}.
\label{LOne4}
\end{equation}
From (\ref{LOne1.6}) and the hypothesis 
$0<2|x_*-x_0|\le |y-x_0|\le d$ we notice that
\begin{equation}
0< r_0\le d
\quad
\hbox{\rm and}
\quad
0< r_0 \le 2r_* \le 3r_0 \le 3d.
\label{LOne5}
\end{equation}
Regarding the factor $|r_0^{2-\mu} - r_*^{2-\mu}|$ in 
(\ref{LOne4}), we have, since the function
$r^{1-\mu}$ is monotonic,
\begin{equation}
|r_0^{2-\mu} - r_*^{2-\mu}|
= \left|\int_{r_0}^{r_*} {d\over dr}[r^{2-\mu}]\; dr \right|
\le (2-\mu) \max\{r_0^{1-\mu},r_*^{1-\mu}\} |r_*-r_0|.
\label{LOne6}
\end{equation}
Combining (\ref{LOne6}) and (\ref{LOne5}) with
(\ref{LOne4}), we find
\begin{equation}
\left| {u(x_0,y) \over |y - x_*|^{2 - \mu}} 
        - {u(x_0,y) \over |y - x_0|^{2 - \mu}} \right|
\le {C|r_* - r_0| \over r_0^{3-\mu}}
\le {C|x_* - x_0| \over |y-x_0|^{3-\mu}},
\label{LOne7}
\end{equation}
where the last inequality follows from the definitions 
of $r_0$ and $r_*$, and the straightforward inequality 
$|r_*-r_0|\le |x_*-x_0|$.  Substituting (\ref{LOne7}) 
and (\ref{LOne2}) into (\ref{LOne1}), and again
using (\ref{LOne1.6}) on the first term, we obtain 
the desired result.
\end{proof}

The next result can be viewed, in part, as a
discrete analog of Lemma \ref{LemmaOne}.  It will
play a central role in establishing the collective
compactness of the operators $\esH_n$ associated 
with the numerical method.  The defining elements 
of the method, which are the surface decomposition, 
quadrature rule, and nodal partition of unity 
functions, are denoted by the set 
$\{\smash{\Gamma^e}, \smash{x_q^e}, \smash{W_q^e}, 
\smash{\zeta_q^e}, \smash{\widehat\zeta_q^e}\}$.

\begin{lemma}
\label{LemmaTwo}
Let $\Gamma$ satisfy {\rm (A0)} with Lyapunov 
radius $d>0$, $H$ satisfy {\rm (A2)} with 
exponent $0 < \mu \le 1$, and 
$\{\smash{\Gamma^e}, \smash{x_q^e}, 
\smash{W_q^e}, \smash{\zeta_q^e}, 
\smash{\widehat\zeta_q^e}\}$
satisfy {\rm (A4)}--{\rm (A7)}.  Then:
\begin{itemize}
\item[{\rm (i)}] $\zeta_b(x)H(x,x_b)$ $(:=0$ at $x=x_b)$ 
is continuous in $x\in\Gamma$ for each $b=1,\ldots,n$
and $n\ge n_0$.
\item[{\rm (ii)}] $\sum_{b=1}^n |\zeta_b(x)H(x,x_b)W_b|\le C$ 
for all $x\in\Gamma$ and $n\ge n_0$.
\item[{\rm (iii)}] For every $\veps>0$ there exists $\delta>0$
such that $\sum_{b=1}^n |\zeta_b(x_*)H(x_*,x_b)W_b - 
\zeta_b(x_0)H(x_0,x_b)W_b| \le \veps$ for all $x_0,x_*\in\Gamma$ 
with $|x_0-x_*|\le \delta$ and $n\ge n_0$.
\end{itemize}
\end{lemma}

\begin{proof}
For part (i), we notice that $\zeta_b(x)$ is continuous
for all $x$, and $H(x,x_b)$ is continuous for all
$x\ne x_b$.  Hence, to establish the result, we need
only verify that $\zeta_b(x)H(x,x_b)\to 0$ as $x\to x_b$.
From {\rm (A2)} and {\rm (A7)}, we get
\begin{equation}
|\zeta_b(x)H(x,x_b)|
=  {|\zeta_b(x)u(x,x_b)|\over |x-x_b|^{2-\mu}}
\le  {C|x-x_b|^{\mu}\over h^2},
\end{equation}
and the result follows since $0<\mu\le 1$.

For part (ii), let $\delta\in(0,d/3]$ and $\beta\ge 10$ 
be any given numbers and consider sufficiently refined 
surface decompositions with $E\ge E_{\beta,\delta}$,
where $E_{\beta,\delta}\ge E_0$ is a given integer 
defined such that $0<\beta h \le \delta/2$ for all 
$E\ge E_{\beta,\delta}$.  We consider an arbitrary 
$x=x_0\in\Gamma$ and consider collections of surface 
elements $\Gamma^e$, $e=1,\ldots,E$, defined by
\begin{equation}
I_{x_0,\delta}=\{e\;|\; \Gamma^e \subset\Gamma_{x_0,\delta}\}
\quad
\hbox{\rm and}
\quad
I_{x_0,\beta h}=\{e\;|\; \Gamma^e \subset\Gamma_{x_0,\beta h}\},
\end{equation}
and note that, by design, the sets $I_{x_0,\delta}$, 
$I_{x_0,\beta h}\subset I_{x_0,\delta}$ and 
$I_{x_0,\delta}\backslash I_{x_0,\beta h}$ 
are non-empty for all $E\ge E_{\beta,\delta}$.  
Moreover, we consider the decomposition
\begin{equation}
\sum_{b=1}^n \zeta_b(x_0)H(x_0,x_b)W_b 
= \esS_1(x_0) + \esS_2(x_0) + \esS_3(x_0),
\label{LTwoDecomp} 
\end{equation}
where
\begin{align}
\esS_1(x_0) 
&= \sum_{e\in I_{x_0,\beta h}} 
                    \sum_{q=1}^Q \zeta_q^e(x_0)H(x_0,x_q^e)W_q^e, 
\label{LTwoSum1} \\
\esS_2(x_0) 
&= \sum_{e\in I_{x_0,\delta}\backslash I_{x_0,\beta h}} 
                    \sum_{q=1}^Q \zeta_q^e(x_0)H(x_0,x_q^e)W_q^e, 
\label{LTwoSum2} \\
\esS_3(x_0) 
&= \sum_{e\not\in I_{x_0,\delta}} 
                    \sum_{q=1}^Q \zeta_q^e(x_0)H(x_0,x_q^e)W_q^e. 
\label{LTwoSum3} 
\end{align}
To establish the result, we show that each of the sums 
$\esS_1$, $\esS_2$ and $\esS_3$ is uniformly bounded for 
all $E\ge E_{\beta,\delta}$ and $x_0\in\Gamma$.  Boundedness 
for the finite interval $E\in[E_0,E_{\beta,\delta}]$ is a 
straightforward consequence of part (i) and will be omitted 
for brevity.

For the term $\esS_1$ in (\ref{LTwoDecomp}), we consider the 
local Cartesian coordinate map $\psi_{x_0}$ in the Lyapunov 
patch $\Gamma_{x_0,d}$.  Since 
$\Gamma_{x_0,\beta h}\subset\Gamma_{x_0,d}$, and $\psi_{x_0}$ 
and $\psi_{x_0}^{-1}$ are Lipschitz, we have the area bound 
$|\Gamma_{x_0,\beta h}| \le C h^2$.  Moreover, from (A4) and 
the fact that 
$(\cup_{e\in I_{x_0,\beta h}}\Gamma^e)\subset\Gamma_{x_0,\beta h}$,
we get $C h^2 |I_{x_0,\beta h}| \le |\Gamma_{x_0,\beta h}|$, 
where $|I_{x_0,\beta h}|$ denotes the number of elements in the
index set $I_{x_0,\beta h}$.  From these two observations, we
deduce the uniform bound
\begin{equation}
|I_{x_0,\beta h}| \le C,
\quad
\forall E\ge E_{\beta,\delta},
\quad
x_0\in\Gamma.
\label{LTwoSum1A}
\end{equation}
From (\ref{LTwoSum1}) we get, using (A2) and (A7), 
\begin{equation}
\begin{split}
|\esS_1(x_0)| 
&\le \sum_{e\in I_{x_0,\beta h}} 
       \sum_{q=1}^Q |\zeta_q^e(x_0)H(x_0,x_q^e)W_q^e|, \\
&\le \sum_{e\in I_{x_0,\beta h}} 
       \sum_{q=1}^Q {C|x_0- x_q^e|^2\over h^2}
                       {C\over |x_0- x_q^e|^{2-\mu}} W_q^e. 
\end{split}
\label{LTwoSum1B}
\end{equation}
Since $x_q^e\in\Gamma^e\subset \Gamma_{x_0,\beta h}$ we
have $|x_0- x_q^e|\le \beta h$, and from (A5) we have
$\sum_{q=1}^Q W_q^e\le C h^2$.  Using these results
together with (\ref{LTwoSum1B}) and (\ref{LTwoSum1A}),
we find that
\begin{equation}
|\esS_1(x_0)| \le C h^{\mu}, 
\quad
\forall E\ge E_{\beta,\delta},
\quad
x_0\in\Gamma,
\label{LTwoSum1C}
\end{equation}
which establishes the result for $\esS_1$ since $0<\mu\le 1$.  
Indeed, the sum $\esS_1(x_0)\to 0$ as $E\to\infty$ ($h\to 0$) 
uniformly in $x_0\in\Gamma$.

For the term $\esS_2$ in (\ref{LTwoDecomp}), we notice 
first that $\Gamma_{x_0,(\beta-1)h}\subset\Gamma_{x_0,\beta h}$
and that $\dist(\partial \Gamma_{x_0,(\beta-1)h},
\partial\Gamma_{x_0,\beta h})\ge h\ge\diam(\Gamma^e)$.
Hence, if $\Gamma^e\cap\Gamma_{x_0,(\beta-1)h}\ne\emptyset$,
then $\Gamma^e\subset \Gamma_{x_0,\beta h}$.
From this we deduce that $\Gamma^e\subset\Gamma_{x_0,\delta}$ 
and $\Gamma^e\cap\Gamma_{x_0,(\beta-1)h}=\emptyset$ for all 
$e\in I_{x_0,\delta}\backslash I_{x_0,\beta h}$.  This 
implies
\begin{equation}
0<(\beta-1)h < |x_q^e-x_0| \le \delta,
\quad
\forall e\in I_{x_0,\delta}\backslash I_{x_0,\beta h},
\quad 
q=1,\ldots,Q.
\label{LTwoSum2A}
\end{equation}
To each quadrature element $\Gamma^e$ with 
$e\in I_{x_0,\delta}\backslash I_{x_0,\beta h}$
we associate a distinguished node 
$q(x_0,e)\in\{1,\ldots,Q\}$, radius 
$\lambda(x_0,e)\in((\beta-1)h,\delta]$ and 
subset $\Gamma_{x_0,\rm in}^e\subset\Gamma^e$ 
defined by
\begin{equation}
\begin{gathered}
\min_{q=1,\ldots,Q} |x_q^e-x_0| = |x_{q(x_0,e)}^e - x_0|, \\
\lambda(x_0,e) = |x_{q(x_0,e)}^e - x_0|, 
\qquad
\Gamma_{x_0,\rm in}^e = \Gamma^e \cap B(x_0,\lambda(x_0,e)).
\end{gathered}
\label{LTwoSum2B}
\end{equation}
From (A5) we get
$Ch \le \dist(x^e_q,\partial\Gamma^e)\le h$ for all
$q=1,\ldots,Q$, and from this we deduce that the area 
of $\Gamma_{x_0,\rm in}^e$ is bounded from below, 
namely $|\Gamma_{x_0,\rm in}^e|\ge Ch^2$.  Combining 
this result with (A4), we find that
\begin{equation}
C|\Gamma^e|\le |\Gamma_{x_0,\rm in}^e| \le |\Gamma^e|,
\quad
\forall e\in I_{x_0,\delta}\backslash I_{x_0,\beta h},
\quad
E\ge E_{\beta,\delta}.
\label{LTwoSum2C}
\end{equation}
From (\ref{LTwoSum2}) we get, using (A2), (A4) and (A5),
and the fact that $|\zeta_q^e(x_0)|\le 1$, together with 
(\ref{LTwoSum2B}$)_1$ and (\ref{LTwoSum2C}),
\begin{equation}
\begin{split}
|\esS_2(x_0)| 
&\le \sum_{e\in I_{x_0,\delta}\backslash I_{x_0,\beta h}} 
       \sum_{q=1}^Q |\zeta_q^e(x_0)H(x_0,x_q^e)W_q^e|, \\
&\le \sum_{e\in I_{x_0,\delta}\backslash I_{x_0,\beta h}} 
       \sum_{q=1}^Q {C\over |x_0- x_q^e|^{2-\mu}} W_q^e, \\ 
&\le \sum_{e\in I_{x_0,\delta}\backslash I_{x_0,\beta h}} 
       {C\over |x_0- x_{q(x_0,e)}^e|^{2-\mu}} \sum_{q=1}^Q W_q^e, \\ 
&\le \sum_{e\in I_{x_0,\delta}\backslash I_{x_0,\beta h}} 
       {C |\Gamma_{x_0,\rm in}^e| \over |x_0- x_{q(x_0,e)}^e|^{2-\mu}} . 
\end{split}
\label{LTwoSum2D}
\end{equation}
Moreover, by definition of $\Gamma_{x_0,\rm in}^e$ and the
fact that $x_{q(x_0,e)}^e\in\partial B(x_0,\lambda(x_0,e))$,
we get
\begin{equation}
{1\over |x_{q(x_0,e)}^e - x_0|} 
         \le {1\over |y-x_0|},
\quad
\forall y\in\Gamma_{x_0,\rm in}^e.
\label{LTwoSum2E}
\end{equation}
Combining (\ref{LTwoSum2E}) with (\ref{LTwoSum2D}), 
and using the fact that 
$\Gamma_{x_0,\rm in}^e\subset\Gamma^e$, we find 
\begin{equation}
\begin{split}
|\esS_2(x_0)| 
&\le \sum_{e\in I_{x_0,\delta}\backslash I_{x_0,\beta h}} 
   \int_{\Gamma_{x_0,\rm in}^e} {C \over |y-x_0|^{2-\mu}}\; dA_y, \\ 
&\le \sum_{e\in I_{x_0,\delta}\backslash I_{x_0,\beta h}} 
   \int_{\Gamma^e} {C \over |y-x_0|^{2-\mu}}\; dA_y, \\ 
&\le \int_{\Gamma_{x_0,\delta}} {C \over |y-x_0|^{2-\mu}}\; dA_y
\le C \delta^\mu, 
\quad 
\forall E\ge E_{\beta,\delta},  
\quad
x_0\in\Gamma,
\end{split}
\label{LTwoSum2F}
\end{equation}
where the last inequality follows from a direct estimate 
of the integral using polar coordinates and the fact that
the local coordinate maps $\psi_{x_0}$ and $\psi_{x_0}^{-1}$
are Lipschitz.  Thus the result for $\esS_2$ is established.

For the final term $\esS_3$ in (\ref{LTwoDecomp}), 
we notice similar to before that 
$\Gamma_{x_0,\delta-h}\subset\Gamma_{x_0,\delta}$
and that $\dist(\partial \Gamma_{x_0,\delta-h},
\partial\Gamma_{x_0,\delta})\ge h\ge\diam(\Gamma^e)$.
Hence, if $\Gamma^e\cap\Gamma_{x_0,\delta-h}\ne\emptyset$,
then $\Gamma^e\subset \Gamma_{x_0,\delta}$.
From this we deduce that 
$\Gamma^e\cap\Gamma_{x_0,\delta-h}=\emptyset$ for all
$e\not\in I_{x_0,\delta}$.  This implies
\begin{equation}
0<\delta-h < |y-x_0| 
\quad
\forall y\in\Gamma^e,
\quad
e\not\in I_{x_0,\delta},
\quad 
E\ge E_{\beta,\delta}.
\label{LTwoSum3A}
\end{equation}
From (\ref{LTwoSum3}) we get, using (A2), (A4) and (A5),
and the fact that $|\zeta_q^e(x_0)|\le 1$ and 
$\beta h \le \delta/2$, together with (\ref{LTwoSum3A}),
\begin{equation}
\begin{split}
|\esS_3(x_0)| 
&\le \sum_{e\not\in I_{x_0,\delta}} 
       \sum_{q=1}^Q |\zeta_q^e(x_0)H(x_0,x_q^e)W_q^e|, \\
&\le \sum_{e\not\in I_{x_0,\delta}} 
       \sum_{q=1}^Q {C\over |x_0- x_q^e|^{2-\mu}} W_q^e, \\ 
&\le \sum_{e\not\in I_{x_0,\delta}} 
       \sum_{q=1}^Q {C\over \delta^{2-\mu}} W_q^e, \\ 
&\le  {C|\Gamma|\over \delta^{2-\mu}} 
\le C \delta^{\mu-2}, 
\quad 
\forall E\ge E_{\beta,\delta},  
\quad
x_0\in\Gamma.
\end{split}
\label{LTwoSum3B}
\end{equation}
Thus the result for $\esS_3$ is established.  The desired 
result stated in part (ii) follows from (\ref{LTwoSum3B}), 
(\ref{LTwoSum2F}), (\ref{LTwoSum1C}) and (\ref{LTwoDecomp}).

For part (iii), let $\delta\in(0,d/3]$, $\beta\ge 10$ 
and $E_{\beta,\delta}\ge E_0$ be given numbers defined 
as before, so $0<\beta h \le \delta/2$ for all 
$E\ge E_{\beta,\delta}$.   Moreover, in view of (A7), 
we suppose that $E_{\beta,\delta}$ is sufficiently large
so that $\diam(\supp(\smash{\widehat\zeta_b}))\le \delta/2$ 
for all $b=1,\ldots,n$ and $E\ge E_{\beta,\delta}$.
Since the result in part (iii) trivially holds when 
$x_0=x_*$, we consider without loss of generality any 
arbitrary points $x_0,x_*\in\Gamma$ with 
$0<|x_0-x_*|\le \delta$.  Given such points, we 
consider the following collections of surface 
elements $\Gamma^e$, $e=1,\ldots,E$, defined in
the same way as before:
\begin{equation}
I_{x_0,\beta h} 
\subset
I_{x_0,2\delta}
\subset
I_{x_0,d}
\quad
\hbox{\rm and}
\quad
I_{x_*,\beta h}
\subset
I_{x_*,3\delta}
\subset
I_{x_*,d}.
\end{equation}
By design, each of the above sets is non-empty, 
as well as each of 
$I_{x_0,d}\backslash I_{x_0,2\delta}$, 
$I_{x_0,2\delta}\backslash I_{x_0,\beta h}$,
$I_{x_*,d}\backslash I_{x_*,3\delta}$ 
and
$I_{x_*,3\delta}\backslash I_{x_*,\beta h}$ 
for all $E\ge E_{\beta,\delta}$.  Moreover, we 
have the inclusion $I_{x_0,2\delta}\subset I_{x_*,3\delta}$
for all $E\ge E_{\beta,\delta}$.  Analogous to before,
we consider the decomposition
\begin{equation}
\sum_{b=1}^n F_b(x_0) - F_b(x_*)
= \esF_1(x_0,x_*) + \esF_2(x_0,x_*) + \esF_3(x_0,x_*),
\label{LTwoDecompFsum} 
\end{equation}
where $F_b(x)= \zeta_b(x)H(x,x_b)W_b$ and 
\begin{align}
\esF_1(x_0,x_*) 
&= \sum_{e\in I_{x_0,2\delta}} 
                    \sum_{q=1}^Q F_q^e(x_0) - F_q^e(x_*), 
\label{LTwoFsum1} \\
\esF_2(x_0,x_*) 
&= \sum_{e\in I_{x_0,d}\backslash I_{x_0,2\delta}} 
                    \sum_{q=1}^Q F_q^e(x_0) - F_q^e(x_*),
\label{LTwoFsum2} \\
\esF_3(x_0,x_*) 
&= \sum_{e\not\in I_{x_0,d}} 
                    \sum_{q=1}^Q F_q^e(x_0) - F_q^e(x_*).
\label{LTwoFsum3} 
\end{align}
To establish the result, we show that, for every $\veps>0$,
there exists a $\delta\in(0,d/3]$ such that 
$|\esF_1|\le \veps$, $|\esF_2|\le\veps$ and 
$|\esF_3|\le \veps$ for all $0<|x_0-x_*|\le\delta$ and 
$E\ge E_{\beta,\delta}$.  Results for the finite interval 
$E\in[E_0,E_{\beta,\delta}]$ are a straightforward 
consequence of part (i) and will be omitted for brevity.

For the term $\esF_1$ in (\ref{LTwoDecompFsum}), we use 
the inclusion $I_{x_0,2\delta}\subset I_{x_*,3\delta}$,
and the inclusions $I_{x_0,\beta h}\subset I_{x_0,2\delta}$ 
and $I_{x_*,\beta h}\subset I_{x_*,3\delta}$, to obtain
\begin{equation}
\begin{split}
|\esF_1(x_0,x_*)| 
&\le \sum_{e\in I_{x_0,2\delta}}\sum_{q=1}^Q |F_q^e(x_0)| 
     +  \sum_{e\in I_{x_*,3\delta}}\sum_{q=1}^Q |F_q^e(x_*)|  \\
&= \Big(\sum_{e\in I_{x_0,\beta h}} 
    + \sum_{e\in I_{x_0,2\delta}\backslash I_{x_0,\beta h}}\Big)
                                    \sum_{q=1}^Q |F_q^e(x_0)| \\
& \qquad + \Big(\sum_{e\in I_{x_*,\beta h}} 
    + \sum_{e\in I_{x_*,3\delta}\backslash I_{x_*,\beta h}}\Big)
                                   \sum_{q=1}^Q |F_q^e(x_*)|. \\
\end{split}
\end{equation}
From the definition of $F_q^e$, and the same arguments in
part (ii) that yielded (\ref{LTwoSum1C}) and (\ref{LTwoSum2F}), 
we find
\begin{equation}
\begin{split}
|\esF_1(x_0,x_*)| 
\le 
(Ch^\mu + C\delta^\mu) 
+ (Ch^\mu + C\delta^\mu)  
\le C\delta^\mu,  
\quad \forall E\ge E_{\beta,\delta},  
\end{split}
\end{equation}
where $0<\mu\le 1$ and the last inequality follows from 
the fact that $h\le \delta/(2\beta)$.  Thus, for every 
$\veps>0$, we can choose $\delta\in(0,d/3]$ sufficiently 
small to get the uniform bound
\begin{equation}
|\esF_1(x_0,x_*)| \le \veps,
\quad 
\forall E\ge E_{\beta,\delta},  
\quad 
x_0,x_*\in\Gamma,
\quad 
0<|x_0-x_*|\le\delta.
\label{LTwoFsum1A} 
\end{equation}

For the term $\esF_2$ in (\ref{LTwoDecompFsum}), 
we notice first that 
$\Gamma_{x_0,2\delta-h}\subset\Gamma_{x_0,2\delta}$
and moreover that $\dist(\partial \Gamma_{x_0,2\delta-h},
\partial\Gamma_{x_0,2\delta})\ge h\ge\diam(\Gamma^e)$.
Hence, if $\Gamma^e\cap\Gamma_{x_0,2\delta-h}\ne\emptyset$,
then $\Gamma^e\subset \Gamma_{x_0,2\delta}$.  From this 
we deduce that $\Gamma^e\subset\Gamma_{x_0,d}$ and
$\Gamma^e\cap\Gamma_{x_0,2\delta-h}=\emptyset$ for all
$e\in I_{x_0,d}\backslash I_{x_0,2\delta}$.  This implies
\begin{equation}
{39\delta\over 20} \le 2\delta-h < |y-x_0| \le d,
\quad
\forall y\in\Gamma^e,
\;
e\in I_{x_0,d}\backslash I_{x_0,2\delta},
\;
E\ge E_{\beta,\delta}.
\label{LTwoFsum2A} 
\end{equation}
Thus, by a slight generalization of Lemma 
\ref{LemmaOne}(iii), we have
\begin{equation}
|H(x_*,y) - H(x_0,y)| 
    \le {C|x_*-x_0|\over|y - x_0|^{3-\mu}},
\quad
\forall y\in\Gamma^e,
\;
e\in I_{x_0,d}\backslash I_{x_0,2\delta},
\; 
E\ge E_{\beta,\delta}.
\label{LTwoFsum2B} 
\end{equation}
Furthermore, since 
$\diam(\supp(\smash{\widehat\zeta_b}))\le \delta/2$, 
we have
\begin{equation}
\zeta_q^e(x_0)=1,
\quad
\zeta_q^e(x_*)=1,
\quad
\forall x_q^e\in\Gamma^e,
\;
e\in I_{x_0,d}\backslash I_{x_0,2\delta},
\; 
E\ge E_{\beta,\delta}.
\label{LTwoFsum2C} 
\end{equation}
Combining (\ref{LTwoFsum2C}) and (\ref{LTwoFsum2B}) 
with (\ref{LTwoFsum2}), and using the fact that
$0<|x_0-x_*|\le\delta$, we get
\begin{equation}
\begin{split}
|\esF_2(x_0,x_*)| 
&\le \sum_{e\in I_{x_0,d}\backslash I_{x_0,2\delta}} 
      \sum_{q=1}^Q |\zeta_q^e(x_0)H(x_0,x_q^e)W_q^e 
                  - \zeta_q^e(x_*)H(x_*,x_q^e)W_q^e|, \\
&\le \sum_{e\in I_{x_0,d}\backslash I_{x_0,2\delta}} 
      \sum_{q=1}^Q |H(x_0,x_q^e) - H(x_*,x_q^e)| W_q^e, \\
&\le \sum_{e\in I_{x_0,d}\backslash I_{x_0,2\delta}} 
      \sum_{q=1}^Q {C\delta\over|x_q^e - x_0|^{3-\mu}} W_q^e. 
\end{split}
\label{LTwoFsum2D} 
\end{equation}
Using the same arguments as in 
(\ref{LTwoSum2D})--(\ref{LTwoSum2F}) with analogous
quantities $q(x_0,e)$, $\lambda(x_0,e)$ and
$\Gamma_{x_0,\rm in}^e$, we obtain
\begin{equation}
\begin{split}
|\esF_2(x_0,x_*)| 
&\le \sum_{e\in I_{x_0,d}\backslash I_{x_0,2\delta}} 
      \int_{\Gamma^e} {C\delta\over|y - x_0|^{3-\mu}} \; dA_y, \\
&\le  \int_{\Gamma_{x_0,d}\backslash\Gamma_{x_0,2\delta-h}} 
                     {C\delta\over|y - x_0|^{3-\mu}} \; dA_y, \\
&\le  
\left.
\begin{cases}
C\delta + C\delta^\mu, &\hfill 0<\mu<1 \\ 
C\delta + C\delta\ln{\delta}, &\hfill \mu=1 
\end{cases}
\right]
\quad 
\forall E\ge E_{\beta,\delta},  
\end{split}
\label{LTwoFsum2E} 
\end{equation}
where the last inequality follows from a direct estimate
of the integral using polar coordinates and the fact that
the local coordinate maps $\psi_{x_0}$ and $\psi_{x_0}^{-1}$
are Lipschitz.  Hence, for every $\veps>0$, we can choose 
$\delta\in(0,d/3]$ sufficiently small to get
the uniform bound
\begin{equation}
|\esF_2(x_0,x_*)| \le \veps,
\quad 
\forall E\ge E_{\beta,\delta},  
\quad 
x_0,x_*\in\Gamma,
\quad 
0<|x_0-x_*|\le\delta.
\label{LTwoFsum2F} 
\end{equation}

For the last term $\esF_3$ in (\ref{LTwoDecompFsum}), 
we notice similar to before that 
$\Gamma_{x_0,d-h}\subset\Gamma_{x_0,d}$
and that $\dist(\partial \Gamma_{x_0,d-h},
\partial\Gamma_{x_0,d})\ge h\ge\diam(\Gamma^e)$.
Hence, if $\Gamma^e\cap\Gamma_{x_0,d-h}\ne\emptyset$,
then $\Gamma^e\subset \Gamma_{x_0,d}$.
From this we deduce that 
$\Gamma^e\cap\Gamma_{x_0,d-h}=\emptyset$  for all 
$e\not\in I_{x_0,d}$.  This implies, using the 
fact that $h\le\delta/(2\beta)\le d/(6\beta)$,
\begin{equation}
{59 d\over 60} \le d-h < |y-x_0|, 
\quad
\forall y\in\Gamma^e,
\quad
e\not\in I_{x_0,d},
\quad 
E\ge E_{\beta,\delta}.
\label{LTwoFsum3A}
\end{equation}
Moreover, since $0<|x_0-x_*|\le\delta\le d/3$, 
we deduce from the triangle inequality that
\begin{equation}
{39 d\over 60} \le |y-x_*|, 
\quad
\forall y\in\Gamma^e,
\quad
e\not\in I_{x_0,d},
\quad 
E\ge E_{\beta,\delta}.
\label{LTwoFsum3B}
\end{equation}
Furthermore, since 
$\diam(\supp(\smash{\widehat\zeta_b}))\le \delta/2\le d/6$, 
we have
\begin{equation}
\zeta_q^e(x_0)=1,
\quad
\zeta_q^e(x_*)=1,
\quad
\forall x_q^e\in\Gamma^e,
\;
e\not\in I_{x_0,d},
\quad 
E\ge E_{\beta,\delta}.
\label{LTwoFsum3C}
\end{equation}
Combining (\ref{LTwoFsum3C}) with (\ref{LTwoFsum3}), 
and using the fact that $H(x,y)$ is Lipschitz on the
set $|y-x|\ge 39d/60$ by (A2), and the fact that 
$0<|x_0-x_*|\le\delta$, we get
\begin{equation}
\begin{split}
|\esF_3(x_0,x_*)| 
&\le \sum_{e\not\in I_{x_0,d}} 
      \sum_{q=1}^Q |\zeta_q^e(x_0)H(x_0,x_q^e)W_q^e 
                  - \zeta_q^e(x_*)H(x_*,x_q^e)W_q^e|, \\
&\le \sum_{e\not\in I_{x_0,d}} 
      \sum_{q=1}^Q |H(x_0,x_q^e) - H(x_*,x_q^e)| W_q^e, \\
&\le \sum_{e\not\in I_{x_0,d}} 
      \sum_{q=1}^Q C|x_0 - x_*| W_q^e, \\ 
&\le  C|\Gamma| |x_0 - x_*| \le C\delta.
\end{split}
\label{LTwoFsum3D} 
\end{equation}
Thus, for every $\veps>0$, we can choose $\delta\in(0,d/3]$ 
sufficiently small to get the uniform bound
\begin{equation}
|\esF_3(x_0,x_*)| \le \veps,
\quad 
\forall E\ge E_{\beta,\delta},  
\quad 
x_0,x_*\in\Gamma,
\quad 
0<|x_0-x_*|\le\delta.
\label{LTwoFsum3E} 
\end{equation}
The desired result stated in part (iii) follows from 
(\ref{LTwoFsum3E}), (\ref{LTwoFsum2F}), (\ref{LTwoFsum1A}) 
and (\ref{LTwoDecompFsum}).
\end{proof}

Our next result establishes some important properties
of the local polynomial $R_x(z)$ of degree $p\ge 0$
with coefficients $\{C_{x,0},C_{x,\alpha_1},\ldots,
C_{x,\alpha_1\alpha_2\cdots\alpha_p}\}$ defined
in (\ref{LocPolyDefn})--(\ref{LocMomCondCoeffThree}).

\begin{lemma}
\label{LemmaThree}
Let $\Gamma$ satisfy {\rm (A0)} with Lyapunov
radius $d>0$, $H$ satisfy {\rm (A2)} with
exponent $0 < \mu \le 1$, and
$\{\smash{\Gamma^e}, \smash{x_q^e}, 
\smash{W_q^e}, \smash{\zeta_q^e}, 
\smash{\widehat\zeta_q^e}\}$
satisfy {\rm (A4)}--{\rm (A8)}.  Then:
\begin{itemize}
\item[{\rm (i)}] For each $x\in\Gamma$ and $n\ge n_0$, 
the local polynomial $R_x:\Gamma_{x,d}\to\Rkk$ is 
invariant to the choice of orthonormal basis in 
$T_x\Gamma$.
\item[{\rm (ii)}] $R_x(z)$ is continuous in $x\in\Gamma$
and $z\in\Gamma_{x,d'}$ for each $n\ge n_0$ and 
$d'\in(0,d)$.
\item[{\rm (iii)}] For every $\veps>0$ there exists $N>0$
such that $|R_x(z)|\le\veps$ for all $x\in\Gamma$, 
$z\in\Gamma_{x,d}$ and $n\ge N$.
\end{itemize}
\end{lemma}

\begin{proof}
For part (i), let $x\in\Gamma$ and $n\ge n_0$ be arbitrary, 
and let $\xi_x(z)$ and ${\widetilde\xi}_x(z)$ be local 
Cartesian coordinates in $\Gamma_{x,d}$ relative to two 
orthonormal bases in $T_x\Gamma$.  Then there exists
an orthogonal matrix $L_x\in\Reals^{2\times 2}$ such
that 
\begin{equation}
{\widetilde\xi}_{x,\alpha}(z) 
          = L_{x,\alpha\beta}\xi_{x,\beta}(z),
\quad
\forall z\in\Gamma_{x,d}. 
\label{LemmaThree1}
\end{equation}
Moreover, let 
$\{C_{x,0},C_{x,\alpha_1},\ldots,
C_{x,\alpha_1\alpha_2\cdots\alpha_p}\}$ and
$\{{\widetilde C}_{x,0},{\widetilde C}_{x,\alpha_1},
\ldots,{\widetilde C}_{x,\alpha_1\alpha_2\cdots\alpha_p}\}$ 
denote polynomial coefficients found using the two bases.  
From (\ref{LocMomCondMat})--(\ref{LocMomCondCoeffThree}) 
and (\ref{LemmaThree1}), and the uniqueness of solutions 
of (\ref{LocMomCondMat}) guaranteed by (A8), and the fact 
that $L_x^{-1}=L_x^T$, we deduce 
\begin{equation}
C_{x,0}={\widetilde C}_{x,0},
\quad
C_{x,\alpha_1\cdots\alpha_s}
=
L_{x,\beta_1\alpha_1}\cdots L_{x,\beta_s\alpha_s}
       {\widetilde C}_{x,\beta_1\cdots\beta_s},
\quad
s=1,\ldots,p.
\label{LemmaThree2}
\end{equation}
Combining (\ref{LemmaThree2}) and (\ref{LemmaThree1})
with (\ref{LocPolyDefn}), we find that the value of
$R_x(z)$ is independent of the basis.

For part (ii), let $n\ge n_0$, $d'\in(0,d)$, 
$x_0\in\Gamma$ and $z_0\in\Gamma_{x_0,d'}$ be 
arbitrary, and consider any $0<\delta<(d-d')/2$.
Then, by the triangle inequality, we find
\begin{equation}
z\in\Gamma_{x,d},
\quad
\forall x\in\Gamma_{x_0,\delta},
\quad
\forall z\in\Gamma_{z_0,\delta}.
\end{equation}
Moreover, let $u_0\in T_{x_0}\Gamma$ be any given 
unit vector, let $\nu(x)$ denote the outward unit 
normal to $\Gamma$ at $x$, and consider a basis 
for $T_x\Gamma$ defined by
\begin{equation}
t_1(x) = {u_0-(u_0\cdot\nu(x))\nu(x)\over
           |u_0-(u_0\cdot\nu(x))\nu(x)|},
\quad
t_2(x) = \nu(x) \times t_1(x).
\end{equation}
Then, for any $x\in\Gamma_{x_0,\delta}$ and
$z\in\Gamma_{z_0,\delta}$, the local Cartesian 
coordinates for $z\in\Gamma_{x,d}$ are given by
\begin{equation}
\xi_{x,\alpha}(z) = (z-x)\cdot t_\alpha(x).
\end{equation}
From the fact that these coordinates depend 
continuously on $x\in\Gamma_{x_0,\delta}$ and 
$z\in\Gamma_{z_0,\delta}$, together with 
Lemmas \ref{LemmaOne} and \ref{LemmaTwo}, 
we deduce that the coefficient matrix and data 
vector for the linear system in (\ref{LocMomCondMat}) 
depend continuously on $x\in\Gamma_{x_0,\delta}$.
Furthermore, by (A8), the unique solution 
$\{C_{x,0},C_{x,\alpha_1},\ldots,
C_{x,\alpha_1\alpha_2\cdots\alpha_p}\}$ 
of the system also depends continuously 
on $x\in\Gamma_{x_0,\delta}$.  Combining these 
results with (\ref{LocPolyDefn}), we find that 
$R_x(z)$ depends continuously on 
$x\in\Gamma_{x_0,\delta}$ and
$z\in\Gamma_{z_0,\delta}$ as required.

For part (iii), we notice that, by (A8) and the
fact that $|\xi_x(z)|\le d$ for all $x\in\Gamma$
and $z\in\Gamma_{x,d}$, it will be sufficient 
to show that, for every $\veps>0$, there exists 
an $N>0$ such that
\begin{equation}
|\Delta_{x}^{0}|\le \veps,
\quad 
|\Delta_{x}^{\beta_1\cdots\beta_s}|\le\veps,
\quad
\forall n\ge N,
\quad
x\in\Gamma,
\quad
s=1,\ldots,p. 
\end{equation}
For brevity, we establish the above bound for
$\smash{\Delta_{x}^{0}}$ only, and note that a bound for 
$\smash{\Delta_{x}^{\beta_1\cdots\beta_s}}$ follows by a 
similar argument.  Using the same notation as in the proof 
of Lemma \ref{LemmaTwo}, let $\delta\in(0,d/3]$, $\beta\ge 10$ 
and $E_{\beta,\delta}\ge E_0$ be given numbers, where
$E_{\beta,\delta}$ is sufficiently large such that
$0<\beta h \le \delta/2$ and 
$\diam(\supp(\smash{\widehat\zeta_b}))\le \delta/2$ 
for all $b=1,\ldots,n$ and $E\ge E_{\beta,\delta}$.
Moreover, for arbitrary $x=x_0\in\Gamma$, we consider 
as before the collections of surface elements 
$\Gamma^e$, $e=1,\ldots,E$, denoted by $I_{x_0,\beta h}$ 
and $I_{x_0,\delta}$.  From (\ref{LocMomCondCoeffOne}), 
we consider the decomposition
\begin{align}
\Delta_{x_0}^{0} 
&= \int_{\Gamma} H(x_0,y)\eta_{x_0}(y)\;dA_y
   - \sum_{b=1}^{n} \zeta_b(x_0)H(x_0,x_b)\eta_{x_0}(x_b)W_b, \\ 
&= \sum_{e=1}^{E}\left[\int_{\Gamma^e} H(x_0,y)\eta_{x_0}(y)\;dA_y
   - \sum_{q=1}^{Q} 
      \zeta_q^e(x_0)H(x_0,x_q^e)\eta_{x_0}(x_q^e)W_q^e\right], \\ 
&= \esD_1(x_0) + \esD_2(x_0) 
                                     \vphantom{\sum_{e=1}^{E}},
\label{LThreeDecomp} 
\end{align}
where 
\begin{align}
\esD_1(x_0) 
&= \sum_{e\in I_{x_0,\delta}} \left[
     \int_{\Gamma^e} H(x_0,y)\eta_{x_0}(y)\;dA_y
       - \sum_{q=1}^Q \zeta_q^e(x_0)H(x_0,x_q^e)\eta_{x_0}(x_q^e)W_q^e
                        \right], 
\label{LThreeSum1} \\
\esD_2(x_0) 
&= \sum_{e\not\in I_{x_0,\delta}} \left[
     \int_{\Gamma^e} H(x_0,y)\eta_{x_0}(y)\;dA_y
       - \sum_{q=1}^Q \zeta_q^e(x_0)H(x_0,x_q^e)\eta_{x_0}(x_q^e)W_q^e
                        \right].
\label{LThreeSum2} 
\end{align}
To establish the result, we show that, for every $\veps>0$, 
we can choose $E_{\beta,\delta}\ge E_0$ sufficiently large
such that $|\esD_1|\le \veps$ and $|\esD_2|\le\veps$ for all 
$E\ge E_{\beta,\delta}$ and $x_0\in\Gamma$.  

For the term $\esD_1$ in (\ref{LThreeDecomp}), we use 
the fact that $|\eta_{x_0}|\le 1$ to obtain
\begin{equation}
|\esD_1(x_0)| 
\le \sum_{e\in I_{x_0,\delta}} \left[ 
   \int_{\Gamma^e} |H(x_0,y)|\;dA_y 
    +  \sum_{q=1}^Q |\zeta_q^e(x_0)H(x_0,x_q^e)W_q^e|\right].  
\label{LThreeSum1A} 
\end{equation}
Working with the first sum in (\ref{LThreeSum1A}), we
have, using (A2),
\begin{equation}
\begin{split}
\sum_{e\in I_{x_0,\delta}}\int_{\Gamma^e} |H(x_0,y)|\;dA_y 
\le \int_{\Gamma_{x_0,\delta}} {C\over|y-x_0|^{2-\mu}}\;dA_y 
\le C\delta^\mu, 
\quad
\forall E\ge E_{\beta,\delta},
\end{split}
\label{LThreeSum1B} 
\end{equation}
where $0<\mu\le 1$ and the last inequality follows from a 
direct estimate of the integral using polar coordinates and 
the fact that the local coordinate maps $\psi_{x_0}$ and 
$\psi_{x_0}^{-1}$ are Lipschitz.  Working with the second 
sum in (\ref{LThreeSum1A}), we have, using the inclusion 
$I_{x_0,\beta h}\subset I_{x_0,\delta}$,
\begin{equation}
\begin{split}
&\sum_{e\in I_{x_0,\delta}} 
   \sum_{q=1}^Q |\zeta_q^e(x_0)H(x_0,x_q^e)W_q^e| \\
&\hskip0.5in
    = \Big(\sum_{e\in I_{x_0,\beta h}} 
    + \sum_{e\in I_{x_0,\delta}\backslash I_{x_0,\beta h}}\Big)
      \sum_{q=1}^Q |\zeta_q^e(x_0)H(x_0,x_q^e)W_q^e|.  
\end{split}
\label{LThreeSum1C} 
\end{equation}
By the same arguments used in the proof of Lemma
\ref{LemmaTwo}(ii) that yielded (\ref{LTwoSum1C}) 
and (\ref{LTwoSum2F}), we find
\begin{equation}
\begin{split}
\sum_{e\in I_{x_0,\delta}} 
 \sum_{q=1}^Q |\zeta_q^e(x_0)H(x_0,x_q^e)W_q^e| 
\le (Ch^\mu + C\delta^\mu)
\le C\delta^\mu,
\quad
\forall E\ge E_{\beta,\delta},
\end{split}
\label{LThreeSum1D} 
\end{equation}
where the last inequality follows from the fact 
that $h\le \delta/(2\beta)$.  Thus, in view of 
(\ref{LThreeSum1D}), (\ref{LThreeSum1B}) and 
(\ref{LThreeSum1A}), for every $\veps>0$, we can 
choose a $\delta\in(0,d/3]$ sufficiently small 
to get the uniform bound
\begin{equation}
|\esD_1(x_0)| \le \veps,
\quad 
\forall E\ge E_{\beta,\delta},  
\quad 
x_0\in\Gamma.
\label{LThreeSum1E} 
\end{equation}

For the term $\esD_2$ in (\ref{LThreeDecomp}), we
notice that $\Gamma_{x_0,\delta-h}\subset\Gamma_{x_0,\delta}$
and moreover that $\dist(\partial \Gamma_{x_0,\delta-h},
\partial\Gamma_{x_0,\delta})\ge h\ge\diam(\Gamma^e)$.
Hence, if $\Gamma^e\cap\Gamma_{x_0,\delta-h}\ne\emptyset$,
then $\Gamma^e\subset \Gamma_{x_0,\delta}$.
From this we deduce that 
$\Gamma^e\cap\Gamma_{x_0,\delta-h}=\emptyset$  for all 
$e\not\in I_{x_0,\delta}$.  This implies, using the 
fact that $h\le\delta/(2\beta)$,
\begin{equation}
{19 \delta\over 20} \le \delta-h < |y-x_0|, 
\quad
\forall y\in\Gamma^e,
\quad
e\not\in I_{x_0,\delta},
\quad 
E\ge E_{\beta,\delta}.
\label{LThreeSum2A}
\end{equation}
Furthermore, since 
$\diam(\supp(\smash{\widehat\zeta_b}))\le \delta/2$, 
we have
\begin{equation}
\zeta_q^e(x_0)=1,
\quad
\forall x_q^e\in\Gamma^e,
\;
e\not\in I_{x_0,\delta},
\quad 
E\ge E_{\beta,\delta}.
\label{LThreeSum2B}
\end{equation}
Combining (\ref{LThreeSum2B}) with (\ref{LThreeSum2}), 
and using the fact that $H(x,y)$ is class $C^{m,1}$ on 
the set $|y-x|\ge 19\delta/20$ by (A2), and that 
$\eta_x(y)$ is class $C^{m,1}$ for all $x$ and $y$, 
together with the quadrature error bound in (A6) for a
rule of order $\ell$, we get
\begin{align}
|\esD_2(x_0)| 
&\le \sum_{e\not\in I_{x_0,\delta}} \left|
     \int_{\Gamma^e} H(x_0,y)\eta_{x_0}(y)\;dA_y
       - \sum_{q=1}^Q H(x_0,x_q^e)\eta_{x_0}(x_q^e)W_q^e
                        \right| \\
&\le \sum_{e\not\in I_{x_0,\delta}} 
                C_\delta|\Gamma^e| h^{\min(\ell,m+1)}
\le C_\delta|\Gamma| h^{\min(\ell,m+1)},
\quad
\forall E\ge E_{\beta,\delta}.
\label{LThreeSum2C} 
\end{align}
Here $C_\delta$ is a Lipschitz constant for the derivatives
of the function $H(x,y)\eta_{x}(y)$ on the set 
$|y-x|\ge 19\delta/20$, and $\delta\in(0,d/3]$ is fixed such 
that (\ref{LThreeSum1E}) holds.  Thus, for any given $\veps>0$, 
we can choose $E_{\beta,\delta}$ sufficiently large 
(equivalently $h$ sufficiently small) to get the uniform bound
\begin{equation}
|\esD_2(x_0)| \le \veps,
\quad 
\forall E\ge E_{\beta,\delta},  
\quad 
x_0\in\Gamma.
\label{LThreeSum2D} 
\end{equation}
The desired result for $\smash{\Delta_{x_0}^{0}}$ follows
from (\ref{LThreeSum2D}), (\ref{LThreeSum1E}) and
(\ref{LThreeDecomp}).
\end{proof}

The next result shows that the linear operators 
$\esA$ and $\esA_n$ ($n\ge n_0$) defined in 
(\ref{DefAop}) and (\ref{IntEqnApprox}) satisfy 
the collective compactness conditions (C1)--(C3).

\begin{lemma}
\label{LemmaFour}
Let $\Gamma$ satisfy {\rm (A0)} with Lyapunov radius $d>0$,
$G$ satisfy {\rm (A1)}, 
$H$ satisfy {\rm (A2)} with exponent $0<\mu\le 1$,
and
$\{\smash{\Gamma^e}, \smash{x_q^e}, 
\smash{W_q^e}, \smash{\zeta_q^e}, 
\smash{\widehat\zeta_q^e}\}$
satisfy {\rm (A4)}--{\rm (A8)}.  
Then:
\begin{itemize}
\item[{\rm (i)}] $\esA$ is a compact operator 
on $C^0(\Gamma,\Rk)$.
\item[{\rm (ii)}] $\esA_n$ is a finite-rank
(hence compact) operator on $C^0(\Gamma,\Rk)$ for 
each $n\ge n_0$.
\item[{\rm (iii)}] $\esA_n$ satisfies 
{\rm (C1)}--{\rm (C3)} and hence is collectively 
compact on $C^0(\Gamma,\Rk)$.
\end{itemize}
\end{lemma}

\begin{proof}
The result for $\esA=\esG+\esH$ in part (i) is classic and 
follows from the continuity of $G$ and the properties of $H$
established in Lemma \ref{LemmaOne}, see for example 
\cite{Kress:1989,Mikhlin:1960}.  
The result for $\esA_n=\esG_n+\esH_n$ in part (ii) is 
analogous and relies on Lemma \ref{LemmaTwo}.  For brevity, 
we illustrate only (iii).  Moreover, we show the result 
only for $\esH_n$.  The result for $\esG_n$ is similar
and more straightforward due to the continuity of $G$.

To establish (C1), let $v\in C^0(\Gamma,\Rk)$ be arbitrary.  
We seek to show that $|(\esH v)(x_0) - (\esH_n v)(x_0)|\to 0$ 
as $n\to\infty$ uniformly in $x_0\in\Gamma$, where
\begin{equation}
\begin{split}
&(\esH v)(x_0) - (\esH_n v)(x_0) \\
&\quad 
  = \int_{\Gamma} H(x_0,y)v(y)\;dA_y
           - \sum_{b=1}^{n} \zeta_b(x_0)H(x_0,x_b)v(x_b)W_b \\
&\hskip2.5in
  - \sum_{b=1}^{n} \widehat\zeta_b(x_0)R_{x_0}(x_b)v(x_b). 
\end{split}
\end{equation}
Using the same notation as in the proof of Lemma 
\ref{LemmaTwo}, let $\delta\in(0,d/3]$, $\beta\ge 10$ 
and $E_{\beta,\delta}\ge E_0$ be given numbers, where 
$E_{\beta,\delta}$ is sufficiently large such that
$0<\beta h \le \delta/2$ and
$\diam(\supp(\smash{\widehat\zeta_b}))\le \delta/2$ for 
all $b=1,\ldots,n$ and $E\ge E_{\beta,\delta}$.  Moreover, 
for any given $x_0\in\Gamma$, consider as before the 
collections of surface elements $\Gamma^e$, $e=1,\ldots,E$, 
denoted by $I_{x_0,\beta h}$ and $I_{x_0,\delta}$, and 
consider the decomposition
\begin{equation}
(\esH v)(x_0) - (\esH_n v)(x_0)
                 = \esK_1(x_0) + \esK_2(x_0) - \esK_3(x_0),
\label{LFourDecomp} 
\end{equation}
where
\begin{align}
\esK_1(x_0) 
&= \sum_{e\in I_{x_0,\delta}} \left[
     \int_{\Gamma^e} H(x_0,y) v(y)\;dA_y
       - \sum_{q=1}^Q \zeta_q^e(x_0)H(x_0,x_q^e) v(x_q^e)W_q^e
                        \right], 
\label{LFourSum1} \\
\esK_2(x_0) 
&= \sum_{e\not\in I_{x_0,\delta}} \left[
     \int_{\Gamma^e} H(x_0,y) v(y)\;dA_y
       - \sum_{q=1}^Q \zeta_q^e(x_0)H(x_0,x_q^e) v(x_q^e)W_q^e
                        \right], 
\label{LFourSum2} \\
\esK_3(x_0) 
&= \sum_{b=1}^{n}
      \widehat\zeta_b(x_0)
                R_{x_0}(x_b) v(x_b).  
\label{LFourSum3} 
\end{align}
To establish the result, we show that, for every $\veps>0$,
we can choose $E_{\beta,\delta}\ge E_0$ sufficiently large
such that $|\esK_1|\le \veps$, $|\esK_2|\le \veps$ and 
$|\esK_3|\le\veps$ for all $E\ge E_{\beta,\delta}$ and 
$x_0\in\Gamma$.

For the term $\esK_1$ in (\ref{LFourDecomp}), the same
arguments used to establish 
(\ref{LThreeSum1A})--(\ref{LThreeSum1D}) in the proof
of Lemma \ref{LemmaThree} can be applied to obtain
\begin{equation}
\begin{split}
|\esK_1(x_0)| 
&\le \sum_{e\in I_{x_0,\delta}} ||v||\left[ 
   \int_{\Gamma^e}  |H(x_0,y)|\;dA_y 
    +  \sum_{q=1}^Q  |\zeta_q^e(x_0)H(x_0,x_q^e)W_q^e|\right],  \\ 
& \le C ||v|| \delta^\mu,
\quad
\forall E\ge E_{\beta,\delta}.
\end{split}
\label{LFourSum1A} 
\end{equation}
Thus, for every $\veps>0$, we can choose a $\delta\in(0,d/3]$ 
sufficiently small to get the uniform bound
\begin{equation}
|\esK_1(x_0)| \le \veps,
\quad 
\forall E\ge E_{\beta,\delta},  
\quad 
x_0\in\Gamma.
\label{LFourSum1B} 
\end{equation}

For the term $\esK_2$ in (\ref{LFourDecomp}), the same
arguments used to establish 
(\ref{LThreeSum2A})--(\ref{LThreeSum2C}) in the proof
of Lemma \ref{LemmaThree} can be applied to obtain
\begin{align}
|\esK_2(x_0)| 
&\le \sum_{e\not\in I_{x_0,\delta}} \left|
     \int_{\Gamma^e} H(x_0,y) v(y)\;dA_y
       - \sum_{q=1}^Q H(x_0,x_q^e) v(x_q^e)W_q^e
                        \right| \\
&\le \sum_{e\not\in I_{x_0,\delta}} 
                     |\Gamma^e| \tau(e,f_{x_0},h)
\le |\Gamma| \max_{e\not\in I_{x_0,\delta}}\tau(e,f_{x_0},h),
\quad
\forall E\ge E_{\beta,\delta}.
\label{LFourSum2A} 
\end{align}
Here $f_{x_0}(y) = H(x_0,y) v(y)$ is the function being
integrated, $\tau(e,f_{x_0},h)$ is the quadrature 
truncation error for this function as given
in (A6), $e\not\in I_{x_0,\delta}$ implies 
$|y-x_0|\ge 19\delta/20$, and $\delta\in(0,d/3]$ is fixed 
such that (\ref{LFourSum1B}) holds.  Since $H(x_0,y)$ is 
class $C^{m,1}$ in $x_0$ and $y$ by (A2) on the set 
$|y-x_0|\ge 19\delta/20$, it follows that the continuity 
properties of $f_{x_0}(y)$ in $y$ are uniform in $x_0$.  
Hence, by (A6), 
$\max_{e\not\in I_{x_0,\delta}}\tau(e,f_{x_0},h)\to 0$
as $h\to 0$ uniformly in $x_0$.  From this we deduce
that, for any given $\veps>0$, we can choose 
$E_{\beta,\delta}$ sufficiently large (equivalently 
$h$ sufficiently small) to get the uniform bound
\begin{equation}
|\esK_2(x_0)| \le \veps,
\quad 
\forall E\ge E_{\beta,\delta},  
\quad 
x_0\in\Gamma.
\label{LFourSum2B} 
\end{equation}

For the term $\esK_3$ in (\ref{LFourDecomp}), we consider the 
index set $J_{x_0}=\{b\;|\;\smash{\widehat\zeta_b(x_0)}>0\}$.
By (A7), we have $\diam(\supp(\smash{\widehat\zeta_b}))\le Ch$, 
and by (A5), we have $|x_a- x_b|\ge Ch$ for any $a\ne b$.  From 
this we deduce that the number of elements in $J_{x_0}$, 
denoted by $|J_{x_0}|$, must be bounded uniformly in $x_0$, 
namely
\begin{equation}
|J_{x_0}| \le C,
\quad
\forall E\ge E_{\beta,\delta},
\quad 
x_0\in\Gamma.
\label{LFourSum3A}
\end{equation}
From (\ref{LFourSum3}) we get, using the definition of $J_{x_0}$ 
and the fact that $0\le\smash{\widehat\zeta_b}\le 1$,
\begin{equation}
|\esK_3(x_0)| 
\le \sum_{b=1}^{n}|\widehat\zeta_b(x_0) R_{x_0}(x_b) v(x_b)| 
\le ||v|| \sum_{b\in J_{x_0}}|R_{x_0}(x_b)|.  
\label{LFourSum3B}
\end{equation}
Using (\ref{LFourSum3B}) and (\ref{LFourSum3A}), together
with Lemma \ref{LemmaThree}(iii), we find that, for any 
given $\veps>0$, we can choose $E_{\beta,\delta}$ 
sufficiently large to get the uniform bound
\begin{equation}
|\esK_3(x_0)| \le \veps,
\quad 
\forall E\ge E_{\beta,\delta},  
\quad 
x_0\in\Gamma.
\label{LFourSum3C} 
\end{equation}
The desired result follows from (\ref{LFourSum3C}), 
(\ref{LFourSum2B}), (\ref{LFourSum1B}) and (\ref{LFourDecomp}).

To establish (C2), we seek to show that $(\esH_n v)(x_0)$ 
is uniformly bounded for all $n\ge n_0$, $x_0\in\Gamma$ 
and $v\in C^0(\Gamma,\Rk)$ with $||v||\le 1$.  By definition
of $\esH_n$, we have  
\begin{equation}
(\esH_n v)(x_0) 
 = \sum_{b=1}^{n}\left[ \zeta_b(x_0)H(x_0,x_b)v(x_b)W_b 
      + \widehat\zeta_b(x_0)R_{x_0}(x_b)v(x_b)\right], 
\end{equation}
and from the definition of $J_{x_0}$ given above, and the 
fact that $0\le\smash{\widehat\zeta_b}\le 1$, we get
\begin{equation}
|(\esH_n v)(x_0)| 
\le ||v|| \sum_{b=1}^{n} |\zeta_b(x_0)H(x_0,x_b)W_b| 
      + ||v|| \sum_{b\in J_{x_0}}|R_{x_0}(x_b)|. 
\end{equation}
From this we deduce, using Lemma \ref{LemmaTwo}(ii) 
and \ref{LemmaThree}(iii), the uniform bound
\begin{equation}
|(\esH_n v)(x_0)| 
\le C,
\quad 
\forall n\ge n_0,  
\quad 
x_0\in\Gamma,
\quad 
v\in C^0(\Gamma,\Rk),
\quad 
||v||\le 1.
\end{equation}

To establish (C3), we seek to show that $(\esH_n v)(x_0)$
is uniformly equicontinuous for all $n\ge n_0$, 
$x_0\in\Gamma$ and $v\in C^0(\Gamma,\Rk)$ with $||v||\le 1$.  
For arbitrary $x_0,x_*\in\Gamma$, we have, by definition 
of $\esH_n$,
\begin{equation}
\begin{split}
&(\esH_n v)(x_*) - (\esH_n v)(x_0) \\
&\quad 
  = \sum_{b=1}^{n} \zeta_b(x_*)H(x_*,x_b)v(x_b)W_b 
           - \sum_{b=1}^{n} \zeta_b(x_0)H(x_0,x_b)v(x_b)W_b \\
&\hskip1.0in
  + \sum_{b=1}^{n} \widehat\zeta_b(x_*)R_{x_*}(x_b)v(x_b) 
  - \sum_{b=1}^{n} \widehat\zeta_b(x_0)R_{x_0}(x_b)v(x_b). 
\end{split}
\end{equation}
Similar to before, we consider the decomposition
\begin{equation}
(\esH_n v)(x_*) - (\esH_n v)(x_0)
                 = \esQ(x_0,x_*) + \esR(x_*) - \esR(x_0),
\label{LFourDecompDiff} 
\end{equation}
where
\begin{gather}
\esQ(x_0,x_*) 
= \sum_{b=1}^{n} \left[ 
        \zeta_b(x_*)H(x_*,x_b)W_b
                     - \zeta_b(x_0)H(x_0,x_b)W_b
                                          \right] v(x_b), 
\label{LFourSumDiff1} \\
\esR(x_*) 
= \sum_{b=1}^{n} 
        \widehat\zeta_b(x_*) R_{x_*}(x_b) v(x_b),  
\quad
\esR(x_0) 
= \sum_{b=1}^{n} 
        \widehat\zeta_b(x_0) R_{x_0}(x_b) v(x_b).  
\label{LFourSumDiff2} 
\end{gather}
To establish the result, we show that, for every $\veps>0$,
we can choose a $\delta\in(0,d/3]$ such that $|\esQ|\le\veps$ 
and $|\esR|\le\veps$ for all $x_0,x_*\in\Gamma$ with 
$|x_0-x_*|<\delta$, $v\in C^0(\Gamma,\Rk)$ with $||v||\le 1$, 
and $n\ge n_0$. Notice that this result immediately follows
from Lemma \ref{LemmaTwo}(iii), Lemma \ref{LemmaThree}(iii) 
and (\ref{LFourSum3A}), after increasing the value of $n_0$
if necessary, since
\begin{gather}
|\esQ(x_0,x_*)| 
\le ||v||\sum_{b=1}^{n} |\zeta_b(x_*)H(x_*,x_b)W_b
      - \zeta_b(x_0)H(x_0,x_b)W_b|, 
\label{LFourSumDiff1A} \\
|\esR(x_*)| \le ||v|| \sum_{b\in J_{x_*}}|R_{x_*}(x_b)|,  
\quad
|\esR(x_0)| \le ||v|| \sum_{b\in J_{x_0}}|R_{x_0}(x_b)|.  
\label{LFourSumDiff2A} 
\end{gather}
\end{proof}

The next result will be crucial in establishing the linear 
rate of convergence of the lowest-order method with $p=0$.
We employ the notation from condition (A3).  Moreover, for 
any given $x_0\in\Gamma$ and $\varphi\in C^{1,1}(\Gamma,\Rk)$, 
we introduce the function $U_{x_0,\varphi}:\Gamma_{x_0,d}\to\Rk$
defined as
\begin{equation}
U_{x_0,\varphi}(y) = 
\begin{cases}
H(x_0,y)[\varphi(y)-\varphi(x_0)]
 - u_{x_0}^{\rm polar}(0,\hxi(y)) D\varphi(x_0)T_{x_0}(y) , 
&y\ne x_0, \\
0,
&y=x_0. \\
\end{cases}
\label{Ufun}
\end{equation}
Here $T_{x_0}(y)\in T_{x_0}\Gamma$ is the unit vector 
defined from the orthogonal projection of the chord $y-x_0$ 
onto $T_{x_0}\Gamma$, and $\smash{\hxi(y)}\in S$ is the 
angular polar coordinate associated with $y\ne x_0$.  
Specifically, we consider the functions
$T_{x_0}:\Gamma_{x_0,d}\backslash\{x_0\}\to\Rt$
and
$T_{x_0}^{\Delta}:\Gamma_{x_0,d}\to\Rt$ defined as
\begin{equation}
T_{x_0}(y) 
= {\tstyle{(y-x_0) - ((y-x_0)\cdot\nu_0)\nu_0 \over 
   |(y-x_0) - ((y-x_0)\cdot\nu_0)\nu_0|}}, 
\quad
T_{x_0}^{\Delta}(y) = 
\begin{cases}
{(y-x_0)\over|y-x_0|} - T_{x_0}(y) , 
&y\ne x_0, \\
0,
&y=x_0, \\
\end{cases}
\label{Tfun}
\end{equation}
where $\nu_0=\nu(x_0)$ is the outward unit normal to 
$\Gamma$ at $x_0$.  We remark that the result in part 
(ii)(b) below, and its discrete analog in part (ii)(c), 
are reminiscent of the classic Tricomi condition that 
arises in the study of singular integral operators 
\cite{Tricomi:1928}. 

\begin{lemma}
\label{LemmaFive}
Let $\Gamma$ satisfy {\rm (A0)} with Lyapunov
radius $d>0$ and regularity index $m\ge 1$, 
$H$ satisfy {\rm (A2)} and {\rm(A3)} with 
exponent $\mu=1$, and 
$\{\smash{\Gamma^e},\smash{x_q^e},\smash{W_q^e}\}$ 
satisfy {\rm (A4)}--{\rm (A6)}.  Then:
\begin{itemize}
\item[{\rm (i)}] For every $x_0\in\Gamma$ and 
$\varphi\in C^{1,1}(\Gamma,\Rk)$ we have 
$U_{x_0,\varphi}\in C^{0,1}(\Gamma_{x_0,d},\Rk)$
with Lipschitz constant uniform in $x_0$.  Hence
$|U_{x_0,\varphi}(y)|\le C_\varphi |y-x_0|$ for all 
$x_0\in\Gamma$ and $y\in \Gamma_{x_0,d}$.
\item[{\rm (ii)}] For every $x_0\in\Gamma$ there 
exists a patch $\Gamma_{x_0}^*\subset\Gamma$ such that
\begin{itemize}
\item[{\rm (a)}] 
$\Gamma_{x_0,d/C}\subset\Gamma_{x_0}^*\subset\Gamma_{x_0,d}$,
\item[{\rm (b)}] 
$\int_{\Gamma_{x_0}^*} u_{x_0}^{\rm polar}(0,\hxi(y)) 
D\varphi(x_0)T_{x_0}(y) \; dA_y = 0$ for each
$\varphi\in C^{1,1}(\Gamma,\Rk)$,
\item[{\rm (c)}] 
$|\sum_{e\in I_{x_0}^*} \sum'_{1\le q\le Q} 
u_{x_0}^{\rm polar}(0,\hxi(x_q^e)) 
D\varphi(x_0)T_{x_0}(x_q^e) W_q^e| \le C_\varphi h$
for all $n\ge n_0$ and each $\varphi\in C^{1,1}(\Gamma,\Rk)$.
\end{itemize}
\end{itemize}
In the above, $C$ denotes a fixed constant, $C_\varphi$ denotes 
a constant depending on $\varphi$, the sum with a prime is over 
those $q$ such that $x_q^e\ne x_0$, and 
$I_{x_0}^*=\{e\;|\; \Gamma^e\subset \Gamma_{x_0}^*\}$. 
\end{lemma}

\begin{proof}
For part (i), let $x_0\in\Gamma$ and 
$\varphi\in C^{1,1}(\Gamma,\Rk)$ be arbitrary.  By working 
with local Cartesian coordinates in the Lyapunov patch 
$\Gamma_{x_0,d}$ and using Taylor's Theorem for Lipschitz 
functions, we find
\begin{equation}
\varphi(y) = \varphi(x_0) + D\varphi(x_0)[y-x_0] + 
\esR_{x_0}(y), 
\quad
\forall y\in\Gamma_{x_0,d},
\label{LemmaFive1}
\end{equation}
where 
\begin{equation}
\begin{gathered}
|\esR_{x_0}(y)|\le C_\varphi |y-x_0|^2, \\
\esR_{x_0}\in C^{1,1}(\Gamma_{x_0,d},\Rk),
\quad
{\esR_{x_0}\over |y-x_0|}\in C^{0,1}(\Gamma_{x_0,d},\Rk).
\end{gathered}
\label{LemmaFive1b}
\end{equation}
Moreover, since $\varphi$ is globally Lipschitz on 
$\Gamma$, the Lipschitz constants for $\esR_{x_0}$ and 
${\esR_{x_0}/|y-x_0|}$ are uniform in $x_0\in\Gamma$,
where the latter function is defined to be zero when 
$y=x_0$.  

To establish properties of the functions $T_{x_0}$
and $T_{x_0}^{\Delta}$, for any given $y\in\Gamma_{x_0,d}$
with $y\ne x_0$, we consider an arbitrary curve 
$\gamma(\tau)\in\Gamma_{x_0,d}$ such that $\gamma(0)=y$.  
Then, by direct calculation from their definitions, we 
find $|{d\over d\tau} T_{x_0}(\gamma(\tau))\vert_{\tau=0}|
\le C |\gamma'(0)|/|y-x_0|$ and
$|{d\over d\tau} T_{x_0}^{\Delta}(\gamma(\tau))\vert_{\tau=0}|
\le C |\gamma'(0)|$.  From this, we deduce
\begin{equation}
\begin{gathered}
|DT_{x_0}(y)|\le {C\over|y-x_0|}, 
\quad
|DT_{x_0}^{\Delta}(y)|\le C, 
\quad
\forall y\in\Gamma_{x_0,d},
\quad 
y\ne x_0, \\
T_{x_0}^{\Delta}\in C^{0,1}(\Gamma_{x_0,d},\Rt),
\end{gathered}
\label{LemmaFive1c}
\end{equation}
where the last result above follows from the fact that 
$T_{x_0}^{\Delta}(y)$ has a uniformly bounded surface 
derivative for almost every $y\in\Gamma_{x_0,d}$.
Notice that the Lipschitz constant for $T_{x_0}^{\Delta}$,
equivalently the bound on its surface derivative, is 
uniform in $x_0\in\Gamma$.  Moreover, by definition
of $T_{x_0}$, and the Lipschitz property of 
$T_{x_0}^{\Delta}$, and the fact that 
$T_{x_0}^{\Delta}(x_0)=0$, we have the bounds 
\begin{equation}
|T_{x_0}(y)| = 1,
\quad
|T_{x_0}^{\Delta}(y)| \le C|y-x_0|, 
\quad
\forall y\in\Gamma_{x_0,d},
\quad 
y\ne x_0. \\
\label{LemmaFive1d}
\end{equation}

Proceeding with part (i), for any $y\in\Gamma_{x_0,d}$ with
$y\ne x_0$, we use (\ref{LemmaFive1b}), (\ref{LemmaFive1}) 
and (\ref{Ufun}), and (A2) with $\mu=1$, to write
\begin{equation}
\begin{split}
&U_{x_0,\varphi}(y) \\
&\quad
 = H(x_0,y)[\varphi(y)-\varphi(x_0)]
   - u_{x_0}^{\rm polar}(0,\hxi(y)) D\varphi(x_0)T_{x_0}(y), \\ 
&\quad
 = u(x_0,y)\left[D\varphi(x_0) {\tstyle{(y-x_0)\over|y-x_0|}}
   + {\tstyle{\esR_{x_0}(y)\over |y-x_0|}}\right]
   - u_{x_0}^{\rm polar}(0,\hxi(y)) D\varphi(x_0)T_{x_0}(y). \\ 
\end{split}
\end{equation}
In terms of the functions $\smash{T_{x_0}^{\Delta}}$ 
defined in (\ref{Tfun}) and $\smash{u_{x_0}^{\Delta}}$ 
defined in (A3)(iii), this becomes
\begin{equation}
\begin{split}
U_{x_0,\varphi}(y) 
&= u(x_0,y)D\varphi(x_0)T_{x_0}^{\Delta}(y) \\
&\hskip1.0in
   + u(x_0,y) {\esR_{x_0}(y)\over |y-x_0|}
   + u_{x_0}^{\Delta}(y) D\varphi(x_0)T_{x_0}(y). 
\end{split}
\label{LemmaFive2}
\end{equation}
Since $|u(x_0,y)|\le C$ by (A2), 
$|\smash{T_{x_0}(y)}|=1$ by definition, 
and $|D\varphi(x_0)|\le C_\varphi$ by the regularity 
assumption on $\varphi$, and
$|\smash{T_{x_0}^{\Delta}(y)}|\le C|y-x_0|$ by
(\ref{LemmaFive1d}), 
${|\esR_{x_0}(y)|/|y-x_0|}\le C_\varphi |y-x_0|$ by 
(\ref{LemmaFive1b}), and $|\smash{u_{x_0}^{\Delta}(y)}|
=|\smash{u_{x_0}^{\Delta}(y)} - \smash{u_{x_0}^{\Delta}(x_0)}|
\le C|y-x_0|$ by (A3)(iii), we obtain the bound
\begin{equation}
|U_{x_0,\varphi}(y)| \le C_\varphi |y-x_0|,
\quad
\forall y\in\Gamma_{x_0,d},
\quad
y\ne x_0.
\end{equation}
In the above, the constant $C_\varphi$ depends on $\varphi$, 
but is independent of $x_0$ and $y$.  Notice that the bound 
also holds when $y=x_0$ by definition of $U_{x_0,\varphi}$.  
To establish the Lipschitz property of this function,
for any given $y\in\Gamma_{x_0,d}$ with $y\ne x_0$, we 
again consider an arbitrary curve $\gamma(\tau)\in\Gamma_{x_0,d}$ 
such that $\gamma(0)=y$.  Then, by direct calculation using
(\ref{LemmaFive2}), (\ref{LemmaFive1d}), (\ref{LemmaFive1b}),
(A2) and (A3)(iii), we find 
$|{d\over d\tau} U_{x_0,\varphi}(\gamma(\tau))\vert_{\tau=0}|
\le C_\varphi |\gamma'(0)|$.  From this, we deduce
\begin{equation}
\begin{gathered}
|DU_{x_0,\varphi}(y)|\le C_\varphi, 
\quad
\forall y\in\Gamma_{x_0,d},
\quad 
y\ne x_0, \\
U_{x_0,\varphi}\in C^{0,1}(\Gamma_{x_0,d},\Rk),
\end{gathered}
\label{LemmaFive3}
\end{equation}
where, as before, the last result above follows from the 
fact that $U_{x_0,\varphi}(y)$ has a uniformly bounded 
surface derivative for almost every $y\in\Gamma_{x_0,d}$.
Notice that the Lipschitz constant for $U_{x_0,\varphi}$
depends on $\varphi$, but is uniform in $x_0\in\Gamma$.  
Thus the result in part (i) is established.

For parts (ii)(a,b), let $x_0\in\Gamma$ be arbitrary 
and consider a local Cartesian coordinate map 
$y=\psi_{x_0}(\xi)\in\Gamma_{x_0,d}$ with
$\xi\in\Omega_{x_0,d}\subset\Reals^2$, and
a local polar coordinate map 
$y=\psi_{x_0}(\varpi(\rho,\hxi))
=\smash{\psi_{x_0}^{\rm polar}}(\rho,\hxi)$ with
$(\rho,\hxi)\in\smash{\Omega_{x_0,d}^{\rm polar}}\subset
\Reals_+\times S$ and $\xi=\varpi(\rho,\hxi)=\rho\hxi$. 
Introducing the Jacobian factor
\begin{equation}
J(\xi) = \Big({\det\Big[
\Big({\partial\psi_{x_0}\over\partial\xi}\Big)^T 
{\partial\psi_{x_0}\over\partial\xi}\Big]}\Big)^{1/2}, 
\label{LemmaFive4}
\end{equation}
we have the relations
\begin{equation}
dA_y 
= J(\xi)\;dA_\xi 
= J(\rho\hxi)\;\rho\;d\rho\;dS_{\hxi}, 
\label{LemmaFive5}
\end{equation}
where $dA_y$ is an area element in $\Gamma_{x_0,d}$,
$dA_\xi$ is an area element in $\Omega_{x_0,d}$ and
$dS_{\hxi}$ is an arclength element on $S$.  By
conditions (L2) and (L3) on $\Gamma$, we note that 
the Lyapunov radius $d>0$ can be chosen small enough 
such that
\begin{equation}
1\le J(\xi) \le 2,
\quad
\forall \xi \in \Omega_{x_0,d},
\quad
x_0\in\Gamma.
\label{LemmaFive6}
\end{equation}
Moreover, if we let $L\ge 1$ denote the Lipschitz 
constant for $\psi_{x_0}$, which is uniform in $x_0$, 
then a straightforward argument using the definitions 
of $\Gamma_{x_0,\delta}$ and $\Omega_{x_0,\delta}$ 
shows that
\begin{equation}
D_{x_0,\delta/L}\subset \Omega_{x_0,\delta} \subset D_{x_0,\delta},
\quad
\forall \delta\in(0,d],
\quad
x_0\in\Gamma,
\label{LemmaFive6b}
\end{equation}
where $D_{x_0,\delta}$ is the closed disc of radius 
$\delta>0$ at the origin. 

Proceeding with (ii)(a,b), we consider the 
function $I:S\to\Reals$ defined by
\begin{equation}
I(\hxi) = \max\left\{
\int_0^{d/2L}J(\rho\hxi)\;\rho\;d\rho,
\int_0^{d/2L}J(-\rho\hxi)\;\rho\;d\rho\right\}.
\label{LemmaFive7}
\end{equation}
Notice that $I(\hxi)$ is well-defined since 
$D_{x_0,d/2L}\subset \Omega_{x_0,d}$ by 
(\ref{LemmaFive6b}).  From (\ref{LemmaFive7}),
we deduce $I(-\hxi)=I(\hxi)$, and from 
(\ref{LemmaFive7}) and (\ref{LemmaFive6})
we deduce, using a change of variable and
the fact that $J(\xi)/J(\xi')\le 2$ for all 
$\xi,\xi'\in\Omega_{x_0,d}$,
\begin{equation}
{d^2\over 8L^2}
\le
\int_0^{d/2L}J(\rho\hxi)\;\rho\;d\rho
\le I(\hxi) \le 
\int_0^{d/L}J(\rho\hxi)\;\rho\;d\rho
\le
{d^2\over L^2},
\label{LemmaFive8}
\end{equation}
where the second integral is also well-defined 
since $D_{x_0,d/L}\subset \Omega_{x_0,d}$ by
(\ref{LemmaFive6b}).  For any given choice 
of orthonormal basis in $T_{x_0}\Gamma$, we 
consider the components $\xi=(\xi_1,\xi_2)$ and 
$\hxi=(\hxi_1,\hxi_2)$, and the decomposition 
$S=S_+\cup S_-\subset\Reals^2$, where $S_+$ and
$S_-$ are disjoint subsets defined as
\begin{equation}
\begin{gathered}
S_+=\{\hxi\in S\;|\; 
       \hxi_2>0 \;\hbox{\rm or}\; \hxi=(1,0)\}, \\
S_-=\{\hxi\in S\;|\; 
       \hxi_2<0 \;\hbox{\rm or}\; \hxi=(-1,0)\}.
\end{gathered}
\label{LemmaFive9}
\end{equation}
Moreover, we consider functions $R_+:S_+\to\Reals$
and $R_-:S_-\to\Reals$ defined implicitly by
\begin{equation}
\begin{gathered}
\int_0^{R_+(\hxi)}J(\rho\hxi)\;\rho\;d\rho
               = I(\hxi), \quad\forall\hxi\in S_+, \\
\int_0^{R_-(\hxi)}J(\rho\hxi)\;\rho\;d\rho
               = I(\hxi), \quad\forall\hxi\in S_-, 
\end{gathered}
\label{LemmaFive10}
\end{equation}
and from (\ref{LemmaFive8}) and (\ref{LemmaFive7}) we 
note that $R_+,R_-\in[d/2L,d/L]$.  Furthermore, we use 
these functions to define subsets of $\Omega_{x_0,d}$ by
\begin{equation}
\begin{gathered}
\Omega_{x_0,+}^*=\{\xi=\rho\hxi\in\Omega_{x_0,d} \;|\; 
       0\le\rho\le R_+(\hxi), \; \hxi\in S_+\}, \\
\Omega_{x_0,-}^*=\{\xi=\rho\hxi\in\Omega_{x_0,d} \;|\; 
       0\le\rho\le R_-(\hxi), \; \hxi\in S_-\},
\end{gathered}
\label{LemmaFive12}
\end{equation}
and we note that, since $R_+,R_-\in[d/2L,d/L]$ and
$S_+ \cup S_- = S$, we have
\begin{equation}
D_{x_0,d/2L}
\subset 
\Big(\Omega_{x_0,+}^* \cup \Omega_{x_0,-}^*\Big)
\subset 
D_{x_0,d/L},
\quad
\forall x_0\in\Gamma.
\label{LemmaFive14}
\end{equation}
Finally, we define a subset of $\Gamma_{x_0,d}$ by
\begin{equation}
\Gamma_{x_0}^*
=\{ y=\varphi_{x_0}(\xi)\in\Gamma_{x_0,d} \;|\;
   \xi\in \Omega_{x_0,+}^* \cup \Omega_{x_0,-}^* \}.
\label{LemmaFive15}
\end{equation}

The subset $\Gamma_{x_0}^*$ has the properties stated 
in part (ii)(a,b).  Specifically, by (\ref{LemmaFive15}), 
(\ref{LemmaFive14}) and (\ref{LemmaFive6b}), it satisfies 
the uniformity condition 
\begin{equation}
\Gamma_{x_0,d/2L}
\subset
\Gamma_{x_0}^*
\subset
\Gamma_{x_0,d},
\quad
\forall x_0\in\Gamma.
\label{LemmaFive16}
\end{equation}
Moreover, for any $\varphi\in C^{1,1}(\Gamma,\Rk)$, we have
\begin{equation}
\begin{split}
&\int_{\Gamma_{x_0}^*} 
  u_{x_0}^{\rm polar}(0,\hxi(y)) 
            D\varphi(x_0)T_{x_0}(y) \; dA_y \\
&\hskip0.3in
= \int_{S_+} \int_{0}^{R_+(\hxi)} 
  u_{x_0}^{\rm polar}(0,\hxi) 
     D\varphi(x_0)T_{x_0}(\hxi) 
       J(\rho\hxi) \;\rho\;d\rho\;dS_{\hxi} \\
&\hskip0.8in
+ \int_{S_-} \int_{0}^{R_-(\hxi)} 
  u_{x_0}^{\rm polar}(0,\hxi) 
     D\varphi(x_0)T_{x_0}(\hxi) 
       J(\rho\hxi) \;\rho\;d\rho\;dS_{\hxi}, \\
\end{split}
\label{LemmaFive17}
\end{equation}
where we note that the unit tangent vector $T_{x_0}(y)$ 
is uniquely determined by the angular coordinate 
$\hxi(y)\in S$, and hence we write $T_{x_0}(\hxi)$ in 
the last two integrals.  Specifically, for any given
orthonormal basis $\{t_{x_0,1},t_{x_0,2}\}$ in 
$T_{x_0}\Gamma$, we have
\begin{equation}
T_{x_0}(\hxi) = \hxi_1 t_{x_0,1} + \hxi_2 t_{x_0,2}.
\label{LemmaFive17a}
\end{equation}
Combining (\ref{LemmaFive17}) with (\ref{LemmaFive10}), 
we get
\begin{equation}
\begin{split}
&\int_{\Gamma_{x_0}^*} 
  u_{x_0}^{\rm polar}(0,\hxi(y)) 
            D\varphi(x_0)T_{x_0}(y) \; dA_y \\
&\hskip0.8in
= \int_{S_+} u_{x_0}^{\rm polar}(0,\hxi) 
     D\varphi(x_0)T_{x_0}(\hxi) I(\hxi)\;dS_{\hxi} \\
&\hskip1.6in
+ \int_{S_-} u_{x_0}^{\rm polar}(0,\hxi) 
     D\varphi(x_0)T_{x_0}(\hxi) I(\hxi)\;dS_{\hxi}. \\
\end{split}
\label{LemmaFive17b}
\end{equation}
Since $u_{x_0}^{\rm polar}(0,\hxi)$ is an even function
on $S$ by (A3)(i,ii), $I(\hxi)$ is an even function on
$S$ by (\ref{LemmaFive7}), and $T_{x_0}(\hxi)$ is an 
odd function on $S$ by (\ref{LemmaFive17a}), and 
since the antipodal map $\hxi\mapsto-\hxi$ is a
length- and orientation-preserving map of $S$ onto $S$
which maps $S_+$ onto $S_-$, we find that the  two
integrals on the right-hand side of (\ref{LemmaFive17b})
cancel, which implies
\begin{equation}
\int_{\Gamma_{x_0}^*} 
  u_{x_0}^{\rm polar}(0,\hxi(y)) 
            D\varphi(x_0)T_{x_0}(y) \; dA_y = 0,
\quad
\forall x_0\in\Gamma,
\quad
\varphi\in C^{1,1}(\Gamma,\Rk).
\label{LemmaFive17c}
\end{equation}
Thus the results in parts (ii)(a,b) are established.

For part (ii)(c), we use notation similar to that used in 
the proof of Lemma \ref{LemmaTwo} and let $\beta\ge 10L+1$ 
and $E_{\beta}^*\ge E_0$ be given numbers, where $E_{\beta}^*$ 
is sufficiently large such that $0<\beta h \le d/4L$ for all 
$E\ge E_{\beta}^*$.  Moreover, for any given $x_0\in\Gamma$, 
we consider the collections of surface elements $\Gamma^e$, 
$e=1,\ldots,E$, defined by
\begin{equation}
I_{x_0}^*=\{e\;|\; \Gamma^e \subset\Gamma_{x_0}^*\}
\quad
\hbox{\rm and}
\quad
I_{x_0,\beta h}=\{e\;|\; \Gamma^e \subset\Gamma_{x_0,\beta h}\},
\end{equation}
and note that, by design, the sets 
$I_{x_0,\beta h}\subset I_{x_0}^*$ and
$I_{x_0}^*\backslash I_{x_0,\beta h}$
are non-empty for all $E\ge E_{\beta}^*$.
Moreover, for any given $\varphi\in C^{1,1}(\Gamma,\Rk)$, 
we consider the decomposition
\begin{equation}
{\tstyle\sum_{e\in I_{x_0}^*}}\;
{\tstyle\sum'_{1\le q\le Q}}\; 
F_{x_0}(\hxi(x_q^e))W_q^e 
= \esB_1(x_0) + \esB_2(x_0),
\label{LFiveDecomp} 
\end{equation}
where $F_{x_0}(\hxi) = u_{x_0}^{\rm polar}(0,\hxi) 
D\varphi(x_0)T_{x_0}(\hxi)$ and
\begin{align}
\esB_1(x_0) 
&= {\tstyle\sum_{e\in I_{x_0,\beta h}}}\;
   {\tstyle\sum'_{1\le q\le Q}}\;
   F_{x_0}(\hxi(x_q^e))W_q^e, 
\label{LFiveSum1} \\
\esB_2(x_0) 
&= {\tstyle\sum_{e\in I_{x_0}^*\backslash I_{x_0,\beta h}}}\;
   {\tstyle\sum_{q=1}^Q}\;
   F_{x_0}(\hxi(x_q^e))W_q^e. 
\label{LFiveSum2} 
\end{align}
To establish the result, we show that the above sums 
satisfy the bounds $|\esB_1|\le C_\varphi h$ and 
$|\esB_2|\le C_\varphi h$ for all $E\ge E_{\beta}^*$
and $x_0\in\Gamma$, where the constant $C_\varphi$ 
depends on $\varphi$, but is independent of $x_0$.  
Results for the finite interval $E\in[E_0,E_{\beta}^*]$ 
are straightforward consequences of the boundedness of
$F_{x_0}$ and will be omitted for brevity.

For the term $\esB_1$ in (\ref{LFiveDecomp}), we use the
fact that $|F_{x_0}|\le C_\varphi$ by (A2) and (A3),
$|I_{x_0,\beta h}|\le C$ by (\ref{LTwoSum1A}), and
$\sum_{q=1}^Q W_q^e\le C h^2$ by (A5), to deduce
\begin{equation}
\begin{split}
|\esB_1(x_0)| 
&\le {\tstyle\sum_{e\in I_{x_0,\beta h}}}\;
   {\tstyle\sum'_{1\le q\le Q}}\;
   |F_{x_0}(\hxi(x_q^e))|W_q^e  \\
&\le C_\varphi h^2 \le C_\varphi h,
\quad
\forall E\ge E_{\beta}^*,
\quad
x_0\in\Gamma,
\end{split}
\label{LFiveSum1A}
\end{equation}
which establishes the result for $\esB_1$. 

For the second term $\esB_2$ in (\ref{LFiveDecomp}), 
we notice that 
$\Gamma_{x_0,(\beta-1)h}\subset\Gamma_{x_0,\beta h}$
and that $\dist(\partial \Gamma_{x_0,(\beta-1)h},
\partial\Gamma_{x_0,\beta h})\ge h\ge\diam(\Gamma^e)$.
Hence, if $\Gamma^e\cap\Gamma_{x_0,(\beta-1)h}\ne\emptyset$,
then $\Gamma^e\subset \Gamma_{x_0,\beta h}$.
From this we deduce that, for all 
$e\in I_{x_0}^*\backslash I_{x_0,\beta h}$, we must have
$\Gamma^e\cap\Gamma_{x_0,(\beta-1)h}=\emptyset$, or 
equivalently $\Omega^e\cap\Omega_{x_0,(\beta-1)h}=\emptyset$ 
in terms of local Cartesian coordinate domains. 
Since $D_{x_0,(\beta-1)h/L}\subset \Omega_{x_0,(\beta-1)h}$
by (\ref{LemmaFive6b}), and $\beta\ge 10L+1$ by assumption, 
this implies
\begin{equation}
0 < 10h < |\xi|,
\quad
\forall \xi\in\Omega^e,
\quad
e\in I_{x_0}^*\backslash I_{x_0,\beta h},
\quad 
E\ge E_{\beta}^*.
\label{LFiveSum2A}
\end{equation}
Moreover, for any two points $\xi,\xi'\in\Omega^e$, with 
polar coordinates $\xi=\rho\hxi$ and $\xi'=\rho'\hxi'$,
we have 
\begin{equation}
\begin{split}
|\hxi'-\hxi| 
&= {|(\rho-\rho')\xi' + \rho'(\xi'-\xi)|\over\rho'\rho} \\
&\le {Ch\over\rho},
\quad
\forall \xi,\xi'\in\Omega^e,
\quad
e\in I_{x_0}^*\backslash I_{x_0,\beta h},
\quad 
E\ge E_{\beta}^*,
\end{split}
\label{LFiveSum2B}
\end{equation}
where the inequality follows from the fact that
$\diam(\Omega^e)\le C\diam(\Gamma^e)\le Ch$ by
the Lipschitz property of the coordinate maps 
$\psi_{x_0}$ and $\psi_{x_0}^{-1}$. 
Furthermore, since $u_{x_0}^{\rm polar}(\rho,\hxi)$ 
is class $C^{0,1}$ in $(\rho,\hxi)$ by (A3)(i), and 
$T_{x_0}(\hxi)$ is class $C^\infty$ in $\hxi$ by 
(\ref{LemmaFive17a}), we notice that 
$F_{x_0}(\hxi) = u_{x_0}^{\rm polar}(0,\hxi) 
D\varphi(x_0)T_{x_0}(\hxi)$ is class $C^{0,1}$ 
in $\hxi$ uniformly in $x_0$.  Hence, for each 
$e\in I_{x_0}^*\backslash I_{x_0,\beta h}$ 
we have, using the fact that $W_q^e>0$, 
and considering each scalar component of 
$F_{x_0}\in\Rk$ separately,
\begin{equation}
\min_{\xi\in\Omega^e} F_{x_0}(\hxi)
\le 
{\tstyle\sum_{q=1}^Q F_{x_0}(\hxi(x_q^e)) W_q^e \over
\tstyle\sum_{q=1}^Q W_q^e}
\le
\max_{\xi\in\Omega^e} F_{x_0}(\hxi).
\label{LFiveSum2C}
\end{equation}
From this we conclude, by the Intermediate Value Theorem, 
that for each scalar component of $F_{x_0}$ there exists 
an $\xi^e=\rho^e\hxi^e\in\Omega^e$ with the property that
\begin{equation}
\sum_{q=1}^Q F_{x_0}(\hxi(x_q^e)) W_q^e
= F_{x_0}(\hxi^e) \sum_{q=1}^Q W_q^e,
\quad
\forall e\in I_{x_0}^*\backslash I_{x_0,\beta h},
\quad 
E\ge E_{\beta}^*.
\label{LFiveSum2D}
\end{equation}
For convenience, we use the same symbol $\hxi^e$ to denote 
the distinguished point for each component of $F_{x_0}$.
In what follows, we will have need to consider the 
difference, denoted by $\Gamma_{x_0}^{*,\Delta}$, 
between the set $\Gamma_{x_0}^*$ and the subset 
$\cup_{e\in I_{x_0}^*\backslash I_{x_0,\beta h}}\Gamma^e$.
From the definition of $I_{x_0}^*$ and $I_{x_0,\beta h}$,
and the fact that $\diam(\Gamma^e)\le h$, we deduce that
\begin{equation}
\Gamma_{x_0}^{*,\Delta}
\;=\;
\Gamma_{x_0}^* \backslash
\Big(\bigcup_{e\in I_{x_0}^*\backslash 
                        I_{x_0,\beta h}}\Gamma^e\Big) 
\;\subset\; 
\Gamma_{x_0,\beta h} \cup N(\partial\Gamma_{x_0}^*,h),
\quad 
E\ge E_{\beta}^*,
\label{LFiveSum2E}
\end{equation}
where $N(\partial\Gamma_{x_0}^*,h)$ denotes the set 
of all $y\in\Gamma_{x_0}^*$ such that 
$\dist(y,\partial\Gamma_{x_0}^*)\le h$, that is, 
$N(\partial\Gamma_{x_0}^*,h)$ is the closed neighborhood
of $\partial\Gamma_{x_0}^*$ in $\Gamma_{x_0}^*$ of 
size $h$.  We note that $\Gamma_{x_0,\beta h}$ and
$N(\partial\Gamma_{x_0}^*,h)$ are disjoint for
$E_{\beta}^*$ sufficiently large.  Moreover, from
(\ref{LemmaFive10}), (\ref{LemmaFive12}) and
(\ref{LemmaFive15}), we deduce that the set 
$\partial\Gamma_{x_0}^*$ is a curve of class 
$C^{0,1}$.

Proceeding with the term $\esB_2$ in (\ref{LFiveDecomp}),
we have, using (\ref{LFiveSum2D}) and (\ref{LFiveSum2}),
and considering each scalar component separately,
\begin{equation}
\begin{split}
\esB_2(x_0) 
&=  \sum_{e\in I_{x_0}^*\backslash I_{x_0,\beta h}}\;
   {\!\!\!\!\!\!\!\!} F_{x_0}(\hxi^e) \sum_{q=1}^Q W_q^e \\
&=  \sum_{e\in I_{x_0}^*\backslash I_{x_0,\beta h}}\;
   {\!\!\!\!\!\!\!\!} F_{x_0}(\hxi^e) |\Gamma^e|  \quad
 + \sum_{e\in I_{x_0}^*\backslash I_{x_0,\beta h}}\;
   {\!\!\!\!\!\!\!\!} F_{x_0}(\hxi^e) 
                   \Big(\sum_{q=1}^Q W_q^e - |\Gamma^e|\Big).
\end{split}
\label{LFiveSum2F}
\end{equation}
Using the fact that 
$\int_{\Gamma_{x_0}^*} F_{x_0}(\hxi(y))\; dA_y = 0$ 
by (\ref{LemmaFive17c}), and the definition of
$\Gamma_{x_0}^{*,\Delta}$ in (\ref{LFiveSum2E}),
we get
\begin{equation}
\begin{split}
\esB_2(x_0) 
&=  \sum_{e\in I_{x_0}^*\backslash I_{x_0,\beta h}}\;
   {\!\!\!\!\!\!\!\!} F_{x_0}(\hxi^e) |\Gamma^e|  \quad
 + \sum_{e\in I_{x_0}^*\backslash I_{x_0,\beta h}}\;
   {\!\!\!\!\!\!\!\!} F_{x_0}(\hxi^e) 
             \Big(\sum_{q=1}^Q W_q^e - |\Gamma^e|\Big) \\
&\hskip2.7in
 -\; \int_{\Gamma_{x_0}^*} F_{x_0}(\hxi(y))\; dA_y \\
&=  \sum_{e\in I_{x_0}^*\backslash I_{x_0,\beta h}}\;
   {\!\!\!\!\!\!\!\!} \Big(\int_{\Gamma^e}
        F_{x_0}(\hxi^e) - F_{x_0}(\hxi(y))\; dA_y\Big) \\
&\hskip0.2in
 + \sum_{e\in I_{x_0}^*\backslash I_{x_0,\beta h}}\;
   {\!\!\!\!\!\!\!\!} F_{x_0}(\hxi^e) 
    \Big(\sum_{q=1}^Q W_q^e - |\Gamma^e|\Big) \quad 
    - \quad  \int_{\Gamma_{x_0}^{*,\Delta}} 
                               F_{x_0}(\hxi(y))\; dA_y \\
&= \quad \esB_{21}(x_0) \quad + \quad
         \esB_{22}(x_0) \quad - \quad \esB_{23}(x_0),
\end{split}
\label{LFiveSum2G}
\end{equation}
where $\esB_{21}(x_0)$, $\esB_{22}(x_0)$ and 
$\esB_{23}(x_0)$ denote the three terms in 
the middle equation of (\ref{LFiveSum2G}).

For the term $\esB_{21}$ in (\ref{LFiveSum2G}), we
notice that, by the Lipschitz property of 
$F_{x_0}(\hxi)$ and (\ref{LFiveSum2B}), we have
\begin{equation}
\begin{split}
|F_{x_0}(\hxi^e) - F_{x_0}(\hxi(y))|
&\le C_\varphi |\hxi^e - \hxi(y)| \\
&\le {C_\varphi h\over \rho(y)}
\le {C_\varphi h\over |y-x_0|},
\quad
\forall e\in I_{x_0}^*\backslash I_{x_0,\beta h},
\quad 
E\ge E_{\beta}^*,
\end{split}
\label{LFiveSum2H}
\end{equation}
where the last inequality follows from the
Lipschitz property of the coordinate map
$\psi_{x_0}$ and the fact that 
$|y-x_0|=|\psi_{x_0}(\xi(y))-\psi_{x_0}(0)|
\le C|\xi(y)|$ and $\rho(y)=|\xi(y)|$.
Hence, for the term $\esB_{21}$ we obtain
the uniform bound
\begin{equation}
\begin{split}
|\esB_{21}(x_0)|
&\le \sum_{e\in I_{x_0}^*\backslash I_{x_0,\beta h}}\;
   {\!\!\!\!} \int_{\Gamma^e}
        |F_{x_0}(\hxi^e) - F_{x_0}(\hxi(y))| \; dA_y \\
&\le C_\varphi h \int_{\Gamma_{x_0}^*} {1\over|y-x_0|} \; dA_y  
 \le C_\varphi h,   
\quad
\forall E\ge E_{\beta}^*,
\quad
x_0\in\Gamma,
\end{split}
\label{LFiveSum2I}
\end{equation}
where the last inequality follows from a direct
estimate of the integral using polar coordinates, 
which shows that the integral is bounded, and the 
constant $C_\varphi$ depends on $\varphi$, but 
is independent of $x_0$.  Thus the result for
$\esB_{21}$ is established.

For the term $\esB_{22}$ in (\ref{LFiveSum2G}), we
first use the quadrature error bound in (A6) for a 
rule of order $\ell\ge 1$ applied to a constant 
function to obtain
\begin{equation}
\Big| \sum_{q=1}^Q W_q^e - |\Gamma^e| \Big|
 =  \Big| \sum_{q=1}^Q W_q^e - \int_{\Gamma^e} \; dA_y \Big| 
\le C |\Gamma^e| h,
\;
\forall e\in I_{x_0}^*\backslash I_{x_0,\beta h},
\; 
E\ge E_{\beta}^*.
\label{LFiveSum2J}
\end{equation}
Using the above, together with the fact that
$|F_{x_0}(\hxi)|\le C_\varphi$ by (A2), (A3)
and the definition of $F_{x_0}$, we then obtain 
the uniform bound
\begin{equation}
\begin{split}
|\esB_{22}(x_0)|
&\le \sum_{e\in I_{x_0}^*\backslash I_{x_0,\beta h}}\;
   {\!\!\!\!} |F_{x_0}(\hxi^e)| \;
          \Big| \sum_{q=1}^Q W_q^e - |\Gamma^e| \Big| \\
&\le \sum_{e\in I_{x_0}^*\backslash I_{x_0,\beta h}}\;
   {\!\!\!\!} C_\varphi |\Gamma^e| h 
 \le C_\varphi |\Gamma_{x_0}^*| h   
 \le C_\varphi h,   
\quad
\forall E\ge E_{\beta}^*,
\quad
x_0\in\Gamma,
\end{split}
\label{LFiveSum2K}
\end{equation}
where the constant $C_\varphi$ depends on $\varphi$, 
but is independent of $x_0$.  Thus the result for 
$\esB_{22}$ is established.

For the term $\esB_{23}$ in (\ref{LFiveSum2G}), we
again use the fact that $|F_{x_0}(\hxi)|\le C_\varphi$ 
as above, and use the inclusion result for
$\Gamma_{x_0}^{*,\Delta}$ in (\ref{LFiveSum2E}), 
to get
\begin{equation}
\begin{split}
|\esB_{23}(x_0)|
&\le \int_{\Gamma_{x_0}^{*,\Delta}} 
                  |F_{x_0}(\hxi(y))|\; dA_y \\
&\le C_\varphi \int_{\Gamma_{x_0,\beta h}} 
                             {\!\!\!\!} dA_y 
    +  C_\varphi \int_{N(\partial\Gamma_{x_0}^*,h)} 
                             {\!\!\!\!} dA_y 
\le C_\varphi h,
\quad
\forall E\ge E_{\beta}^*,
\quad
x_0\in\Gamma,
\end{split}
\label{LFiveSum2L}
\end{equation}
where the last inequality follows from the 
definition of the patch $\Gamma_{x_0,\beta h}$
and the set $N(\partial\Gamma_{x_0}^*,h)$, and the 
constant $C_\varphi$ depends on $\varphi$, but 
is independent of $x_0$.  Thus the result for 
$\esB_{23}$ is established.  The result for 
$\esB_2$ follows from 
(\ref{LFiveSum2L}), (\ref{LFiveSum2K}) and 
(\ref{LFiveSum2I}), and the result stated in
part (ii)(c) follows from these, together with
(\ref{LFiveSum2G}), (\ref{LFiveSum1A}) and
(\ref{LFiveDecomp}).
\end{proof}

\subsection{Proof of main result}

Here we combine the collective compactness result
in Theorem \ref{thmConv} and the results in
Lemmas \ref{LemmaOne} -- \ref{LemmaFive}
to establish our main result in Theorem \ref{ConvThm}.
We consider the locally-corrected Nystr{\"o}m method
defined in (\ref{DefOpApprox})--(\ref{LocMomCondCoeffThree})
with a quadrature rule of arbitrary order $\ell\ge 1$,
a local polynomial correction of arbitrary degree 
$p\ge 0$, and a surface with regularity index $m\ge 0$.  

For part (i), consider any $\ell\ge 1$, $p\ge 0$ and 
$m\ge 0$.  Then, by Lemma \ref{LemmaFour}, the operators 
$\esA$ and $\esA_n$ ($n\ge n_0$) satisfy the collective 
compactness conditions (C1)--(C3).  Hence, by Theorem 
\ref{thmConv}, there exist constants $C_\varphi>0$ and 
$N_\varphi\ge n_0$ such that 
$c \varphi_n - \esA_n\varphi_n = f$ is uniquely 
solvable for $\varphi_n$, and moreover
\begin{equation}
||\varphi_n - \varphi || 
   \le C_\varphi || \esA_n \varphi - \esA\varphi ||, 
\quad\forall n\ge N_\varphi.
\label{Main1}
\end{equation}
By (C1) and the fact that $\varphi\in C^{0}(\Gamma,\Rk)$, 
we have $(\esA_n\varphi)(x)\to (\esA\varphi)(x)$ uniformly 
in $x\in\Gamma$, which in view of (\ref{Main1}) implies
\begin{equation}
||\varphi_n-\varphi|| \to 0 
\quad\hbox{\rm as}\quad
n\to\infty, 
\quad
\forall \ell\ge 1, p\ge 0, m\ge 0. 
\end{equation}
Thus the result in part (i) of the main theorem is 
established.

For convenience, we next consider part (iii), and will 
consider part (ii) afterwards since it requires a 
special treatment.  Accordingly, consider any 
$\ell\ge 1$, $p\ge 1$ and $m\ge 1$, and assume
$\varphi\in C^{m,1}(\Gamma,\Rk)$.  In view of 
(\ref{Main1}), we seek a bound for  
$|| \esA_n \varphi - \esA\varphi ||$.  By definition
of $\esA_n$ and $\esA$, we have
\begin{equation}
|| \esA_n \varphi - \esA\varphi ||
\le || \esG_n \varphi - \esG\varphi ||
  + || \esH_n \varphi - \esH\varphi ||,
\label{Main2}
\end{equation}
where $\esG$ and $\esH$ are the operators defined
in (\ref{DefGHops}), and $\esG_n$ and $\esH_n$ are 
the operators defined in (\ref{DefOpApprox}); we 
continue to use subscripts $n\to\infty$ in place 
of $h_n\to 0$.  

For the first term in (\ref{Main2}), 
we let $x_0\in\Gamma$ be arbitrary and consider the 
function $(G_{x_0}\varphi)(y) = G(x_0,y)\varphi(y)$.
Since $\varphi\in C^{m,1}(\Gamma,\Rk)$ by assumption,
and $G\in C^{m,1}(\Gamma\times\Gamma,\Rkk)$ by (A1), 
it follows that $G_{x_0}\varphi\in C^{m,1}(\Gamma,\Rk)$,
where the Lipschitz constant for each derivative is 
uniform in $x_0$.  In view of the quadrature error bound 
in (A6) for a rule of order $\ell$, we get, by definition 
of $\esG$ and $\esG_n$,
\begin{equation}
\begin{split}
|(\esG_n\varphi)(x_0) - (\esG\varphi)(x_0)| 
&= \Big|\sum_{e=1}^{E} \sum_{q=1}^Q
         G(x_0,x_q^e)\varphi(x_q^e)W_q^e
              - \int_{\Gamma} G(x_0,y)\varphi(y)\; dA_y \Big| \\
&\le \sum_{e=1}^{E}\Big| 
         \sum_{q=1}^Q(G_{x_0}\varphi)(x_q^e)W_q^e
            - \int_{\Gamma^e} (G_{x_0}\varphi)(y)\; dA_y \Big| \\
&\le C_\varphi h^{\min(\ell,m+1)},
\quad
\forall n\ge N_\varphi,
\quad
x_0\in\Gamma.
\end{split}
\label{Main4} 
\end{equation}
In the above, the constant $C_\varphi$ depends on 
$\varphi$, but is independent of $x_0$.  From this 
we deduce the bound
\begin{equation}
|| \esG_n \varphi - \esG\varphi || 
\le C_\varphi h^{\min(\ell,m+1)},
\quad
\forall n\ge N_\varphi.
\label{Main5}
\end{equation}

To bound the second term in (\ref{Main2}), we 
let $T_{x_0}^j\varphi:\Gamma_{x_0,d}\to\Rk$
be the local Taylor polynomial of degree $j$
for $\varphi$ at $x_0\in\Gamma$.  Just as
in (\ref{LocPolyDefn}), the local polynomial 
$(T_{x_0}^j\varphi)(z)$ is defined using local 
Cartesian coordinates $\xi_{x_0}(z)$ in 
$\Gamma_{x_0,d}$.  Specifically,
we have
\begin{equation}
\begin{split}
(T_{x_0}^j\varphi)(z)
&= C_{x_0,0} 
  \;+\;C_{x_0,\alpha_1} \xi_{x_0,\alpha_1}(z)
  \;+\; C_{x_0,\alpha_1\alpha_2} 
                    \xi_{x_0,\alpha_1}(z)\xi_{x_0,\alpha_2}(z) \\
&\hskip0.5in
  + \cdots +\;
  C_{x_0,\alpha_1\alpha_2\cdots\alpha_j} 
     \xi_{x_0,\alpha_1}(z)\xi_{x_0,\alpha_2}(z)\cdots
                                        \xi_{x_0,\alpha_j}(z), \\
\end{split}
\label{LocPolyDefnTay}
\end{equation}
where
\begin{equation}
C_{x_0,0} = \varphi(x_0),
\quad
C_{x_0,\alpha_1}
= {\partial(\varphi\circ\psi_{x_0}) \over 
                 \partial\xi^{\alpha_1}},
\quad
\ldots
\quad
C_{x_0,\alpha_1\cdots\alpha_j}
= {1\over j!}{\partial^j(\varphi\circ\psi_{x_0}) \over 
\partial\xi^{\alpha_1}\cdots\partial\xi^{\alpha_j}}.
\label{LocPolyDefnTayCoeff}
\end{equation}
In the above, all derivatives are evaluated at the 
surface point $x_0$, or equivalently $\xi_{x_0}(x_0)=0$ 
in local Cartesian coordinates.  Since 
$\varphi\in C^{m,1}(\Gamma,\Rk)$, it follows that
$T_{x_0}^j\varphi$ is well-defined for any $0\le j\le m$.
Moreover, in view of the Lipschitz properties of
the coordinates $\xi_{x_0}(z)$ for a surface $\Gamma$
of class $C^{m+1,1}$, it follows that
$T_{x_0}^j\varphi\in C^{m+1,1}(\Gamma_{x_0,d},\Rk)$,
where the Lipschitz constant for each derivative is 
uniform in $x_0$.  Furthermore, by definition of
the locally-corrected operator $\esH_n$, the relation 
in (\ref{LocMomCond}) holds, which implies
\begin{equation}
(\esH_n \, \eta_{x_0}T_{x_0}^j\varphi)(x_0)  
 = (\esH \, \eta_{x_0}T_{x_0}^j\varphi)(x_0), 
\quad
\forall 0\le j\le p,
\quad
x_0\in\Gamma.
\label{LocMomCondTay}
\end{equation}
Here $\eta_{x_0}\in C^{m,1}(\Gamma,[0,1])$ is any given 
cutoff function as described earlier.  Notice that
$\eta_{x_0}T_{x_0}^j\varphi$ is defined at all points 
of the surface and 
$\eta_{x_0}T_{x_0}^j\varphi\in C^{m,1}(\Gamma,\Rk)$,
where the Lipschitz constant for each derivative is
uniform in $x_0$.  

In what follows, we will have need to consider the 
function $\varphi_{x_0}^{\Delta}\in C^{m,1}(\Gamma,\Rk)$
defined by
\begin{equation}
\varphi_{x_0}^{\Delta}(y)
      = [\varphi - \eta_{x_0}T_{x_0}^j\varphi](y),
\quad
j=\min(p,m).
\label{LocDiffFun}
\end{equation}
From the definition of the local Taylor polynomial 
$T_{x_0}^j\varphi$, and the fact that $\eta_{x_0}$ 
is identically equal to one in a fixed neighborhood 
of $x_0$, we deduce that
\begin{equation}
\begin{gathered}
|\varphi_{x_0}^{\Delta}(y)|
\le C_\varphi |y-x_0|^{j+1}, 
\quad
|D^s \varphi_{x_0}^{\Delta}(y)|
\le C_\varphi |y-x_0|^{j+1-s}, \\
\quad 
0\le s\le j,
\quad
x_0,y\in\Gamma,
\end{gathered}
\label{LocDiffTay}
\end{equation}
where $D^s\varphi_{x_0}^{\Delta}$ denotes the surface 
derivative of order $s$.  For any given 
$x_0$, we will also have need to consider the function 
$(H_{x_0}\varphi_{x_0}^{\Delta})(y) 
= H(x_0,y)\varphi_{x_0}^{\Delta}(y)$. 
Using (\ref{LocDiffTay}), together with (A2) and (A3), 
we find
\begin{equation}
\begin{gathered}
|(H_{x_0}\varphi_{x_0}^{\Delta})(y)|
\le C_\varphi |y-x_0|^{j}, 
\quad
|D^s(H_{x_0}\varphi_{x_0}^{\Delta})(y)|
\le C_\varphi |y-x_0|^{j-s}, \\
\quad 
0\le s\le j,
\quad
x_0,y\in\Gamma.
\end{gathered}
\label{LocDiffTayH}
\end{equation}
Here and above $C_\varphi$ is a constant depending 
on $\varphi$, but is independent of $x_0$ and $y$.  
From (\ref{LocDiffTayH}) with $s=j$, we find that 
the surface derivative of $H_{x_0}\varphi_{x_0}^{\Delta}$
of order $j$ is uniformly bounded, namely
$|D^j(H_{x_0}\varphi_{x_0}^{\Delta})(y)|\le C_\varphi$ 
for all $x_0,y\in\Gamma$.  Thus it follows that 
$H_{x_0}\varphi_{x_0}^{\Delta}\in C^{j-1,1}(\Gamma,\Rk)$, 
where $j=\min(p,m)$ and the Lipschitz constant for 
each derivative is uniform in $x_0$.  This last 
regularity result provides a hint of the delicate 
nature of the case $p=0$ that will be considered later.

Proceeding with the second term in (\ref{Main2}), 
we consider an arbitrary $x_0\in\Gamma$ and use 
(\ref{LocMomCondTay}) and (\ref{LocDiffFun}) to write
\begin{equation}
(\esH_n\varphi - \esH\varphi)(x_0) 
   =(\esH_n\varphi_{x_0}^{\Delta} - \esH\varphi_{x_0}^{\Delta})(x_0).
\label{DiffHop}
\end{equation}
Using the definitions of $\esH$ and $\esH_n$ in (\ref{DiffHop}), 
and the fact that $\zeta_b + \smash{\widehat\zeta_b}=1$, we get
\begin{equation}
\begin{split}
&(\esH_n\varphi - \esH\varphi)(x_0) \\
&\quad 
  = \sum_{b=1}^{n} \zeta_b(x_0)H(x_0,x_b)
                                      \varphi_{x_0}^{\Delta}(x_b)W_b 
  + \sum_{b=1}^{n} \widehat\zeta_b(x_0)R_{x_0}(x_b)
                                         \varphi_{x_0}^{\Delta}(x_b) \\
&\hskip2.5in 
  - \int_\Gamma H(x_0,y)\varphi_{x_0}^{\Delta}(y)\; dA_y, \\
&\quad 
  = \sum_{b=1}^{n} H(x_0,x_b) \varphi_{x_0}^{\Delta}(x_b)W_b 
  - \int_\Gamma H(x_0,y)\varphi_{x_0}^{\Delta}(y)\; dA_y \\
&\hskip0.5in
  + \sum_{b=1}^{n} \widehat\zeta_b(x_0)R_{x_0}(x_b)
                                         \varphi_{x_0}^{\Delta}(x_b) 
  - \sum_{b=1}^{n} \widehat\zeta_b(x_0)H(x_0,x_b) 
                                     \varphi_{x_0}^{\Delta}(x_b)W_b.  
\end{split}
\label{Main6}
\end{equation}
For convenience, we decompose the above into three parts,
namely
\begin{equation}
(\esH_n\varphi - \esH\varphi)(x_0) 
   =\esM_1(x_0) + \esM_2(x_0) - \esM_3(x_0),
\label{Main6Sum}
\end{equation}
where
\begin{align}
\esM_1(x_0)  
 &= \sum_{b=1}^{n} H(x_0,x_b)\varphi_{x_0}^{\Delta}(x_b)W_b 
    - \int_\Gamma H(x_0,y)\varphi_{x_0}^{\Delta}(y)\; dA_y, 
\label{Main6Sum1} \\
\esM_2(x_0)  
 &= \sum_{b=1}^{n} \widehat\zeta_b(x_0)R_{x_0}(x_b)
                               \varphi_{x_0}^{\Delta}(x_b), 
\label{Main6Sum2} \\
\esM_3(x_0)  
 &= \sum_{b=1}^{n} \widehat\zeta_b(x_0)H(x_0,x_b) 
                            \varphi_{x_0}^{\Delta}(x_b)W_b.  
\label{Main6Sum3}
\end{align}
To establish the required result, we proceed to bound each 
of $\esM_1$, $\esM_2$ and $\esM_3$ uniformly in $x_0$.

For the term $\esM_1$ in (\ref{Main6Sum}), we note that 
$H_{x_0}\varphi_{x_0}^{\Delta}\in C^{j-1,1}(\Gamma,\Rk)$
by the discussion following (\ref{LocDiffTayH}).  In view 
of the quadrature error bound in (A6) for a rule of order 
$\ell$, we get
\begin{equation}
\begin{split}
|\esM_1(x_0)| 
&= \Big|\sum_{e=1}^{E} \sum_{q=1}^Q
      (H_{x_0}\varphi_{x_0}^{\Delta})(x_q^e)W_q^e
      - \int_{\Gamma} 
              (H_{x_0}\varphi_{x_0}^{\Delta})(y)\; dA_y \Big| \\
&\le \sum_{e=1}^{E}\Big| 
         \sum_{q=1}^Q(H_{x_0}\varphi_{x_0}^{\Delta})(x_q^e)W_q^e
            - \int_{\Gamma^e} 
              (H_{x_0}\varphi_{x_0}^{\Delta})(y)\; dA_y \Big| \\
&\le C_\varphi h^{\min(\ell,j)}
  =  C_\varphi h^{\min(\ell,p,m)},
\quad
\forall n\ge N_\varphi,
\quad
x_0\in\Gamma,
\end{split}
\label{Main6Sum1A} 
\end{equation}
where the last equality follows from the fact that
$j=\min(p,m)$.  In the above, the constant $C_\varphi$ 
depends on $\varphi$, but is independent of $x_0$.  
Thus a uniform bound for $\esM_1$ is established.

For the term $\esM_2$ in (\ref{Main6Sum}), we consider the 
index set $J_{x_0}=\{b\;|\;\smash{\widehat\zeta_b(x_0)}>0\}$,
and we note by the same arguments that led to 
(\ref{LFourSum3A}), that the number of elements in this 
set is uniformly bounded, namely $|J_{x_0}|\le C$ for all 
$n\ge N_\varphi$ and $x_0\in\Gamma$.  From (\ref{Main6Sum2}) 
we get, using the definition of $J_{x_0}$ and the fact 
that $0\le\smash{\widehat\zeta_b}\le 1$,
\begin{equation}
|\esM_2(x_0)| 
\le \sum_{b=1}^{n}|\widehat\zeta_b(x_0) R_{x_0}(x_b)
                               \varphi_{x_0}^{\Delta}(x_b)| 
\le \sum_{b\in J_{x_0}}|R_{x_0}(x_b)| 
                              |\varphi_{x_0}^{\Delta}(x_b)|.  
\label{Main6Sum2A}
\end{equation}
Since $|R_{x_0}(x_b)|\le C$ by Lemma \ref{LemmaThree}(iii), 
$|\varphi_{x_0}^{\Delta}(x_b)|\le C_\varphi |x_b-x_0|^{j+1}$
by (\ref{LocDiffTay}), and $|x_b-x_0|\le C h$ for all
$b\in J_{x_0}$ by definition of $J_{x_0}$ and (A7), we get
\begin{equation}
|\esM_2(x_0)| 
\le C_\varphi h^{j+1}
 =  C_\varphi h^{\min(p+1,m+1)},
\quad
\forall n\ge N_\varphi,
\quad
x_0\in\Gamma,
\label{Main6Sum2B}
\end{equation}
where the last equality follows from the fact that
$j=\min(p,m)$.  As before, the constant $C_\varphi$
depends on $\varphi$, but is independent of $x_0$.
Thus a uniform bound for $\esM_2$ is established.

For the term $\esM_3$ in (\ref{Main6Sum}), we again 
use the definition of $J_{x_0}$ and the fact 
that $0\le\smash{\widehat\zeta_b}\le 1$, together 
with (\ref{Main6Sum3}), to write
\begin{equation}
|\esM_3(x_0)|  
\le \sum_{b=1}^{n} |\widehat\zeta_b(x_0)H(x_0,x_b) 
                             \varphi_{x_0}^{\Delta}(x_b)W_b|  
\le \sum_{b\in J_{x_0}} 
                 |(H_{x_0}\varphi_{x_0}^{\Delta})(x_b)| W_b.  
\label{Main6Sum3A}
\end{equation}
Since $|(H_{x_0}\varphi_{x_0}^{\Delta})(x_b)|\le 
C_\varphi |x_b-x_0|^{j}$ by (\ref{LocDiffTayH}), 
and $|x_b-x_0|\le C h$ for all $b\in J_{x_0}$ 
by definition of $J_{x_0}$ and (A7), 
and $W_b\le C h^2$ by (A5), 
we get
\begin{equation}
|\esM_3(x_0)| 
\le C_\varphi h^{j+2}
 =  C_\varphi h^{\min(p+2,m+2)},
\quad
\forall n\ge N_\varphi,
\quad
x_0\in\Gamma,
\label{Main6Sum3B}
\end{equation}
where the last equality follows from the fact that
$j=\min(p,m)$.  Again, the constant $C_\varphi$
depends on $\varphi$, but is independent of $x_0$.
Thus a uniform bound for $\esM_3$ is established.
Combining (\ref{Main6Sum3B}), (\ref{Main6Sum2B}),
(\ref{Main6Sum1A}) and (\ref{Main6Sum}), we obtain
a bound for the second term in (\ref{Main2}), namely
\begin{equation}
|| \esH_n \varphi - \esH\varphi || 
\le C_\varphi h^{\min(\ell,p,m)},
\quad
\forall n\ge N_\varphi.
\label{Main18}
\end{equation}
Combining (\ref{Main18}), (\ref{Main5}), (\ref{Main2}) 
and (\ref{Main1}), we get
\begin{equation}
||\varphi_n-\varphi|| \le C_\varphi h^{\min(\ell,p,m)}
\quad\hbox{\rm as}\quad
n\to\infty, 
\quad
\forall \ell\ge 1, p\ge 1, m\ge 1. 
\end{equation}
Thus the result stated in part (iii) of the main theorem
is established.

For part (ii) of the main theorem, we assume $p=0$ and 
consider any $\ell\ge 1$ and $m\ge 1$.  In view of 
(\ref{Main1}), we again seek a bound for 
$|| \esA_n \varphi - \esA\varphi ||$ and consider 
the decomposition in (\ref{Main2}).   The term
$|| \esG_n \varphi - \esG\varphi ||$ can be bound
exactly as in (\ref{Main5}).  Hence we focus on 
the term $|| \esH_n \varphi - \esH\varphi ||$,
which requires a different treatment than 
before due to the limited regularity of 
$H_{x_0}\varphi_{x_0}^{\Delta}$ in the case $p=0$.
To this end, we consider an arbitrary $x_0\in\Gamma$
and let $\Gamma_{x_0}^*$ and $I_{x_0}^*$ be as in 
Lemma \ref{LemmaFive}(ii).  Then, using (\ref{DiffHop}) 
and the definitions of $\esH$ and $\esH_n$, we get
\begin{equation}
\begin{split}
&(\esH_n\varphi - \esH\varphi)(x_0) \\
&\quad 
  = \sum_{b=1}^{n} \zeta_b(x_0)H(x_0,x_b)
                                      \varphi_{x_0}^{\Delta}(x_b)W_b 
  + \sum_{b=1}^{n} \widehat\zeta_b(x_0)R_{x_0}(x_b)
                                         \varphi_{x_0}^{\Delta}(x_b) \\
&\hskip2.5in 
  - \int_\Gamma H(x_0,y)\varphi_{x_0}^{\Delta}(y)\; dA_y, \\
&\quad 
  = \sum_{e\in I_{x_0}^*} \left[\sum_{q=1}^{Q}
     \zeta_q^e(x_0)H(x_0,x_q^e) \varphi_{x_0}^{\Delta}(x_q^e)W_q^e 
  - \int_{\Gamma^e} 
      H(x_0,y)\varphi_{x_0}^{\Delta}(y)\; dA_y \right] \\
&\quad 
  + \sum_{e\not\in I_{x_0}^*} \left[\sum_{q=1}^{Q}
     \zeta_q^e(x_0)H(x_0,x_q^e) \varphi_{x_0}^{\Delta}(x_q^e)W_q^e 
  - \int_{\Gamma^e} 
      H(x_0,y)\varphi_{x_0}^{\Delta}(y)\; dA_y \right]\\
&\hskip2.5in
  + \sum_{b=1}^{n} \widehat\zeta_b(x_0)R_{x_0}(x_b)
                                       \varphi_{x_0}^{\Delta}(x_b). 
\end{split}
\label{Main23}
\end{equation}

For the first term in the first bracketed expression in 
(\ref{Main23}), we observe that
$\zeta_q^e(x_0)H(x_0,x_q^e) \varphi_{x_0}^{\Delta}(x_q^e)=0$
if $x_q^e=x_0$ by (\ref{LocDiffTayH}) and (A7).  Thus, using
the notation from Lemma \ref{LemmaFive} and the fact that
$\zeta_q^e + \smash{\widehat\zeta_q^e}=1$, we have
\begin{equation}
\begin{split}
&\sum_{q=1}^{Q}
   \zeta_q^e(x_0)H(x_0,x_q^e) 
                   \varphi_{x_0}^{\Delta}(x_q^e)W_q^e \\
&= \sum\nolimits_{1\le q\le Q}^{'}
           \zeta_q^e(x_0)H(x_0,x_q^e) 
                   \varphi_{x_0}^{\Delta}(x_q^e)W_q^e, \\
&= \sum\nolimits_{1\le q\le Q}^{'}
           H(x_0,x_q^e) 
                   \varphi_{x_0}^{\Delta}(x_q^e)W_q^e 
 - \sum\nolimits_{1\le q\le Q}^{'}
           \widehat\zeta_q^e(x_0)H(x_0,x_q^e) 
                   \varphi_{x_0}^{\Delta}(x_q^e)W_q^e. \\
\end{split}
\end{equation}
Using the definition of $U_{x_0,\varphi}$ given in 
(\ref{Ufun}), and the definition of $\varphi_{x_0}^{\Delta}$  
given in (\ref{LocDiffFun}) with $p=0$, we get
\begin{equation}
\begin{split}
&\sum_{q=1}^{Q}
   \zeta_q^e(x_0)H(x_0,x_q^e) 
                   \varphi_{x_0}^{\Delta}(x_q^e)W_q^e \\
&= \sum\nolimits_{1\le q\le Q}^{'}
           U_{x_0,\varphi}(x_q^e)W_q^e 
 - \sum\nolimits_{1\le q\le Q}^{'}
           \widehat\zeta_q^e(x_0)H(x_0,x_q^e) 
                   \varphi_{x_0}^{\Delta}(x_q^e)W_q^e \\
&\hskip1.4in
 + \sum\nolimits_{1\le q\le Q}^{'}
           u_{x_0}^{\rm polar}(0,\hxi(x_q^e)) 
                    D\varphi(x_0)T_{x_0}(x_q^e) W_q^e. 
\end{split}
\label{Main24}
\end{equation}

Similarly, for the second term in the first bracketed 
expression in (\ref{Main23}), we use Lemma 
\ref{LemmaFive}(ii)(b) and (\ref{LFiveSum2E}),
together with the definitions of $U_{x_0,\varphi}$ 
and $\varphi_{x_0}^{\Delta}$, to write
\begin{equation}
\begin{split}
&\sum_{e\in I_{x_0}^*}
   \int_{\Gamma^e} H(x_0,y)\varphi_{x_0}^{\Delta}(y)\; dA_y \\
&=\sum_{e\in I_{x_0}^*}
   \int_{\Gamma^e} H(x_0,y)\varphi_{x_0}^{\Delta}(y)\; dA_y 
 - \int_{\Gamma_{x_0}^*} 
            u_{x_0}^{\rm polar}(0,\hxi(y)) 
                             D\varphi(x_0)T_{x_0}(y)\; dA_y, \\
&=\sum_{e\in I_{x_0}^*}
   \int_{\Gamma^e} U_{x_0,\varphi}(y)\; dA_y 
 - \int_{\Gamma_{x_0}^{*,\Delta}} 
            u_{x_0}^{\rm polar}(0,\hxi(y)) 
                             D\varphi(x_0)T_{x_0}(y)\; dA_y. \\
\end{split}
\label{Main25}
\end{equation}

Substituting (\ref{Main25}) and (\ref{Main24}) into
(\ref{Main23}), and using the fact that 
$U_{x_0,\varphi}(x_q^e)=0$ if $x_q^e=x_0$ by Lemma 
\ref{LemmaFive}(i), we get the decomposition
\begin{equation}
(\esH_n\varphi - \esH\varphi)(x_0) 
   =\esN_1(x_0) + \esN_2(x_0) - \esN_3(x_0)
                         + \esN_4(x_0) + \esN_5(x_0),
\label{Main26Sum}
\end{equation}
where
\begin{equation}
\esN_1(x_0)  
=\sum_{e\in I_{x_0}^*} \left[\sum_{q=1}^{Q}
     U_{x_0,\varphi}(x_q^e)W_q^e 
  - \int_{\Gamma^e} 
      U_{x_0,\varphi}(y)\; dA_y \right],
\label{Main26Sum1} 
\end{equation}
\begin{equation}
\begin{split}
\esN_2(x_0)  
&=\sum_{e\in I_{x_0}^*}\sum\nolimits_{1\le q\le Q}^{'}
     u_{x_0}^{\rm polar}(0,\hxi(x_q^e)) 
                    D\varphi(x_0)T_{x_0}(x_q^e)W_q^e \\
&\hskip1.0in 
 + \int_{\Gamma_{x_0}^{*,\Delta}} 
          u_{x_0}^{\rm polar}(0,\hxi(y)) 
                     D\varphi(x_0)T_{x_0}(y)\; dA_y, 
\end{split}
\label{Main26Sum2} 
\end{equation}
\begin{equation}
\esN_3(x_0)  
 = \sum_{e\in I_{x_0}^*}\sum\nolimits_{1\le q\le Q}^{'}
             \widehat\zeta_q^e(x_0)H(x_0,x_q^e) 
                       \varphi_{x_0}^{\Delta}(x_q^e)W_q^e, 
\label{Main26Sum3} 
\end{equation}
\begin{equation}
\esN_4(x_0)  
=\sum_{e\not\in I_{x_0}^*} \left[\sum_{q=1}^{Q}
     \zeta_q^e(x_0)H(x_0,x_q^e) 
                        \varphi_{x_0}^{\Delta}(x_q^e)W_q^e 
  - \int_{\Gamma^e} 
          H(x_0,y)\varphi_{x_0}^{\Delta}(y)\; dA_y \right],
\label{Main26Sum4} 
\end{equation}
\begin{equation}
\esN_5(x_0)  
=\sum_{b=1}^{n} \widehat\zeta_b(x_0)R_{x_0}(x_b)
                              \varphi_{x_0}^{\Delta}(x_b). 
\label{Main26Sum5}
\end{equation}
To establish the required result, we proceed to bound each
of $\esN_1,\ldots,\esN_5$ uniformly in $x_0$.

For the term $\esN_1$, we note that
$\Gamma^e\subset \Gamma_{x_0}^*\subset \Gamma_{x_0,d}$
for all $e\in I_{x_0}^*$ by definition, and 
$U_{x_0,\varphi}\in C^{0,1}(\Gamma_{x_0,d},\Rk)$
with a Lipschitz constant that is uniform in $x_0$
by Lemma \ref{LemmaFive}(i).  In view of the quadrature 
error bound in (A6) for a rule of order $\ell$, we get
\begin{equation}
\begin{split}
|\esN_1(x_0)| 
&\le \sum_{e\in I_{x_0}^*}\Big| 
        \sum_{q=1}^Q U_{x_0,\varphi}(x_q^e)W_q^e
        - \int_{\Gamma^e} U_{x_0,\varphi}(y)\; dA_y \Big| \\
&\le C_\varphi h^{\min(\ell,1)}
  =  C_\varphi h,
\quad
\forall n\ge N_\varphi,
\quad
x_0\in\Gamma.
\end{split}
\label{Main26Sum1A} 
\end{equation}
For the term $\esN_2$, we use Lemma \ref{LemmaFive}(ii)(c) 
and (\ref{LFiveSum2L}) to conclude directly that
\begin{equation}
\begin{split}
|\esN_2(x_0)| \le C_\varphi h,
\quad
\forall n\ge N_\varphi,
\quad
x_0\in\Gamma.
\end{split}
\label{Main26Sum2A} 
\end{equation}
For the term $\esN_3$, we use the same arguments that 
led to (\ref{Main6Sum3B}) to get
\begin{equation}
|\esN_3(x_0)| 
\le C_\varphi h^{\min(p+2,m+2)}
= C_\varphi h^2,
\quad
\forall n\ge N_\varphi,
\quad
x_0\in\Gamma.
\label{Main26Sum3A}
\end{equation}
For the term $\esN_5$, we use the same arguments that 
led to (\ref{Main6Sum2B}) to get
\begin{equation}
|\esN_5(x_0)| 
\le C_\varphi h^{\min(p+1,m+1)} =  C_\varphi h,
\quad
\forall n\ge N_\varphi,
\quad
x_0\in\Gamma.
\label{Main26Sum5A}
\end{equation}
In the above, the constant $C_\varphi$ depends on 
$\varphi$, but is independent of $x_0$.  Thus 
uniform bounds for $\esN_1$, $\esN_2$, 
$\esN_3$ and $\esN_5$ are established.

It remains to bound the term $\esN_4$.  For every
$x_0\in\Gamma$, we note from Lemma \ref{LemmaFive}(ii)(a) 
that the set $\Gamma_{x_0}^*$ contains the patch 
$\Gamma_{x_0,d/C}$.  Since $\diam(\Gamma^e)\le h$
and $\diam(\supp(\widehat\zeta_q^e))\le Ch$,
we have, after increasing the size of $N_\varphi$ 
if necessary,  
\begin{equation}
\begin{gathered}
{d\over 2C} \le {d\over C}-h \le |y-x_0|,
\quad
\widehat\zeta_q^e(x_0)=0,
\quad
\zeta_q^e(x_0)=1, \\
\forall y\in\Gamma^e,
\quad
e\not\in I_{x_0}^*,
\quad
q=1,\ldots,Q,
\quad
n\ge N_\varphi,
\quad
x_0\in\Gamma.
\end{gathered}
\end{equation}
Using the above, and the fact that $H(x_0,y)$ is 
class $C^{m,1}$ on the set $|y-x_0|\ge d/(2C)$ by 
(A2), and $\varphi_{x_0}^{\Delta}$ is class $C^{m,1}$
by (\ref{LocDiffFun}), both with Lipschitz constants for 
each derivative uniform in $x_0$, together with the quadrature 
error bound in (A6) for a rule of order $\ell$, we get
\begin{equation}
\begin{split}
|\esN_4(x_0)|  
&\le\sum_{e\not\in I_{x_0}^*} \Big| \sum_{q=1}^{Q}
     (H_{x_0}\varphi_{x_0}^{\Delta})(x_q^e)W_q^e 
   - \int_{\Gamma^e} 
     (H_{x_0}\varphi_{x_0}^{\Delta})(y)\; dA_y \Big| \\
&\le C_\varphi h^{\min(\ell,m+1)},
\quad
\forall n\ge N_\varphi,
\quad
x_0\in\Gamma.
\end{split}
\label{Main26Sum4A} 
\end{equation}
As before, the constant $C_\varphi$ depends on
$\varphi$, but is independent of $x_0$.  
Thus a uniform bound for $\esN_4$ is established.
Combining (\ref{Main26Sum4A}), (\ref{Main26Sum5A}),
(\ref{Main26Sum3A}), (\ref{Main26Sum2A}), 
(\ref{Main26Sum1A}) and (\ref{Main26Sum}), we obtain
the bound
\begin{equation}
|| \esH_n \varphi - \esH\varphi || 
\le C_\varphi h,
\quad
\forall n\ge N_\varphi.
\label{Main26SumBnd}
\end{equation}
Combining (\ref{Main26SumBnd}), (\ref{Main5}), 
(\ref{Main2}) and (\ref{Main1}), we get
\begin{equation}
||\varphi_n-\varphi|| \le C_\varphi h
\quad\hbox{\rm as}\quad
n\to\infty, 
\quad
\forall \ell\ge 1, p=0, m\ge 1. 
\end{equation}
Thus the result stated in part (ii) of the main theorem
is established.


\bibliographystyle{amsplain}
\bibliography{nystrom}

\end{document}